\documentclass[11pt,reqno]{amsart}
\usepackage{geometry}                
\usepackage{amsmath,amsthm,amssymb,amsfonts,relsize}
\usepackage{mathtools}
\usepackage{makecell}
\usepackage{caption}
\usepackage[english]{babel}
\usepackage{breqn}
\usepackage{xcolor}
\usepackage{tikz}
\usepackage{yfonts}
\usetikzlibrary{calc,hobby}
\usepackage{bbm}
\usepackage[shortlabels]{enumitem}
\usepackage{graphicx}
\usepackage{theoremref}
\newtheorem{theorem}{Theorem}[section]
\numberwithin{equation}{section}
\newtheorem{remark}[theorem]{Remark}

\newtheorem{corollary}[theorem]{Corollary}
\newtheorem{lemma}[theorem]{Lemma}

\theoremstyle{definition}
\newtheorem{assumption}{Assumption}
\newtheorem{definition}{Definition}[section]
\usepackage{amssymb}
\usepackage{blindtext}
\usepackage{xcolor}
\usepackage{mathrsfs}
\usepackage{ulem}
\usepackage{stmaryrd}

\usepackage{hyperref}
\usepackage{comment}
\hypersetup{
    colorlinks=true,
    linkcolor=blue,
    filecolor=magenta,      
    urlcolor=cyan,
}
\newcommand\sometext

\newcommand{\T}[1]{\textup{#1}}

\newcommand{\jump}[1]{\left\llbracket{#1}\right\rrbracket}
\newcommand{\norm}[1]{\left\lVert#1\right\rVert}
\DeclareMathOperator{\sech}{sech}

\title[Stability of Solitary waves in two-layer water] {Existence and stability of interfacial capillary-gravity solitary waves with constant vorticity}
\author[D. Sinambela]{Daniel Sinambela}
\address{Department of Mathematics, University of Missouri, Columbia, MO 65211} 
\email{dsf25@mail.missouri.edu}

\begin{document}

\begin{abstract}
In this paper, we consider 
capillary-gravity waves propagating on the interface separating two fluids of finite depth and constant density. The flow in each layer is assumed to be incompressible and of constant vorticity.  We prove the existence of small-amplitude solitary wave solutions to this system in the strong surface tension regime via a spatial dynamics approach. We then use a variant of the classical Grillakis--Shatah--Strauss (GSS) method to study the orbital stability/instability of these waves.  We find an explicit function of the parameters (Froude number, Bond number, and the depth and density ratios) that characterizes the stability properties. In particular, conditionally orbitally stable and unstable waves are shown to be possible.
\end{abstract}
\maketitle
\setcounter{tocdepth}{1}

\tableofcontents

\section{Introduction}\label{Introduction}
\allowdisplaybreaks
The main object of study in the present paper is internal water waves. The formation of these waves is primarily due to variations in salinity and temperature in the water bulk.  Internal waves can carry enormous amount of energy, while maintaining coherence over long distances.  They are known to play a crucial role in ocean dynamics as a means of transporting and mixing nutrients inside the water.

 In recent years, internal waves have gained a great deal of attention among mathematicians. A plethora of works has been devoted to proving existence of various types of internal traveling waves, see for instance \cite{amick1986global, Amick1989, SunShen1993, Nilsson2017} or the survey in \cite{haziot2022traveling}. Largely for reasons of mathematical convenience, many of these results assume the flow in each layer is irrotational; that is, the vorticity is identically zero. The literature pertaining to rotational flows is unsurpisingly far more limited due to the significant increase in mathematical complexity created by the presence of vorticity. Nonetheless,  rotational effects play a significant part in many physical situations, such as waves propagating over a background current; see, for example, \cite{Calin2022}. 
 
The results of the present paper come in two parts.  First, we prove the existence of a family of small-amplitude internal solitary waves with constant vorticity.  This is done under the assumption of strong surface tension regime, in a sense to be explained shortly.  Second, as our primary contribution, we  investigate the stability properties of these waves as solutions to the dynamical problem.  In particular, we exhibit an explicit function of the physical parameters whose sign determines whether a sufficiently small-amplitude wave is (conditionally) orbitally or stable or orbitally unstable.  We also consider several specific parameter regimes, and find that both stable and unstable waves exist.  These are among a very small number of analytical results concerning the nonlinear stability or instability of internal waves.  In particular, Chen and Walsh \cite{Ming--WalshOrbital2022} proved the orbital stability of a class of internal waves in the irrotational setting.  We will adopt their basic strategy, but incorporating constant vorticity introduces numerous complications.  

Mathematically, the problem is formulated as follows. Let $(x,y)\in \mathbb{R}^2$ be a point in the standard Cartesian coordinates system, with $x$ being the direction of wave propagation and gravity acting in the negative $y$ direction. For time $t \geq 0$, we assume that the fluid is confined to a channel and organized into two superposed layers: 
\[
\Omega(t):=\Omega_+(t) \cup \Omega_-(t),
\]
 bounded above and below by infinitely-long rigid walls, $\{y=d_+\}$ and $\{y=-d_-\}$.  Here and throughout the paper, we will use subscripts of $\pm$ to indicate the restriction of a quantity to $\Omega_\pm$.  Both layers share a common boundary $\mathscr{S}=\mathscr{S}(t)$ that is free.  We call this the internal interface and assume it is has can be parameterized as the graph of an unknown function $y=\eta(t,x)$.  Our focus will be on solitary waves, meaning $\eta$ is spatially localized in that  $\eta\to 0$ as $|x|\to \infty$. We take the density to be constant in each layer with $0<\rho_+\leq \rho_-$. More precisely, the upper and lower regions of the fluid can be written as
\begin{equation}
    \Omega_+(t)=\{(x,y)\in \mathbb{R}^2: \eta(t,x) < y< d_+\}
\end{equation}
and 
\begin{equation}
    \Omega_-(t)=\{(x,y)\in \mathbb{R}^2: -d_-< y< \eta(t,x) \}.
\end{equation}
\begin{figure}
\centering
\begin{tikzpicture}[xscale=1.25,yscale=0.65]
\draw [fill=gray,thin,gray] (-4,-1.7) rectangle (4,-1.5);				
\draw [thin,-] (-4, -1.5) -- (4, -1.5); 

\draw [domain=-4:4,fill=gray!10] plot[smooth] (\x, {0.25+2.25/(exp(1.25*\x)+exp(-1.25*\x))})-|(4,2.8)-|(-4,2.8)--cycle;
\draw [domain=-4:4,fill=gray!45] plot[smooth] (\x, {0.25+2.25/(exp(1.25*\x)+exp(-1.25*\x))})-|(4,-1.5)-|(-4,-1.5)--cycle;	
\draw [fill=gray,thin,gray] (-4,3.0) rectangle (4,2.8);				
\draw [thin,-] (-4, 2.8) -- (4, 2.8); 

\node [below  ] at (1.25,2) {\footnotesize $y = \eta(x)$};
\node [below  ] at (-1.35,1.4) {\footnotesize $\mathscr{S}$};

\node [right] at (3.6,1.35) {\footnotesize $d_+$};
\draw [thick,stealth-stealth]  (3.6,0.35) -- (3.6, 2.7) ; 

\node [right] at (3.6,-0.75) {\footnotesize $d_-$};
\draw [thick,stealth-stealth]  (3.6,0.15) -- (3.6, -1.45) ; 

\node  at (0,2.2) {\footnotesize $\Omega_+$};
\node  at (0,0) {\footnotesize $\Omega_-$};

\end{tikzpicture}
    \caption{Configuration of the fluid domain. Two fluids of different densities are confined in an infinitely long channel.}
    \label{fluid domain figure}
\end{figure}
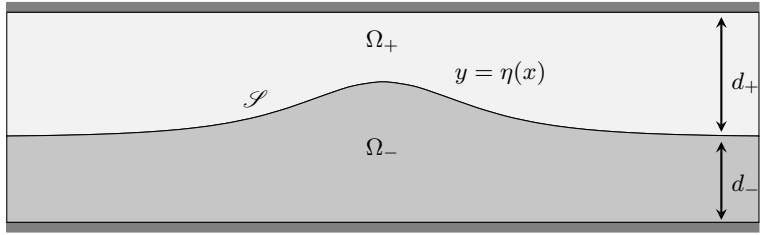

We will assume that the flow in each layer is incompressible and governed by the incompressible Euler equations.  Due to incompressibility, the velocity field $(u_\pm,v_\pm)$ in $\Omega_\pm$ can be expressed in terms of a stream function $\psi_\pm = \psi_\pm(t,x,y)$ by
\[ u_\pm=(\psi_{y})_{\pm},\qquad v_\pm=-(\psi_{x})_{\pm} .\]
Suppose that the vorticity in $\Omega_\pm$ is the constant $\omega_\pm := v_{\pm x}-u_{\pm y} \in \mathbb{R}$.  Note that in two-dimensions the vorticity is transported, so this assumption is valid even for the dynamical problem.  Taking the curl of the velocity field, we then find that the stream function satisfies the Poisson equations
\[\Delta \psi_\pm= -\omega_\pm \quad \T{in}\; \Omega_\pm(t).\]

As is common with constant vorticity waves, we wish to subtract out the background shear from $\psi$ to obtain a harmonic function.  That is, define $\Tilde{\psi}_\pm :=\psi_\pm + \frac{\omega_\pm y^2}{2}$, which will satisfy $\Delta \Tilde{\psi}_\pm = 0$ in $\Omega_\pm(t)$.   Let $\phi_\pm$ be a harmonic conjugate:
\[ \left(\phi_x\right)_\pm=\left(\psi_y\right)_\pm+\omega_\pm y, \T{ and } \left(\phi_y\right)_\pm=-\left(\psi_x\right)_\pm.\]
Then, the rotational incompressible Euler equations can be recast as follows. In the interior we have 
\begin{subequations}\label{Euler-Incompressible in terms of Harmonic Conj}
\begin{equation}\label{Euler-Incompressible in terms of Harmonic Conj-a}
    \Delta \phi_\pm=0\quad \T{ in}\; \Omega_\pm(t).
\end{equation}
On the internal interface and both rigid walls, the kinematic conditions read
\begin{equation}\label{Euler-Incompressible in terms of Harmonic Conj-b}
\left\{\begin{aligned}
    \eta_t&=\left(\phi_\pm\right)_y-\left((\phi_\pm)_x-\omega_\pm \eta \right) \eta_x & \T{ on}\; &y=\eta(t,x),\\
(\phi_\pm)_y&=0 &\T{ on}\; &y=\pm d_\pm. 
\end{aligned}\right.
\end{equation}
Physically, \eqref{Euler-Incompressible in terms of Harmonic Conj-b} states that the velocity field is tangential along the boundaries.  Finally, via the dynamic condition and Young--Laplace law, we can infer that the pressure jump across the internal interface is proportional to the signed curvature.  Using Bernoulli's principle in each layer, it can therefore be stated as
\begin{equation}\label{Euler-Incompressible in terms of Harmonic Conj-c}
\jump{\rho \phi_t}= -\jump{\frac{1}{2} \rho |\nabla \psi|^2 +g \rho \eta +\rho\omega \psi}-\sigma\left(\dfrac{\eta_x}{\sqrt{1+(n_x)^2}}\right)_x  \T{ on}\qquad y=\eta(t,x),
\end{equation}
\end{subequations}
The notation $\jump{\cdot}:=(\cdot)_+-(\cdot)_-$ denotes the difference in trace between two quantities on the internal interface in the upper and lower layer, $g>0$ is the gravitational constant, and $\sigma>0$ is the coefficient of surface tension. 

Observe that the functions $\phi_\pm$ are defined on a moving spatial domain, which complicates the task of finding an appropriate functional analytic setting for the problem.  We therefore prefer to work with the unknown 
\begin{equation}\label{phi to xi}
\xi_{\pm}(t,x)=\phi_{\pm}(t,x,\eta(t,x))
\end{equation}
which corresponds to trace on the surface.  Using this new variable allows us to push the entire problem to the free boundary, rendering it nonlocal but more tractable for analysis.

\subsection{Statement of results}

Now, we are ready to state our main results. We record them in Theorem~\ref{existence theorem} and Theorem~\ref{stability criterion theorem}.
We begin by introducing some important terminology and physical parameters that describe the system. 

A \textit{steady} or \textit{traveling} wave is a solution to the Euler equations \eqref{Euler-Incompressible in terms of Harmonic Conj} that translate in the $x$-direction at a fixed wave speed $c\in \mathbb{R}$ without altering its shape. Thus, in a moving frame of reference, it appears stationary. Concretely, this means the unknowns can be written as
\[\xi_\pm = \eta(t,x)=\eta_{c}(x-ct), \qquad \xi_\pm(x-ct)=\xi_{{c}_\pm}(x-ct),\]
for some steady profiles $\eta_c$ and $\xi_{{c}_\pm}$. Recall that we will focus on solitary waves, for which $\eta_c$ is localized.

There is an extensive body of work devoted to establishing the existence of traveling internal waves in various parameter regimes. Most well studied is the pure gravity case ($\sigma=0,g >0$), where both solitary waves  \cite{amick1986global, BonaBoseTurner1983, Mielke1995} and periodic waves \cite{amick1986global, Amick1989} have been constructed. It is, however, important to note that in the absence of surface tension, the water wave problem \eqref{Euler-Incompressible in terms of Harmonic Conj} becomes ill-posed dynamically; see for example \cite{Lannes2013}. Hence, when looking at questions pertaining to stability/instability, one has to assume $\sigma>0$, which we will do throughout this work.


The existence of small-amplitude internal waves in the presence of surface tension was obtained previously by many authors, for instance \cite{Kirrmann2022,SunShen1993, Nilsson2017}. However, none of those results allows for vorticity. Since understanding the effects of rotation on the stability is our objective, we spend the first part of our analysis developing an existence theory for small-amplitude internal waves with layer-wise constant vorticity.  This is accomplished using a spatial dynamics method: we view the $x$-coordinate as a time-like variable, and then use a center manifold reduction approach. The process is closely inspired by the work of Nilsson \cite{Nilsson2017}.

Internal solitary waves can be described using four dimensionless parameters.  The first two of them are the Bond number $\beta$ and the inverse square of the Froude number $\alpha$ defined by
\begin{equation}\label{Definition of beta and alpha}
    \beta:=\dfrac{\sigma}{d_+\rho_-c^2}, \qquad \alpha:=\dfrac{-g \jump{\rho}d_+}{\rho_-c^2}. 
\end{equation}
From its definition, we see that the Bond number $\beta$ measures the strength of the surface tension. In view of  \eqref{Definition of beta and alpha}, the Froude number $1/\sqrt{\alpha}$ can be thought of as the non-dimensionalized wave speed.

Upon linearizing \eqref{Euler-Incompressible in terms of Harmonic Conj} at the trivial solution and inserting the plane-wave ansatz $\eta= \T{exp}(ik(x-ct))$, we arrive at the following dispersion relation
\begin{equation}\label{dispersion relation}
        \alpha + \beta k^2 = \sum_{\pm}\dfrac{\rho_{\pm}}{\rho_-}k \coth \left(\dfrac{d_{\pm}}{d_+}k\right)+\left(\dfrac{ \omega_+d_+\rho_+}{c\rho_-}-\dfrac{\omega_-d_+}{c}\right).
\end{equation}It can be checked easily that $k=0$ is a root of \eqref{dispersion relation} exactly when
\begin{equation}\label{definition of Alphanot and Betanot}
    \beta=\beta_0:=\dfrac{1}{3}\left(\dfrac{\rho_+}{\rho_-}+\dfrac{d_-}{d_+}\right), \qquad \alpha=\alpha_0:=\dfrac{\rho_+}{\rho_-}+\dfrac{d_+}{d_-}+\dfrac{ \omega_+d_+\rho_+}{c\rho_-}-\dfrac{\omega_-d_+}{c}.
\end{equation}
We will regard $\beta_0$ as the critical Bond number: the range $\beta > \beta_0$ corresponds to the strong surface regime, and $\beta < \beta_0$ is the weak surface tension regime.   Heuristically, one expects that solitary waves will bifurcate from the trivial solutions at $\alpha_0$.  We will specifically be concerned with the strong surface tension case $\beta > \beta_0$.  

Apart from $\beta$ and $\alpha$, there are two  other physical parameters that have to be considered when studying interfacial waves: the density ratio $\varrho$ and the asymptotic height ratio $d$ given as follows
\begin{equation}\label{defintion of varrho and d}
\varrho:=\dfrac{\rho_+}{\rho_-}, \qquad d:=\dfrac{d_-}{d_+}.
\end{equation}
Unlike $\alpha$ and $\beta$, these are specific to the two-layer case.  Nilsson \cite{Nilsson2017} proved that for irrotational flow ($\omega_\pm=0$), when $\varrho-1/d^2<0$ and $O(1)$ as $\alpha \searrow \alpha_0$, there exist waves of depression $(\eta<0)$. On the other hand, if $\varrho-1/d^2>0$ and $O(1)$ as $\alpha \searrow \alpha_0$, then waves of elevation exist $(\eta>0)$. 

That said, our main result on existence of solitary waves is as follows.  
\begin{theorem}[Existence]\label{existence theorem}
Let $\Pi = \{(\rho_{\pm\epsilon},d_{\pm\epsilon},\omega_{\pm \epsilon},\sigma_\epsilon,c_\epsilon):0<\epsilon \ll 1\}$ be a smooth curve in the physical parameter space such that along $\Pi$, the corresponding Bond number is supercritical $\beta>\beta_0$, and the inverse-square Froude number is $\alpha=\alpha_0+\epsilon^2$.
Suppose that 
\begin{equation}
    \varrho-\dfrac{1}{d^2}+\dfrac{\omega_+d_+\varrho}{c}+\dfrac{\omega_-d_+}{cd}+\dfrac{\omega_+^2d_+^2\varrho}{3c^2}-\dfrac{\omega_-^2d_+^2}{3c^2}=O(1) \T{ as } \epsilon \searrow 0,
\end{equation}
where we have suppressed the $\epsilon$ dependence of the quantities on the left-hand side.  Then for any $k>1/2$ there exists a smooth curve of internal wave solutions
\begin{equation}
\mathcal{C}=\{(\eta_{\epsilon;\beta}, \xi_{+\epsilon;\beta}, \xi_{-\epsilon;\beta}):0<\epsilon \ll 1\}\subset H^{k+\frac{1}{2}}(\mathbb{R}) \times \left(\dot{H}^k(\mathbb{R}) \cap \dot{H}^{1/2}(\mathbb{R})\right)^2.
\end{equation} For every solution on the curve $\mathcal{C}$, the surface profile exhibits the following asymptotics:
\begin{equation}
\label{intro eta asymptotics}
    \eta_{\epsilon;\beta}(x)=\dfrac{d_+\epsilon^2 \sech^2{\left(\dfrac{\epsilon x}{2d_+\sqrt{\beta-\beta_0}}\right)}}{\varrho-\dfrac{1}{d^2}+\dfrac{\omega_+d_+\varrho}{c}+\dfrac{\omega_-d_+}{cd}+\dfrac{\omega_+^2d_+^2\varrho}{3c^2}-\dfrac{\omega_-^2d_+^2}{3c^2}}+O(\epsilon^3)
\end{equation}
in $H^{k+\frac{1}{2}}(\mathbb{R})$ as $\epsilon \searrow 0$.
\end{theorem}

We note that, the denominator in \eqref{intro eta asymptotics} determines whether the solution is a wave of elevation or depression for $\epsilon$ sufficiently small.  In contrast to the irrotational case, this will depend not only on the relative sizes of $\varrho$ and $d$, but also the strength of the vorticity in each layer.  It is also important to observe that, while $\epsilon$ is the appropriate parameter for proving existence, stability is best studied by fixing the physical parameters $\rho_\pm, d_\pm, \omega_\pm, \sigma$, and varying $c$.  Because $\epsilon = \sqrt{\alpha-\alpha_0}$, we can solve \eqref{Definition of beta and alpha} in terms of the wave speed and write $\alpha = \alpha_c$ and $\beta = \beta_c$.


The main result of the paper characterizes the conditional stability of these solutions in the orbital sense.  More precisely, we say a solitary wave $(\eta_c, \xi_{c+}, \xi_{c-})$ is \textit{conditionally orbitally stable} provided that, for all $R > 0$ and $r > 0$, there exists $r_0 > 0$ such that if $(\eta, \xi_+, \xi_-)$ is a solution to the internal wave problem on the time interval $[0,t_0)$ that obeys the a priori bound
\begin{equation}\label{well-posedness}
    \sup_{t\in[0,t_0)} \left(\norm{\eta(t)}_{H^{3+}}+\norm{\xi_+(t)}_{\dot{H}^{\frac{5}{2}+}\cap \dot{H}^{\frac{1}{2}+}}+\norm{\xi_-(t)}_{\dot{H}^{\frac{5}{2}+}\cap \dot{H}^{\frac{1}{2}+}}\right)<R,
\end{equation}
and whose initial data satisfies 
\begin{equation}\label{energy regularity}
    \norm{\eta(0)-\eta_c}_{H^1}+\norm{\xi_+(0)-\xi_{{c}_+}}_{\dot{H}^{\frac{1}{2}}}+\norm{\xi_-(0)-\xi_{{c}_-}}_{\dot{H}^{\frac{1}{2}}}<r_0,
\end{equation}
then
\begin{equation}\label{translation norm}
   \sup_{t\in [0,t_0)} \inf_{s \in \mathbb{R}}\left(\norm{\eta(t,\cdot-s)-\eta_c}_{H^1}+\norm{\xi_+(t,\cdot-s)-\xi_{{c}_+}}_{\dot{H}^{\frac{1}{2}}}+\norm{\xi_-(t,\cdot-s)-\xi_{{c}_-}}_{\dot{H}^{\frac{1}{2}}}\right)<r.
\end{equation}
The inequality \eqref{translation norm} measures the distance between the translated solutions $(\eta, \xi_+,\xi_-)$ and the family of traveling waves. The norms in \eqref{well-posedness} represents the lowest regularity required for local well-posedness of the Cauchy problem that is currently available. The meaning of the superscript (+) on the regularity will be made clear later in Section~\ref{hamiltonian formulation }. Furthermore, as we will see shortly, the regularity in \eqref{energy regularity} and \eqref{translation norm} matches the regularity of the energy space. 

Notice that this result is conditional in that we must assume a priori that the solution exists on a give time interval, since global well-posedness for the system is not known.   However, as $r_0$ is independent of the life span $t_0$, the bound in \eqref{translation norm} is substantially stronger result than merely continuity of the data-to-solution map.  In particular, if global existence is known, then we obtain orbital stability in the classical sense.

Conversely, we say a steady solution $(\eta_c, \xi_{c+}, \xi_{c-})$ is \textit{orbitally unstable} provided there exists $r > 0$ such that, for all $r_0 > 0$, there exists initial data 
\[
	(\eta(0), \xi_+(0), \xi_-(0)) \in H^1(\mathbb{R})) \times \left( \dot H^{\frac{1}{2}}(\mathbb{R}) \cap \dot H^{\frac{5}{2}}(\mathbb{R}))  \right)^2
\]
satisfying \eqref{energy regularity} and for which the corresponding solution $(\eta(t), \xi_+(t), \xi_-(t))$ to the Cauchy problem exits the tubular neighborhood of the solitary wave in finite time:
\[
	\inf_{s \in \mathbb{R}}\left(\norm{\eta(t,\cdot-s)-\eta_c}_{H^1}+\norm{\xi_+(t,\cdot-s)-\xi_{{c}_+}}_{\dot{H}^{\frac{1}{2}}}+\norm{\xi_-(t,\cdot-s)-\xi_{{c}_-}}_{\dot{H}^{\frac{1}{2}}}\right) > r
\]
for some time $t < \infty$.

We can now state our result on (conditional) orbital stability/instability. 
\begin{theorem}[Stability/instability]\label{stability criterion theorem}
Fix the physical parameters: $\rho_{\pm},d_{\pm},\omega_{\pm},\sigma$. There exists an explicit smooth function $m = m(c)$ such that, any sufficiently small-amplitude solitary internal wave $(\eta_c,\xi_{{c}_+},\xi_{{c}_-})$ with $\beta_c>\beta_0$ and $\alpha_c=\alpha_0+\epsilon^2$ is conditionally orbitally stable if $m'(c) > 0$ and orbitally unstable if $m'(c) < 0$. 
\end{theorem}

The formula for $m$ is given in \eqref{d'}.  Because it is rather complicated, it is instructive to look at a few special cases.  If the Bond number  is sufficiently close to critical  $0 < \beta_c - \beta_0 \ll 1$, then the waves given by Theorem~\ref{stability criterion theorem} are always orbitally stable.   We can also obtain more definitive statements by assuming the vorticity in one layer is $0$.  Parameter regime for which the waves are stable are given in Figure~\ref{table1}, and unstable in Figure~\ref{table2}.   Observe that both involve the relative size of the density ratio and a non-dimensionalized measure of the vorticity strength in the rotational layer.  The criterion also changes depending on whether we are considering a wave of elevation or depression, which is determined by the sign of the denominator in \eqref{intro eta asymptotics}.  Note that when $\omega_+ = \omega_-=0$, we recover the result in \cite{Ming--WalshOrbital2022} that all sufficiently small-amplitude waves are orbitally stable.

%

\begin{figure}
\begin{tabular}{|c|c|}
\hline
 $\eta>0$ & $ \eta < 0$  \\
\hline
 $ \makecell{\omega_+=0,\; 2(\varrho-1)<\dfrac{c\omega_-}{g} \leq 0}$ & $\makecell{\omega_-=0,\; 2\dfrac{(1-\varrho)}{\varrho} >\dfrac{c\omega_+}{g}\geq 0}$ \\[2ex]

\hline
\end{tabular}
\caption{Stability}
\label{table1}
\end{figure}
%

\begin{figure}
\begin{tabular}{|c|c|}
\hline
 $\eta > 0$ & $\eta < 0$  \\
\hline
$\omega_-=0,\; c\omega_+> 0,\; 2\dfrac{(\varrho-1)}{\varrho} < \dfrac{c\omega_+}{g}$ &   $ \omega_+=0,\; c\omega_-< 0,\; 2(\varrho-1)>\dfrac{c\omega_-}{g}$\\[2ex]
\hline
\end{tabular}
\caption{Instability}
\label{table2}
\end{figure}

\subsection{Idea of the proof}
In Section~\ref{Existence}, we prove Theorem~\ref{existence theorem} on the existence of the small-amplitude internal wave solutions. Following the strategy of Nilsson \cite{Nilsson2017}, we write the corresponding steady water wave problem as a spatial dynamical  Hamiltonian system. This is obtained by identifying a Lagrangian (the flow force), then applying a Legendre transform to arrive at the desired Hamiltonian. In the strong surface tension regime, we find that the linearized operator at the trivial solution has a $0$ eigenvalue of multiplicity $2$, corresponding to a Hamiltonian $0^2$ resonance.  Performing a center manifold reduction, we show that solitary waves of elevation or depression exist when
\begin{equation}
    \varrho-\dfrac{1}{d^2}+\dfrac{\omega_+d_+\varrho}{c}+\dfrac{\omega_-d_+}{cd}+\dfrac{\omega_+^2d_+^2\varrho}{3c^2}-\dfrac{\omega_-^2d_+^2}{3c^2},
\end{equation}
is positive or negative, respectively.  Note that when $\omega_\pm=0$, this analysis coincides with \cite[Section 3.3]{Nilsson2017}.

Next, we consider the stability or instability of these waves. Similar to approach of the existence theory, we again exploit the Hamiltonian structure of problem \eqref{Euler-Incompressible in terms of Harmonic Conj}. This time, however, it is the Hamiltonian for the time-dependent problem rather than spatial dynamical. In Section~\ref{hamiltonian formulation }, we show that \eqref{Euler-Incompressible in terms of Harmonic Conj} can be written as
\begin{equation}\label{abstract hamiltonian}
    \partial_t u= J\T{D}E(u),
\end{equation}
where $u=u(t,x)$ is an unknown represented by $(\eta,\xi_+,\xi_-)$, $J$ is a skew-adjoint operator called the Poisson map, and $E$ is an energy functional. It is well known that the translation-invariant nature of the problem gives rise to another conserved quantity known as the momentum $P$. By construction, it is clear that a traveling (steady) wave solution is a critical point of the augmented Hamiltonian given by $E_c:=E-cP$.

The main technique used to prove Theorem~\ref{stability criterion theorem} is based on the seminal works of Grillakis, Shatah, and Strauss  \cite{GSS11990,GSS21990}.  Their approach, known as the GSS method, provides a systematic way to prove nonlinear stability/instability for Hamiltonian systems that are  invariant under a continuous symmetry group. Although GSS has been successfully used to treat many model equations for water waves, the full free boundary Euler system exhibits a number of features that have made it resistant to this machinery.  For instance, GSS requires the Poisson map $J$ to be an isomorphism, which is not satisfied in the present setup as we will see shortly. Further, they require the Cauchy problem to be globally well-posed in the energy space.  Currently, only local well-posedness of \eqref{Euler-Incompressible in terms of Harmonic Conj} is known, and this assumes considerably higher regularity.  

In recent work, Varholm, Wahlén, and Walsh [VWW20] obtained a variant of the GSS method with hypotheses sufficiently relaxed so that it can be applied directly to the water wave problem. Their framework allows for the Poisson map $J$ to only have a dense range. It also permits the mismatch between the local well-posedness space and the energy space that the water wave problem possess. This theory was also the basis for the paper of Chen and Walsh \cite{Ming--WalshOrbital2022} on irrotational internal waves.  For the benefit of the reader, an abbreviated statement of the abstract result is given in Section~\ref{Stability}.

To apply this machinery to our problem, the main difficulty is to characterize the spectrum of the linearized augmented Hamiltonian at a traveling wave.  Introducing vorticity increases the complexity of the calculations substantially.  However, after a series of nontrivial computations, we find that the linearized operator at a sufficiently small-amplitude wave has Morse index $1$, as required by the abstract theory.  This analysis is carried out in Section~\ref{Spectral Analysis}. Then, by the general theory, stability or instability of the wave is determined by the convexity or concavity of the moment of instability, a scalar-valued function of the wave speed. In Section~\ref{proof of theorem}, we then prove the statement regarding conditional stability/instability in Theorem~\ref{stability criterion theorem}.  We emphasize that these are considerably more involved than the irrotational regime, and in particular both stable and unstable waves exist, which points toward the importance of including vortical effects in the model.   

%
%

\section{Existence theory} \label{Existence}
We start this section by proving the existence of small amplitude solitary water waves.  In doing that, we pattern the approach presented in \cite{Nilsson2017}. In comparison to \cite{Nilsson2017}, due to the rotational assumption, the equations that we have to deal with are several order more complex. Seeking traveling wave solutions, we impose a change of variables  $(t,x,y)\mapsto(x-ct,y)$. The main governing equation \eqref{Euler-Incompressible in terms of Harmonic Conj} can then be recast in terms of the relative streamfunctions $\psi_\pm$ and posed in a frame of reference moving with the waves as follows

\begin{equation}\label{govening equation in terms of the relative streamfunction}
 \begin{aligned}
    \Delta \psi_\pm&=-\omega_{\pm} & \T{in}\; &\Omega_\pm, \\
    \psi_\pm&=\mp m_\pm &\T{on}\; &y=\pm d_{\pm},
    \\
    \psi_\pm&=0 &\T{on}\; &y=\eta(x),\\
    \jump{\frac{1}{2} \rho |\nabla \psi|^2 +g \rho \eta} &=- \sigma\left(\dfrac{\eta_x}{\langle\eta_x\rangle}\right)_x + \mathcal{Q} & \T{on}\;& y=\eta(x),
\end{aligned} 
\end{equation}
for some constants $m_{\pm}$ together with the following asymptotic condition
\[\eta \rightarrow 0 \quad \text{as}\; x\rightarrow \infty.\] The variable $\mathcal{Q}$ in \eqref{govening equation in terms of the relative streamfunction} is the hydraulic head constant.

Next, we introduce the following non-dimensionalized variables 
\begin{equation}
    (x',y')=\frac{1}{d_+}(x,y),\qquad \eta'(x')=\frac{\eta(x)}{d_+},\qquad {\psi}'_\pm(x',y')=\frac{{\psi}_\pm(x,y)}{d_+c}.
\end{equation}
 Under these variables, the problem now reads
\begin{equation}
\left\{ \begin{aligned}
    \Delta \psi_+&=-\dfrac{\omega_{+}d_+}{c} &\T{for}\; &\eta(x)<y<1,\\
    \Delta \psi_-&=-\dfrac{\omega_{-}d_+}{c} &\T{for}\; &d<y<\eta(x), \\
    \psi_+&=\dfrac{- m_+}{cd_+} &\T{on}\; &y=1,
    \\
    \psi_-&=\dfrac{ m_-}{cd_+} &\T{on}\; &y=-d,
    \\
    \psi_\pm&=0 &\T{on}\; &y=\eta(x),
    \\
    \jump{\frac{1}{2} \rho c^2|\nabla \psi|^2 +g \rho d_+ \eta} + \dfrac{\sigma}{d_+}\left(\dfrac{\eta_x}{\langle\eta_x\rangle}\right)_x &= \mathcal{Q} &\T{on}\; &y=\eta(x),
\end{aligned} \right.
\end{equation}
where we have dropped the $'$ for notational convenience.

Next, to obtain the harmonic function $\tilde{\psi}$, we subtract the shear flow from $\psi$:
\[\tilde{\psi}_\pm:=\psi_\pm +\dfrac{\omega_\pm d_+ y^2}{2c}.
\]
Clearly, $\tilde{\psi}_\pm$ is harmonic in both layers, that is 
\begin{equation}\label{nabla tilde psi}
    \Delta \tilde{\psi}=0\quad \T{in}\; \Omega_\pm.
\end{equation}
Moreover, the boundary conditions on the rigid walls and internal interface, respectively, become  
\begin{equation}\label{boundary conditions on walls}
\begin{aligned}
\tilde{\psi}_+&=0  & \T{on}\; &y=1,
\\
\tilde{\psi}_-&=0 & \T{on}\;& y=-d,
\end{aligned}
\end{equation}
and 
\begin{equation}\label{boundary conditions on eta}
\begin{aligned}
    \tilde{\psi}_\pm&=\dfrac{\omega_\pm d_+\eta^2}{2c} & \T{on}\; &y=\eta(x),\\
    \jump{\frac{1}{2} \rho c^2 |\nabla \tilde{\psi}|^2-c\rho\omega d_+ \eta \partial_y\tilde{\psi} + \dfrac{1}{2}\rho \omega^2 d_+^2 \eta^2 +g \rho d_+ \eta} &=- \dfrac{\sigma}{d_+}\left(\dfrac{\eta_x}{\langle\eta_x\rangle}\right)_x + \mathcal{Q} & \T{on}\; &y=\eta(x).
\end{aligned} 
\end{equation}

Rather than work with $\tilde{\psi}$, we will use its harmonic conjugate $\phi$ to reformulate  equations \eqref{nabla tilde psi}, \eqref{boundary conditions on walls}, and \eqref{boundary conditions on eta}. One can view $\phi$ as the velocity potential. Using the fact that $\phi_{\pm_{x}}=\tilde{\psi}_{\pm_{y}}$ and $\phi_{\pm_{y}}=-\tilde{\psi}_{\pm_{x}}$, the equations now become
\begin{equation}
\left\{ \begin{aligned}
    \Delta \phi_\pm&=0&\T{in}\; &\Omega_\pm, \\
    \phi_{+_{y}}&=0 &\T{on}\; &y=1,
    \\
   \phi_{-_{y}}&=0 &\T{on}\; &y=-d,
    \\
    \phi_{\pm_{y}}&=\phi_{\pm_{x}}\eta_x-\dfrac{\omega_\pm d_+\eta\eta_x}{c}-\eta_x & \T{on}\; &y=\eta(x),\\
    \Bigl\llbracket\frac{1}{2} \rho c^2 |\nabla \phi|^2-c\rho\omega d_+ \eta \partial_x\phi \Bigr. \\ \Bigl. \qquad + \dfrac{1}{2}\rho \omega^2 d_+^2 \eta^2 +g \rho d_+ \eta\Bigr \rrbracket &=- \dfrac{\sigma}{d_+}\left(\dfrac{\eta_x}{\langle\eta_x\rangle}\right)_x + \mathcal{Q} & \T{on}\; &y=\eta(x).
\end{aligned} \right.
\end{equation}

Consider the following rescaling of the domain via the mapping $(x,y) \mapsto (x,z)$, where
\begin{equation}
    z(x,y):=\begin{cases}
    \dfrac{y-1}{\eta(x)-1} \qquad \text{for } \eta(x)<y<1,\\ \\
    \dfrac{y+d}{\eta(x)+d} \qquad \text{for } -d<y<\eta(x).
    \end{cases}
\end{equation}
As a result, we have the following change of variables formulas
\begin{equation}\label{change of variables derivative formulas}
\begin{aligned}
     \partial_y&=\dfrac{1}{\eta+d}\partial_z,\qquad
    \partial_y=\dfrac{1}{\eta-1}\partial_z,\\
    \partial_x&=\partial_X -\dfrac{z\eta_x}{\eta-1}\partial_z, \qquad \partial_x=\partial_X -\dfrac{z\eta_x}{\eta+d}\partial_z.
\end{aligned}
\end{equation}
Recycling notations, let us define $\phi_\pm(x,z):=\phi_\pm(x,y)$. 
Under the derivative formulas \eqref{change of variables derivative formulas}, equation \eqref{nabla tilde psi} now read
\begin{equation}
    \phi_{+_{xx}}-\dfrac{2z\eta_x}{\eta-1} \phi_{+_{xz}} -\dfrac{z\eta_{xx}}{\eta- 1}\phi_{+_{z}}+\dfrac{2z\eta_x^2}{(\eta-1)^2} \phi
_{+_{z}}\\+\dfrac{z^2\eta_x^2}{(\eta- 1)^2}\phi_{+_{zz}}+\dfrac{1}{(\eta- 1)^2}\phi_{+_{zz}}=0 \qquad 0<z<1,
\end{equation}
\begin{equation}
    \phi_{-_{xx}}-\dfrac{2z\eta_x}{\eta+ d} \phi_{-_{xz}} -\dfrac{z\eta_{xx}}{\eta + d}\phi_{-_{z}}+\dfrac{2z\eta_x^2}{(\eta + d)^2} \phi
_{-_{z}}+\dfrac{z^2\eta_x^2}{(\eta + d)^2}\phi_{-_{zz}}+\dfrac{1}{(\eta+ d)^2}\phi_{-_{zz}}=0 \qquad 0<z<1.
\end{equation}
The boundary conditions on the rigid walls \eqref{boundary conditions on walls} and the internal interface \eqref{boundary conditions on eta} translate to 
\begin{equation}
    \phi_{\pm_{z}}=0 \qquad \text{ on }z=0,
\end{equation}
and 
\begin{equation}
\begin{aligned}
    \phi_{+_{z}} &=(\eta-1)\left(\dfrac{-\omega_+d_+\eta\eta_x}{c}+\phi_{+_{x}}\eta_x -\dfrac{\eta_x^2\phi_z}{\eta-1}-\eta_x\right),\qquad \T{on } z=\eta,
    \\\phi_{-_{z}} &=(\eta+d)\left(\dfrac{-\omega_-d_+\eta\eta_x}{c}+\phi_{-_{x}}\eta_x -\dfrac{\eta_x^2\phi_z}{\eta+d}-\eta_x\right),\qquad \T{on } z=\eta,
    \\\varrho &\Bigg[\frac{1}{2} \left(\phi_{+_{x}}-\dfrac{ \eta_x\phi_{+_{z}}}{\eta- 1}\right)^2+\frac{1}{2}\left(\dfrac{\phi_{+_{z}}}{(\eta-1)}
         \right)^2-\dfrac{\omega_+ d_+ \eta}{c}\left(\phi_{+_{x}}-\dfrac{z\eta_x}{\eta-1} \phi_{+_{z}} \right) +\dfrac{1}{2}\dfrac{ \omega_+^2 d_+^2 \eta^2}{c^2}\Bigg]\\&-\Bigg[\frac{1}{2} \left(\phi_{-_{x}}-\dfrac{ \eta_x\phi_{-_{z}}}{\eta+ d}\right)^2+\frac{1}{2} \left(\dfrac{\phi_{-_{z}}}{(\eta+ d)}
         \right)^2-\dfrac{\omega_- d_+ \eta}{c}\left(\phi_{-_{x}}-\dfrac{z\eta_x}{\eta+h} \phi_{-_{z}} \right)+\dfrac{1}{2}\dfrac{\omega_-^2  d_+^2 \eta^2}{c^2}\Bigg] 
        \\&= \alpha \eta - \beta \left(\dfrac{\eta_x}{\langle\eta_x\rangle}\right)_x + \mathcal{Q},\qquad \T{on } z=\eta,
\end{aligned}
\end{equation}
where $\alpha, \beta$, and $\varrho$ are variables defined earlier in \eqref{Definition of beta and alpha} and \eqref{defintion of varrho and d}. The energy can be formulated as
\begin{equation*}
\begin{aligned}
        E&=K+V\\&= \dfrac{c^2 d_+^2\rho_+}{2}\int_{\mathbb{R}} \int_{0}^{1}\left(\left(\phi_{+_{x}}-\dfrac{ z\eta_x\phi_{+_{z}}}{\eta- 1}-\dfrac{\omega_+ d_+(z(\eta-1)+1)}{c}\right)^2+\left(\dfrac{\phi_{+_{z}}}{(\eta-1)}\right)^2\right)
        (1-\eta)\; dz \; dx
        \\&\quad \dfrac{c^2 d_-^2\rho_-}{2}\int_{\mathbb{R}} \int_{0}^{1}\left(\left(\phi_{-_{x}}-\dfrac{z \eta_x\phi_{-_{z}}}{\eta+d}-\dfrac{\omega_- d_+(z(\eta+d)-d)}{c}\right)^2+ \left(\dfrac{\phi_{-_{z}}}{(\eta+d)}\right)^2\right)(d+\eta)\; dz\;dx \\&\quad-\dfrac{1}{2}g d_+^3 \jump{\rho}\int_{\mathbb{R}} \eta^2\;dx +\sigma d_+ \int_{\mathbb{R}}\left( \sqrt{1+\eta_x^2}-1\right)\; dx.
\end{aligned}
\end{equation*}
Further, the momentum $P$ is given by
\begin{equation*}
\begin{aligned}
    P&=d_+^2 c \int_{\mathbb{R}}\int_{0}^1 \rho_+\left(\phi_{+_{x}}-\dfrac{ z\eta_x\phi_{+_{z}}}{\eta- 1}-\dfrac{\omega_+ d_+(z(\eta-1)+1)}{c}\right)
        (1-\eta) \;dz\; dx \\&\quad+d_+^2c\int_{\mathbb{R}} \int_{0}^1 \rho_-\left(\phi_{-_{x}}-\dfrac{z \eta_x\phi_{-_{z}}}{\eta+d}-\dfrac{\omega_- d_+(z(\eta+d)-d)}{c}\right)(d+\eta) \;dz\;dx.
\end{aligned}
\end{equation*}

From many literature, it is known that solitary waves can be detected by looking at the critical points of the functional $E-cP$. For this reason, we will study the Hamiltonian that arises from taking the Lagrangian of $E-cP$. Using the expressions for $E$ and $P$, we state the functional
\[
\begin{aligned}
        &E-cP= \\&d_+^2c^2\rho_-\Biggl[\dfrac{\varrho}{2} \int_{\mathbb{R}}\int_{0}^{1}\left(\left(\phi_{+_{x}}-\dfrac{ z\eta_x\phi_{+_{z}}}{\eta- 1}-\dfrac{\omega_+ d_+(z(\eta-1)+1)}{c}-1\right)^2\right.\\&\left. 
        \phantom{\int_{\mathbb{R}}\int_{0}^{1}\left(\phi_{+_{x}}-\dfrac{ z\eta_x\phi_{+_{z}}}{\eta- 1}-\dfrac{\omega_+)}{c}-1\right)^2}
        +\left(\dfrac{\phi_{+_{z}}}{(\eta-1)}\right)^2 \right)(1-\eta)\; dz\;dx\Biggr. \\& \Biggl.
        \phantom{d_+^2c^2\rho_-}
        \quad+ \dfrac{1}{2} \int_{\mathbb{R}}\int_{0}^{1}\left(\left(\phi_{-_{x}}-\dfrac{z \eta_x\phi_{-_{z}}}{\eta+d}-\dfrac{\omega_- d_+(z(\eta+d)-d)}{c}-1\right)^2\right. \\& \left.
        \phantom{\int_{\mathbb{R}}\int_{0}^{1}\left(\phi_{-_{x}}-\dfrac{z \eta_x\phi_{-_{z}}}{\eta+d}-\dfrac{\omega_-}{c}-1\right)^2}
         \;+\left(\dfrac{\phi_{-_{z}}}{(\eta+d)}\right)^2\right)(d+\eta)\; dz\;dx \Biggr. \\& \Biggl. 
         \phantom{d_+^2c^2\rho_-}
         -\int_{\mathbb{R}} \left(\dfrac{1}{2}\alpha  \eta^2 -\beta \left( \sqrt{1+\eta_x^2}-1\right)+\dfrac{\varrho}{2}(1-\eta)+\dfrac{1}{2}(\eta+d)\right)\;dx \Biggr].
\end{aligned}
\]
From that, we derive the corresponding Lagrangian
\begin{equation}\label{lagrangian}
\begin{split}
        L&=\dfrac{\varrho}{2} \int_{0}^{1}\left(\left(\phi_{+_{x}}-\dfrac{ z\eta_x\phi_{+_{z}}}{\eta- 1}-\dfrac{\omega_+ d_+(z(\eta-1)+1)}{c}-1\right)^2+\left(\dfrac{\phi_{+_{z}}}{(\eta-1)}\right)^2 \right)(1-\eta)\; dz\\&
        + \dfrac{1}{2} \int_{0}^{1}\left(\left(\phi_{-_{x}}-\dfrac{z \eta_x\phi_{-_{z}}}{\eta+d}-\dfrac{\omega_- d_+(z(\eta+d)-d)}{c}-1\right)^2 +\left(\dfrac{\phi_{-_{z}}}{(\eta+d)}\right)^2\right)(d+\eta)\; dz\\& 
        -\dfrac{1}{2}\alpha \eta^2 +\beta \left( \sqrt{1+\eta_x^2}-1\right)-\dfrac{\varrho}{2}(1-\eta)-\dfrac{1}{2}(\eta+d).
\end{split}
\end{equation}
In order to formulate the correct Hamiltonian, we need to know variational derivatives of $L$ with respect to the individual variable $\phi_{+_{x}}, \phi_{-_{x}}$, and $\eta_x$,
\begin{equation}\label{variational derivatives}
    \begin{split}
        \Phi_+:=\dfrac{\delta L}{\delta \phi_{+_{x}}}&=\varrho(1-\eta)\left(\phi_{+_{x}}-\dfrac{ z\eta_x\phi_{+_{z}}}{\eta- 1}-\dfrac{\omega_+ d_+(z(\eta-1)+1)}{c}-1\right),\\
        \Phi_-:=\dfrac{\delta L}{\delta \phi_{-_{x}}}&=(\eta+d)\left(\phi_{-_{x}}-\dfrac{z \eta_x\phi_{-_{z}}}{\eta+d}-\dfrac{\omega_- d_+(z(\eta+d)-d)}{c}-1\right),\\ 
        \gamma:=\dfrac{\delta L}{\delta \eta_x}&=\int_{0}^1 \varrho\left(\phi_{+_{x}}-\dfrac{ z\eta_x\phi_{+_{z}}}{\eta- 1}-\dfrac{\omega_+ d_+(z(\eta-1)+1)}{c}-1\right)z \phi_{+_{z}}\;dz \\& \quad-\int_{0}^1  \left(\phi_{-_{x}}-\dfrac{z \eta_x\phi_{-_{z}}}{\eta+d}-\dfrac{\omega_- d_+(z(\eta+d)-d)}{c}-1\right)z\phi_{-_{z}}\;dz+\beta \dfrac{\eta_x}{\langle\eta_x\rangle} \\&=-\int_{0}^1 \dfrac{\Phi_+z \phi_{+_{z}}}{\eta-1}\;dz-\int_{0}^1 \dfrac{\Phi_-z \phi_{-_{z}}}{\eta+d}\;dz+\beta \dfrac{\eta_x}{\langle\eta_x\rangle}.
    \end{split}
\end{equation}
Having the information above, the Hamiltonian can , therefore, be expressed in terms of $u=(\eta,\gamma,\phi_+,\Phi_+,\phi_-,\Phi_-)$ as follows

\begin{equation}\label{original hamiltonian}
\begin{split}
      H(u)&=\int_{0}^1 \Phi_+ \phi_{+_{x}}\;dz +\int_{0}^1 \Phi_- \phi_{-_{x}} \;dz +\gamma \eta_x -L\\&=\int_{0}^1\dfrac{1}{2\varrho(1-\eta)}\left(\left(\Phi_++\varrho(1-\eta)\right)^2-\varrho^2\phi_{+_{z}}^2\right)\;dz\\&
      \quad+\int_{0}^1\dfrac{1}{2(d+\eta)}\left(\left(\Phi_-+(d+\eta)\right)^2-\phi_{-_{z}}^2\right)\;dz\\& 
      \quad+\int_{0}^1\dfrac{\Phi_+\omega_+ d_+(z(\eta-1)+1)}{c}\;dz+\int_{0}^1\dfrac{\Phi_-\omega_- d_+(z(\eta+d)-d)}{c}\;dz \\&
      \quad-\sqrt{\beta^2-\bar{\gamma}^2}+\beta-\dfrac{\alpha\eta^2}{2},
\end{split}
\end{equation}where,
\[\bar{\gamma}=\gamma +\int_{0}^1 \dfrac{\Phi_+z \phi_{+_{z}}}{\eta-1}\;dz+\int_{0}^1 \dfrac{\Phi_-z \phi_{-_{z}}}{\eta+d}\;dz.\]

To formalize this, for $s\geq 0$, we define the following product spaces
\begin{equation*}
    \mathscr{X}_s=\mathbb{R} \times \mathbb{R} \times H^{s+1}(0,1) \times H^{s+1}(0,1)\times H^{s+1}(0,1)\times H^{s+1}(0,1).
\end{equation*} We would like to point out that the symbol $H^{s+1}$ refers to a Sobolev space of order $s+1$, not the Hamiltonian $H$.
Further, we let $\widehat{M}=\mathscr{X}_0$ be a manifold with $m\in \widehat{M}$ and let $v=(\eta,\gamma, \phi_+, \Phi_+,\phi_-,\Phi_-) \in T_m\widehat{M}$. On $T_m\widehat{M} \times T_m\widehat{M}$, consider the position independent symplectic form 
\begin{equation}\label{symplectic form}
    \widehat{\Omega}(v,v^*)=\gamma^*\eta-\eta^*\gamma+ \int_{0}^1 \left(\Phi_+^*\phi_+-\phi_+^*\Phi_+\right)\;dz+ \int_{0}^1 \left(\Phi_-^*\phi_--\phi_-^*\Phi_-\right)\;dz.
\end{equation} One may observe that $(\widehat{M},\widehat{\Omega})$ is a symplectic manifold. The corresponding set
\[
\widehat{N}=\{m \in \widehat{M}:|\bar{\gamma}|<\beta, -d<\eta<1\}
\] is a manifold domain of $\widehat{M}$ where the Hamiltonian $H$ is a smooth functional on it (i.e., $H\in C^{\infty}(\widehat{N},\mathbb{R})$). Hence, the tuple $(\widehat{M},H,\widehat{\Omega})$
forms a Hamiltonian system.
Via the symplectic form \eqref{symplectic form} and standard computations,  the associated Hamilton's equations read
\begin{equation}\label{Hamilton's equation}
    \begin{aligned}
          \dot{\eta}&=\dfrac{\bar{\gamma}}{\sqrt{\beta^2-\bar{\gamma}^2}},\\\dot{\gamma}&= \int_{0}^1\bigg[-\dfrac{\varrho}{2(1-\eta)^2}\left(\dfrac{\Phi_+^2}{\varrho^2}-\phi_{+_{z}}^2\right)+\dfrac{\varrho}{2}\bigg]\;dz\\ &\quad +\int_{0}^1\bigg[\dfrac{1}{2(d+\eta)^2}\left(\Phi_-^2-\phi_{-_{z}}^2\right)-\dfrac{1}{2}\bigg]\;dz-\int_{0}^1\dfrac{\Phi_+\omega_+ d_+z}{c}\;dz-\int_{0}^1\dfrac{\Phi_-\omega_- d_+z}{c}\; dz\\&\quad +\dfrac{\bar{\gamma}}{\sqrt{\beta^2-\bar{\gamma}^2}}\left(\int_{0}^1\dfrac{z\phi_{+_{z}}\Phi_+}{(1-\eta)^2}\;dz+\int_{0}^1\dfrac{z\phi_{-_{z}}\Phi_-}{(d+\eta)^2}\;dz\right)+\alpha\eta,\\ \dot{\phi}_+&=\dfrac{1}{\eta-1}\left(\dfrac{-\Phi_+}{\varrho}+(\eta-1)+\dfrac{\bar{\gamma}z\phi_{+_{z}}}{\sqrt{\beta^2-\bar{\gamma}^2}} \right)+\dfrac{\omega_+ d_+(z(\eta-1)+1)}{c},\\
          \dot{\Phi}_+&=\dfrac{1}{\eta-1}\left(\dfrac{\bar{\gamma}(z\Phi_+)_z}{\sqrt{\beta^2-\bar{\gamma}^2}}+\varrho\phi_{+_{zz}}\right),\\ \dot{\phi}_-&=\dfrac{1}{\eta+d}\left(\Phi_-+(\eta+d)+\dfrac{\bar{\gamma}z\phi_{-_{z}}}{\sqrt{\beta^2-\bar{\gamma}^2}} \right)+\dfrac{\omega_- d_+(z(\eta+d)-d)}{c},\\ \dot{\Phi}_-&=\dfrac{1}{\eta+d}\left(\dfrac{\bar{\gamma}(z\Phi_-)_z}{\sqrt{\beta^2-\bar{\gamma}^2}}-\phi_{-_{zz}}\right),
    \end{aligned}
\end{equation} where the Hamiltonian vector field also satisfies the corresponding boundary conditions
\begin{equation}
    \begin{aligned}
       &\varrho\phi_{+_{z}}(1)=-\dfrac{\bar{\gamma}\Phi_+(1)}{\sqrt{\beta^2 -\bar{\gamma}^2}},\\& \phi_{-_{z}}(1)=\dfrac{\bar{\gamma}\Phi_-(1)}{\sqrt{\beta^2 -\bar{\gamma}^2}},\\& \phi_{\pm_{z}}(0)=0.
    \end{aligned}
\end{equation}

In order to set a firmer ground for the latter analysis, we define the product space 
\begin{equation}
    Y_s=\mathbb{R} \times \mathbb{R} \times H^{s+1}(0,1) \times H^{s+1}_0(0,1) \times H^{s+1}(0,1) \times H^{s+1}_0(0,1),
\end{equation}
where $H^{s+1}_0(0,1)=\{f \in H^{s+1}(0,1) : f(0)=f(1)=0\}$. Additionally, let us also define these spaces
\begin{equation}
    \begin{split}
    &M=\{m\in \mathbb{R}^2 \times H^{1}(0,1)^4: \Gamma_+(0)=\Gamma_-(0)=\int_{0}^{1} \bar{\phi}_+\; dz=\int_{0}^1 \bar{\phi}_- \;dz=0\},\\&
    \tilde{M}=\{m \in Y_0: \int_{0}^{1} \bar{\phi}_+\; dz=\int_{0}^1 \bar{\phi}_- \;dz=0\},\\&
    \tilde{N}=\{m \in \tilde{M}: |\bar{\gamma}|<\beta, -d<\eta<1\},
\end{split}
\end{equation} where $m=(\eta, \gamma, \bar{\phi}_+, \Gamma_+, \bar{\phi}_-,\Gamma_-)$.
Going back to the Hamilton's equations \eqref{Hamilton's equation}, we can see that it has an equilibrium point 
\[
\begin{pmatrix}\eta\\ \gamma\\ \phi_+\\ \Phi_+\\ \phi_-\\ \Phi_-\end{pmatrix}=\begin{pmatrix}0\\0\\0\\ \dfrac{\omega_+d_+\varrho(z-1)}{c}-\varrho\\ 0\\ \dfrac{w_-d_+d^2(1-z)}{c}-d
\end{pmatrix}
.\] 

However, in preparation for the Hamiltonian reduction process on a center manifold, we need to shift the equilibrium point obtained before to the origin $(0,0,0,0,0,0)$. To achieve that, we impose another change of variables for the unknowns $\Phi_{\pm}, \bar{\phi}_{\pm}$, and $\chi_{\pm}$
\begin{equation}\label{new variable}
\begin{aligned}
&\Gamma_+:=\int_{0}^z(\Phi_+ + \varrho-\dfrac{\omega_+d_+\varrho(s-1)}{c})\;ds,
\\&\Gamma_-:=\int_{0}^z(\Phi_- +d-\dfrac{w_-d_+d^2(1-s)}{c})\;ds,\\
&\bar{\phi}_+:=\phi_+-\chi_+,
\\&\bar{\phi}_-:=\phi_--\chi_-,
\\&\chi_+:=\int_{0}^1\phi_+\;dz,
\\&\chi_-:=\int_{0}^1\phi_-\;dz.
\end{aligned}
\end{equation}

\pagebreak
This gives rise to a new formulation of the Hamiltonian equation.
 Precisely, one may think of this change of variables as a mapping that sends $(\eta, \gamma, \phi_+,\Phi_+,\phi_-,\Phi_-) \in \widehat{M}$ to $(\eta, \gamma, \bar{\phi}_+,\Gamma_+,\bar{\phi}_-,\Gamma_-,\chi_+, \chi_-)\in M \times \mathbb{R}^2$.
Thus, in the new variables the new symplectic form $\tilde{\Omega}$ becomes
\begin{equation}
\begin{split}
    \tilde{\Omega}(v,v^*)&=\omega^*\eta-\eta^*\omega + \int_{0}^1 \left(\Gamma_{+z}^*\bar{\phi}_+-\bar{\phi}_+^*\Gamma_{+z}\right)\;dz+ \int_{0}^1 \left(\Gamma_{-z}^*\bar{\phi}_--\bar{\phi}_-^*\Gamma_{-z}\right)\;dz\\& \quad +\Gamma_{+}^*(1)\chi_+ -\chi_+^*\Gamma_+(1)+\Gamma_{-}^*(1)\chi_- -\chi_-^*\Gamma_-(1).
\end{split}
\end{equation}

Further, using variables in \eqref{new variable}, the Hamiltonian now reads
\begin{equation}\label{second Hamiltonian}
\begin{aligned}
      H&=\int_{0}^1 \dfrac{1}{2\varrho(1-\eta)}\left(\left(\Gamma_{+_{z}} +\dfrac{\omega_+d_+\varrho(z-1)}{c}-\eta\varrho\right)^2-\varrho^2\bar{\phi}_{+_{z}}^2\right) dz\\& \quad +\int_{0}^1 \dfrac{1}{2(d+\eta)}\left(\left(\Gamma_{-_{z}} +\dfrac{w_-d_+d^2(1-z)}{c}+\eta\right)^2-\bar{\phi}_{-_{z}}^2\right) dz\\&\quad+\dfrac{1}{c}\int_{0}^1\left(\Gamma_{+_z} -\varrho + \dfrac{\omega_+d_+\varrho(z-1)}{c}\right)\omega_+ d_+(z(\eta-1)+1)\;dz\\&\quad+\dfrac{1}{c}\int_{0}^1 \left(\Gamma_{-_z} -d + \dfrac{\omega_-d_+d^2(1-z)}{c}\right)\omega_- d_+(z(\eta+d)-d)\;dz\\&\quad-\sqrt{\beta^2-\Bar{\gamma}^2} +\beta - \dfrac{\alpha}{2}\eta^2+\dfrac{\omega_+^2d_+^2\varrho}{6c^2}+\dfrac{\omega_-^2d_+^2d^3}{6c^2}+\dfrac{\omega_+d_+\varrho}{2c}-\dfrac{\omega_-d_+d^2}{2c},
\end{aligned}
\end{equation}
where
\[\bar{\gamma}=\gamma +\int_{0}^1 \dfrac{z\bar{\phi}_{+_{z}}}{\eta-1}\left(\Gamma_{+_{z}}+\dfrac{\omega_+d_+\varrho(z-1)}{c}-\varrho\right)\;dz + \int_{0}^1 \dfrac{z\bar{\phi}_{-_{z}}}{\eta+d}\left(\Gamma_{-_{z}}+\dfrac{w_-d_+d^2(1-z)}{c}-d\right)\; dz.\]
Notice that we have added the constants in the definition of the Hamiltonian so that $H(0)=0$.

The new Hamiltonian stucture  \eqref{second Hamiltonian} gives rise to the new Hamilton's equations which are given by
\begin{equation}\label{Final Hamilton's equations}
\begin{aligned}
    \dot{\eta}&= \bar{\gamma} (\beta^2-\bar{\gamma}^2)^{-1/2},\\
          \dot{\gamma}&= \int_{0}^1\bigg[-\dfrac{\varrho}{2(1-\eta)^2}\left(\dfrac{(\Gamma_{+_{z}}+\dfrac{\omega_+d_+\varrho(z-1)}{c}-\varrho)^2}{\varrho^2}-\bar{\phi}_{+_{z}}^2\right)+\dfrac{\varrho}{2}\bigg]\;dz\\ &\quad +\int_{0}^1\bigg[\dfrac{1}{2(d+\eta)^2}\left((\Gamma_{-_{z}}+\dfrac{\omega_-d_+d^2(1-z)}{c}-d)^2-\bar{\phi}_{-_{z}}^2\right)-\dfrac{1}{2}\bigg]\;dz\\ &\quad-\int_{0}^1\dfrac{\left(\Gamma_{+_{z}}+\dfrac{\omega_+d_+\varrho(z-1)}{c}-\varrho\right)\omega_+ d_+z}{c}\;dz\\&\quad-\int_{0}^1\dfrac{\left(\Gamma_{-_{z}}+\dfrac{\omega_-d_+d^2(1-z)}{c}-d\right)\omega_- d_+z}{c}\; dz\\&\quad +\dfrac{\bar{\gamma}}{\sqrt{\beta^2-\bar{\gamma}^2}}\Biggl(\int_{0}^1\dfrac{z\bar{\phi}_{+_{z}}(\Gamma_{+_{z}}+\dfrac{\omega_+d_+\varrho(z-1)}{c}-\varrho)}{(1-\eta)^2}\;dz \Biggr. \\& \phantom{+\dfrac{\bar{\gamma}}{\sqrt{\beta^2-\bar{\gamma}^2}}}\qquad \qquad \Biggl.+\int_{0}^1\dfrac{z\bar{\phi}_{-_{z}}(\Gamma_{-_{z}}+\dfrac{\omega_-d_+d^2(1-z)}{c}-d)}{(d+\eta)^2}\;dz\Biggr)+\alpha\eta,
\\ 
\dot{\bar{\phi}}_+&=\dfrac{1}{\eta-1}\left(\dfrac{-\Gamma_{+_{z}}}{\varrho}+\dfrac{\Gamma_+(1)}{\varrho}-\dfrac{\omega_+d_+(2z-1)}{2c}+\dfrac{\bar{\gamma}\left(z\bar{\phi}_{+_{z}}-\bar{\phi}_+(1)\right)}{\sqrt{\beta^2-\bar{\gamma}^2}} \right)
\\&\quad +\dfrac{\omega_+d_+(\eta-1)(2z-1)}{2c},
\\
\dot{\Gamma}_+&=\dfrac{1}{\eta-1}\left(\dfrac{\bar{\gamma}z}{\sqrt{\beta^2-\bar{\gamma}^2}}(\Gamma_{+_{z}}-\varrho+\dfrac{\omega_+d_+\varrho(z-1)}{c})+\varrho\bar{\phi}_{+_{z}}\right),
\\ 
\dot{\bar{\phi}}_-&=\dfrac{1}{\eta+d}\left(\Gamma_{-_{z}}-\Gamma_-(1)+\dfrac{\omega_-d_+d^2(1-2z)}{2c}+\dfrac{\bar{\gamma}\left(z\bar{\phi}_{-_{z}}-\bar{\phi}_-(1)\right)}{\sqrt{\beta^2-\bar{\gamma}^2}} \right)
\\&\quad+\dfrac{\omega_-d_+(\eta+d)(2z-1)}{2c},\\ \dot{\Gamma}_-&=\dfrac{1}{\eta+d}\left(\dfrac{\bar{\gamma}z}{\sqrt{\beta^2-\bar{\gamma}^2}}(\Gamma_{-_{z}}-d+\dfrac{\omega_-d_+d^2(1-z)}{c})-\bar{\phi}_{-_{z}}\right),\\
\end{aligned}
\end{equation}
along with 
\begin{equation}\label{equation for chi}
\begin{split}
     \dot{\chi}_+&=\dfrac{1}{\eta-1}\left(\dfrac{-\Gamma_+(1)}{\varrho}+\eta+\dfrac{\bar{\gamma}\bar{\phi}_{+}(1)}{\sqrt{\beta^2-\bar{\gamma}^2}}\right),\\
           \dot{\chi}_-&=\dfrac{1}{\eta+d}\left(\Gamma_-(1)+\eta+\dfrac{\bar{\gamma}\bar{\phi}_{-}(1)}{\sqrt{\beta^2-\bar{\gamma}^2}}\right),
\end{split}
\end{equation}
    and the boundary conditions
\begin{equation} \label{boundary conditions}
    \begin{aligned}
        &\bar{\phi}_{+_{z}}(1)=-\dfrac{\bar{\gamma}\left(\Gamma_{+_{z}}(1)-\varrho\right)}{\varrho\sqrt{\beta^2 -\bar{\gamma}^2}},\\& \bar{\phi}_{-_{z}}(1)=\dfrac{\bar{\gamma}\left(\Gamma_{-_{z}}(1)-d\right)}{\sqrt{\beta^2 -\bar{\gamma}^2}},\\& \bar{\phi}_{\pm_{z}}(0)=0.
    \end{aligned}
\end{equation}
Note that, the two equations in \eqref{equation for chi} can be neglected since they can be recovered from the rest of the equations in \eqref{Final Hamilton's equations}.

We now proceed with linearizing \eqref{Final Hamilton's equations} around the equilibrium point $(0,0,0,0,0,0)$. This leads us to the linearized problem stated in terms of the operator $\mathcal{L}$ 
\begin{equation}\label{lineraized operator}
\begin{aligned}
    \mathcal{L}\begin{bmatrix} \eta \\
           \gamma \\
           \bar{\phi}_+ \\
           \Gamma_+\\
           \bar{\phi}_-\\
           \Gamma_-
           \end{bmatrix}= \begin{bmatrix}
           \dfrac{1}{\beta}\left(\gamma+\varrho\bar{\phi}_+(1)+\int_0^1\dfrac{2z\bar{\phi}_+\omega_+d_+\varrho}{c}\;dz+\int_0^1\dfrac{2z\bar{\phi}_-\omega_-d_+d}{c}\;dz -\bar{\phi}_-(1)\right)\\2\int_0^1 \Gamma_+\dfrac{\omega_+d_+}{c}\;dz+2\int_0^1 \Gamma_-\dfrac{\omega_-d_+}{c}\;dz \\+\left[-\dfrac{1}{\varrho}\left(\dfrac{\omega_+^2d_+^2\varrho^2}{3c^2}+\dfrac{\omega_+ d_+ \varrho^2}{c}+\varrho^2\right)-\dfrac{1}{d^3}\left(\dfrac{\omega_-^2d_+^2d^4}{3c^2}-\dfrac{\omega_-d_+d^3}{c}+d^2\right)+\alpha \right]\eta
\\
           \dfrac{\Gamma_{+_{z}}}{\varrho}+\dfrac{\omega_+d_+\eta(2z-1)}{c}
\\
           \dfrac{z(\varrho-\dfrac{\omega_+d_+\varrho(z-1)}{c})}{\beta}\Biggl(\gamma+\varrho\bar{\phi}_+(1)+\int_0^1\dfrac{2z\bar{\phi}_+\omega_+d_+\varrho}{c}\;dz\Biggr.\\
           \Biggl.\phantom{\dfrac{z(\varrho-\dfrac{\omega_+d_+\varrho(z-1)}{c})}{\beta}}\qquad \qquad+\int_0^1\dfrac{2z\bar{\phi}_-\omega_-d_+d}{c}\;dz -\bar{\phi}_-(1)\Biggr)-\varrho\bar{\phi}_{+_{z}}\\\dfrac{\Gamma_{-_{z}}}{d}+\dfrac{\omega_-d_+\eta(2z-1)}{c}
\\
           \dfrac{z(-d+\dfrac{\omega_-d_+d^2(1-z)}{c})}{d\beta}\Biggl(\gamma+\varrho\bar{\phi}_+(1)+\int_0^1\dfrac{2z\bar{\phi}_+\omega_+d_+\varrho}{c}\;dz\Biggr.\\\Biggl.\phantom{\dfrac{z(-d+\dfrac{\omega_-d_+d^2(1-z)}{c})}{d\beta}} \qquad \qquad +\int_0^1\dfrac{2z\bar{\phi}_-\omega_-d_+d}{c}\;dz -\bar{\phi}_-(1)\Biggr)-\dfrac{\bar{\phi}_{-_{z}}}{d}
           \end{bmatrix},
\end{aligned}
\end{equation}
coupled with the linearized boundary conditions
\begin{equation}\label{linearized boundary conditions}
    \begin{aligned}
        \bar{\phi}_{+_{z}}(1)&=\dfrac{1}{\beta}\left(\gamma+\varrho\bar{\phi}_+(1)+\int_0^1\left(\dfrac{2z\bar{\phi}_+\omega_+d_+\varrho}{c}+\dfrac{2z\bar{\phi}_-\omega_-d_+d}{c}\right)\;dz -\bar{\phi}_-(1)\right) ,\\\bar{\phi}_{-_{z}}(1)&= \dfrac{-d}{\beta}\left(\gamma+\varrho\bar{\phi}_+(1)+ \int_0^1\left(\dfrac{2z\bar{\phi}_+\omega_+d_+\varrho}{c}+\dfrac{2z\bar{\phi}_-\omega_-d_+d}{c}\right)\;dz  -\bar{\phi}_-(1)\right),\\ \bar{\phi}_{\pm_{z}}(0)&=0.
    \end{aligned}
    \end{equation}
    
    Consider the eigenvalue problem $\mathcal{L}u=\lambda u$, together with the boundary conditions \eqref{linearized boundary conditions}. Upon setting $\lambda=ik$, we obtain the dispersion relation
    \begin{equation}
        \alpha + \beta k^2 = \dfrac{k\varrho}{\tanh{(k)}}+\dfrac{k}{\tanh{(kd)}}+\left( \dfrac{\omega_+d_+\varrho}{c}-\dfrac{ \omega_-d_+}{c}\right).
    \end{equation}We would like to mention that this dispersion relation is equivalent and consistent to the dispersion relation obtained in \eqref{dispersion relation}

Our next objective in the construction of small-amplitude solutions is to apply the center manifold approach due to Mielke \cite{mielke1988} to the system \eqref{lineraized operator} and \eqref{linearized boundary conditions}. For convenience, the main approach used for this is outlined in the theorem stated in Appendix~\ref{quoted results appendixa}, which is a version used, for instance, in \cite[Section 3]{Nilsson2017}. Due to nonlinear boundary conditions \eqref{linearized boundary conditions}, however, we are not able to crudely implement the theorem right away. As an intermediate step, we do change of variables via the operator $\mathfrak{G}$.  This linearizes the boundary conditions at the cost of complication of the problem in the bulk. Explicitly, the operator $\mathfrak{G}$ takes the form

\begin{equation}\label{operator G}
\mathfrak{G}(\eta,\gamma,\bar{\phi}_+, \Gamma_+,\bar{\phi}_-,\Gamma_-)=(\eta,\nu,\varphi_+,\Gamma_+,\varphi_-,\Gamma_-),
\end{equation}
where,
\begin{equation}\label{variables for boundary linearization}
\left\{\begin{aligned}
    \nu&=\varrho \bar{\phi}_+(1) -\bar{\phi}_-(1),\\
    \varphi_+&=\varrho \bar{\phi}_+ +W \left(A[\Gamma_+](z)-\dfrac{\varrho}{2}(z^2-\frac{1}{3}) \right),\qquad 
     \varphi_-=\bar{\phi}_- -W \left(A[\Gamma_-](z)-\dfrac{d}{2}(z^2-\frac{1}{3}) \right),
\end{aligned}\right.
\end{equation} and

\[
\left\{\begin{aligned}
W&=\dfrac{\bar{\gamma}}{\sqrt{\beta^2-\bar{\gamma}^2}},\\
A[f](s)&=\int_0^z sf_s\;ds-\int_0^1\int_0^z s f_s(s)\; ds\; dz.
\end{aligned}\right.
\]
One may check easily that $\varphi_{\pm_z} (1)=\varphi_{\pm_z}(0)=0$. Further, via the definition of $\varphi_{\pm}$ in \eqref{variables for boundary linearization}, one can check that
\[\int_0^1 \varphi_\pm\; dz=0
\quad \text{provided} \quad \int_0^1 \bar{\phi}_\pm\; dz =0
.\]

It is also worth noting that the operator $\mathfrak{G}$ is invertible in some neighborhood of the origin and its inverse is explicitly given by
\[
\mathfrak{G^{-1}} \begin{bmatrix}
\eta\\ \nu\\ \varphi_+\\ \Gamma_+\\ \varphi_-\\\Gamma_-)
\end{bmatrix}= 
\begin{bmatrix} 
\eta\\
\frac{\beta R}{\sqrt{1+R^2}}-\text{I}-\text{II}\\
\frac{\varphi_+}{\varrho}-\frac{R}{\varrho}\left(A[\Gamma_+](z)-\frac{\varrho}{2}(z^2-\frac{1}{3})\right)\\\Gamma_+\\
\varphi_-+R\left(A[\Gamma_-](z)-\frac{d_+}{2}(z^2-\frac{1}{3})\right)\\
\Gamma_-
\end{bmatrix} 
,\]
where
\[
\begin{aligned}
\text{I}&=\displaystyle \int_0^1 \dfrac{z}{\eta-1} \left( \frac{\varphi_{+_z}}{\varrho} - Rz(\Gamma_{+_z} -\varrho +\dfrac{\omega_+d_+\varrho(z-1)}{c})\right)\left(\Gamma_{+_z} -\varrho + \frac{w_+d_+\varrho}{c}(z-1)\right)\;dz,\\
\text{II}&=\displaystyle \int_0^1 \dfrac{z}{\eta+d} \left(\varphi_{-_z} - Rz(\Gamma_{-_z} -d+\dfrac{\omega_-d_+d^2(1-z)}{c})\right)\left(\Gamma_{-_z} -d + \frac{w_-d_+d^2}{c}(1-z)\right)\;dz,\\
R&=\dfrac{\varphi_+(1)-\varphi_-(1)-\nu}{A[\Gamma_+](1)+A[\Gamma_-](1)-\frac{\varrho+d}{3}}.
\end{aligned}
\]
In this new coordinate system, the Hamiltonian exhibits a new expression:
\begin{equation}
\begin{aligned}
      H&=\int_{0}^1 \dfrac{1}{2\varrho(1-\eta)}\left(\left(\Gamma_{+_{z}} +\dfrac{\omega_+d_+\varrho(z-1)}{c}-\eta\varrho\right)^2-\left(\varphi_{+_z}-Rz(\Gamma_{+_z}-\varrho)\right)^2\right) dz\\& \quad +\int_{0}^1 \dfrac{1}{2(d+\eta)}\left(\left(\Gamma_{-_{z}} +\dfrac{w_-d_+d^2(1-z)}{c}+\eta\right)^2-\left(\varphi_{-_z}+Rz\left(\Gamma_{-_z}-d\right)\right)^2\right) dz\\&\quad+\dfrac{1}{c}\int_{0}^1\left(\Gamma_{+_z} -\varrho + \dfrac{\omega_+d_+\varrho(z-1)}{c}\right)\omega_+ d_+(z(\eta-1)+1)\;dz\\&\quad+\dfrac{1}{c}\int_{0}^1 \left(\Gamma_{-_z} -d_+ + \dfrac{\omega_-d_+d^2(1-z)}{c}\right)\omega_- d_+(z(\eta+d)-d)\;dz\\&\quad-\dfrac{\beta}{\sqrt{1+R^2}} +\beta + \dfrac{\alpha}{2}\eta^2 +\dfrac{\omega_+^2d_+^2\varrho}{6c^2}+\dfrac{\omega_-^2d_+^2d^3}{6c^2}+\dfrac{\omega_+d_+\varrho}{2c}-\dfrac{\omega_-d_+d^2}{2c}.
\end{aligned}.
\end{equation}
Consequently, the Hamilton's equations now become
\begin{equation}
    \begin{aligned}
          \dot{\eta}&=R,\\\dot{\nu}&= \frac{1}{\eta-1}\left(-\Gamma_{+_z}(1) + \frac{\omega_+d_+}{2c}+R\left[-R(\Gamma_{+_z}-\varrho)-\varphi_+(1)+R\left(A[\Gamma_+](1)-\frac{\varrho}{3}\right)\right]\right)\\&\quad-\frac{1}{\eta+d}\left(-\Gamma_{-_z}(1) + \frac{\omega_-d_+d^2}{2c}+R\left[R(\Gamma_{-_z}-d)-\varphi_-(1)-R\left(A[\Gamma_-](1)-\frac{d}{3}\right)\right]\right),\\
          \dot{\varphi}_+&=\frac{1}{\eta-1}\Bigl(-\Gamma_{+_z}+\frac{\omega_+d_+(2z-1)}{2c}+R\Bigl[z\varphi_{+_z}-Rz^2\left(\Gamma_{+_z}-\varrho\right) \Bigr.\Bigr.
          \\ &\Bigl.\Bigl.\quad-\varphi_+(1) +R\left(A[\Gamma_+](1)-\frac{\varrho}{3}\right)\Bigr]\Bigr)\\&\quad+\frac{\dot{\bar{\gamma}}(1+R^2)^{3/2}}{\beta} \left(A[\Gamma_+](z)-\frac{\varrho}{2}(z^2-1/3)\right)+\frac{RA[\varphi_{+_z}](z)}{\eta-1},\\
          \dot{\Gamma}_+&=\frac{1}{\eta-1} \varphi_{+_z},\\
          \dot{\varphi}_-&=\frac{1}{\eta+d}\Bigl(-\Gamma_{-_z}+\frac{\omega_-d_+d^2(1-2z)}{2c}+R\Bigl[z\varphi_{-_z}+Rz^2(\Gamma_{-_z}-d)  \Bigr.\Bigr.
          \\ &\Bigl.\Bigl.\qquad-\varphi_-(1) -R(A[\Gamma_-](1)-\frac{d}{3})\Bigr]\Bigr)\\&\qquad+\frac{\dot{\bar{\gamma}}(1+R^2)^{3/2}}{\beta} \left(A[\Gamma_-](z)-\frac{d}{2}(z^2-1/3)\right)+\frac{RA[\varphi_{-_z}](z)}{\eta+d},\\
           \dot{\Gamma}_-&=\frac{-1}{\eta+d} \varphi_{-_z},
    \end{aligned}
    \end{equation}
    where
\[
\dot{\bar{\gamma}}=(1+R^2)\left(\dfrac{(\Gamma_{+_z}+\dfrac{\omega_+d_+\varrho(z-1)}{c}-\varrho)^2}{2\varrho(\eta-1)^2}+\dfrac{(\Gamma_{-_z}+\dfrac{\omega_-d_+d^2(1-z)}{c}-d)^2}{2(\eta+d)^2}\right)+\dfrac{\varrho-1}{2}+\alpha\eta.
\]

Recall that we are interested in the solutions that satisfy the condition $(\beta,\alpha)=(\beta,\alpha_0)+(0,\epsilon^2)$ with $\beta >\beta_0$ where these parameters are defined in \eqref{Definition of beta and alpha}. It can be shown that the imaginary part of the spectrum of the linearized operator $\mathcal{L}$ consists of zero, which is an eigenvalue of (algebraic) multiplicity 2 when $ \alpha=\alpha_0$ and $\beta=\beta_0$, as given by \eqref{definition of Alphanot and Betanot}. The associated eigenvector and the generalized eigenvector, namely $e_1$ and $e_2$, of the zero eigenvalue are then computed. Explicitly, they take the form
\begin{equation}\label{eigenvectors}
    e_1=\begin{pmatrix} 1 \\ 0 \\0 \\ \dfrac{1}{c} \omega_+ d_+ \varrho (z-z^2) \\ 0 \\ \dfrac{1}{c}\omega_-d_+d(z-z^2) \end{pmatrix},\
 e_2=\begin{pmatrix} 0 \\ \beta -\dfrac{\varrho+d}{3}-(\dfrac{\omega_+d_+\varrho-\omega_-d_+d^2}{12c}) \\ \dfrac{z^2-1/3}{2} \\ 0 \\ \dfrac{-d(z^2-1/3)}{2} \\ 0 \end{pmatrix}.
\end{equation}
 It is straightforward to check that $\mathcal{L}e_1=0$ and $\mathcal{L}e_2=e_1$ with $\tilde{\Omega}(e_1,e_2)=\beta -\dfrac{\varrho +d}{3}=:\beta_*$, provided $\alpha=\varrho+\dfrac{1}{d}+\dfrac{\omega_+d_+\varrho}{c}-\dfrac{\omega_-d_+}{c}$.
 
 Let 
 \[v_1=\dfrac{e_1}{\sqrt{\beta_*}} \; \text{and} \; v_2=\dfrac{e_2}{\sqrt{\beta_*}}
 .\] 
 They, indeed, form a symplectic basis of  the vector space spanned by the eigenvectors $e_1$ and $e_2$. Let us define $f_i:=d\mathfrak{G}(0)(v_i)$.
 Upon applying the center manifold theorem along with Darboux's theorem, we obtain a Hamiltonian system $(X^{\mu}_C, \Psi,\tilde{H}^{\mu})$,
\[X^{\mu}_C=\{u_1+r(u_1,\mu):u_1 \in \tilde{U_1}\}\]
 and $\tilde{U}_1$ is a neighborhood of 0 as stated in Appendix~\ref{quoted results appendixa}. We would like to note here that all hyphothesis $H1-H4$ in the center manifold theorem are satisfied. As a conclusion, we obtain small bounded solutions on the two-dimensional center manifold. Precisely, every solution $u_1$ can be respresented as
 \[
 u_1=(q,p)=qf_1+pf_2,
 \]
 where
\begin{equation}
     f_1=\dfrac{1}{\sqrt{\beta_*}} \begin{pmatrix} 1 \\ 0 \\0 \\ \dfrac{1}{c} \omega_+ d_+ \varrho (z-z^2) \\ 0 \\ \dfrac{1}{c}\omega_-d_+d(z-z^2) \end{pmatrix},\
 f_2= \dfrac{1}{\sqrt{\beta_*}} \begin{pmatrix} 0 \\ \dfrac{\varrho+d}{3}\\ 0 \\ 0 \\ 0 \\ 0 \end{pmatrix}.
\end{equation} 
Upon completing a number of nontrivial computations, the Taylor expansion of the reduced Hamiltonian is derived and explicitly given by

\begin{equation}
\begin{split}
        \tilde{H}^{\mu}(q,p)&=\dfrac{1}{2}p^2-\dfrac{1}{2\beta_*}\epsilon^2 q^2+\dfrac{\varrho-\dfrac{1}{d^2}+\dfrac{\omega_+d_+\varrho}{c}+\dfrac{\omega_-d_+}{cd}+\dfrac{\omega_+^2d_+^2\varrho}{3c^2}-\dfrac{\omega_-^2d_+^2}{3c^2}}{2\beta_*^{2}}q^3\\& \quad +O(|p||(p,q)||\epsilon^2,p,q|)+O(|(p,q)|^2|(\epsilon^2,p,q)|^2) .
\end{split}
\end{equation}
From there, we obtain the corresponding Hamilton's equations
\begin{equation}
    \begin{aligned}\label{reduced hamilton's eqs}
    q_x&=p+O(|(p,q)||\epsilon^2,p,q|), \\
    p_x&=\dfrac{\epsilon^2 q}{\beta_*}+\dfrac{-\varrho+\dfrac{1}{d^2}-\dfrac{\omega_+d_+\varrho}{c}-\dfrac{\omega_-d_+}{cd}-\dfrac{\omega_+^2d_+^2\varrho}{3c^2}+\dfrac{\omega_-^2d_+^2}{3c^2}}{2\beta_*^{3/2}}q^2 &\\&\quad +O(|p||\epsilon^2,p,q|)+O(|(p,q)|^2|(\epsilon^2,p,q)|^2).
    \end{aligned}
\end{equation}
Consider the following rescaling
\begin{equation}
    X=\dfrac{\epsilon}{\sqrt{\beta_*}}x, \quad q(x)=\beta_*^2\epsilon^2 Q(X), \quad p(x)=\epsilon^{3}\beta_*^{3/2}P(X).
\end{equation}
Under these rescaling, the Hamilton's equations in \eqref{reduced hamilton's eqs} read

\begin{equation}
    \begin{aligned}\label{rescalled hamilton's eq}
        Q_X&= P + O(\epsilon),\\
        P_X&= Q+\dfrac{3K(\varrho,d,\omega_+,\omega_-,c)}{2}Q^2 +O(\epsilon),
    \end{aligned}
\end{equation} 
where 
\begin{equation}
    K(\varrho,d,\omega_+,\omega_-,c)=\beta_*^{3/2}\left(-\varrho+\dfrac{1}{d^2}-\dfrac{\omega_+d_+\varrho}{c}-\dfrac{\omega_-d_+}{cd}-\dfrac{\omega_+^2d_+^2\varrho}{3c^2}+\dfrac{\omega_-^2d_+^2}{3c^2}\right).
\end{equation}
Upon truncating the rescaled Hamilton's equations in \eqref{rescalled hamilton's eq}, we obtain 
\begin{equation}
        Q_X= P,\qquad
        P_X= Q+\dfrac{3K(\varrho,d,\omega_+,\omega_-,c)}{2}Q^2,
\end{equation} 
which has solutions
\begin{equation}
    \begin{aligned}
        Q(X)&= \dfrac{-\sech^2(X/2)}{K(\varrho,d,\omega_+,\omega_-,c)},\\
        P(X)&= \dfrac{\sech^2(X/2)\tanh(X/2)}{K(\varrho,d,\omega_+,\omega_-,c)}.
    \end{aligned}
\end{equation} 

Thanks to the structure of the symplectic basis in \eqref{eigenvectors}, we then obtain the profile of $\eta$ in the original variables
\[
\eta(x)=\dfrac{d_+\epsilon^2 \sech^2(\dfrac{\epsilon x}{2d_+\sqrt{\beta_*}})}{\varrho-\dfrac{1}{d^2}+\dfrac{\omega_+d_+\varrho}{c}+\dfrac{\omega_-d_+}{cd}+\dfrac{\omega_+^2d_+^2\varrho}{3c^2}-\dfrac{\omega_-^2d_+^2}{3c^2}}+O(\epsilon^3).
\]
Observe that, depending on the sign of the denominator in the expression above, we obtain a wave of depression or elevation. Hence, the proof of Theorem~\ref{existence theorem} is now complete.

\section{Stability}\label{Stability}
\subsection{General theory}\label{general theory}
Having proved the existence of small-amplitude waves, in the remaining part of the work, we will investigate the aspect concerning their orbital stability/instability. For that, we are using the general theory stated in \cite{VarholWahlenWalsh2020}, which is a variant of the well known GSS machinery introduced in \cite{GSS11990, GSS21990}. As it is to any mathematical approach, there are a number of preliminary assumptions that first have to hold before applying the theory. For the water wave problem, however, there are some conditions that obstruct a direct use of the classical GSS approach. The variant in \cite{VarholWahlenWalsh2020} essentially solves these issues by weakening some of the requirements in GSS which then permits its application to the water wave problem. Although, the assumptions are weakened, the final conclusions of the both approaches in \cite{VarholWahlenWalsh2020} and \cite{GSS11990, GSS21990} remain the same. We have outlined all the required hypothesis below and we will refer to them again later in the paper. For more in-depth and detailed explanations on the general theory, see \cite[Section 2]{VarholWahlenWalsh2020} and the references therein.
\begin{assumption}[Spaces]\label{assumption on spaces}
Let $\mathbb{X}, \mathbb{V}, \mathbb{W}$ be spaces defined by \eqref{space X},\eqref{space V}, and \eqref{space W}. Assume there exists a  constant $C>0$ and $\theta \in (0,1]$, so that the following inequality holds
\begin{equation}
    \norm{u}^{3}_{\mathbb{V}}\leq C \norm{u}^{2 +\theta}_{\mathbb{X}}\norm{u}^{1-\theta}_{\mathbb{W}}, \text{ for any } u\in \mathbb{W}.
\end{equation}

Let $\mathcal{O} \subset \mathbb{X}$ be an open set where solutions live. Assume that $J: \mathcal{D}(J) \subset \mathbb{X}^* \to \mathbb{X}$ is a closed linear operator and for any $u\in \mathcal{O}\cap \mathbb{V}$.
\end{assumption}
\begin{assumption}[Poisson map] \label{assumption of poisson map}
\
\begin{enumerate}
    \item The domain $\mathcal{D}(J)$ is dense in $\mathbb{X}^*$.
    \item $J$ is injective.
    \item For each $u \in \mathcal{O}\cap \mathbb{V}$, $J(u)$ is skew-adjoint, that is
    \begin{equation}
    \langle J(u)v,w \rangle=-\langle v, J(u)w \rangle,
    \end{equation}
    for all $v,w \in \mathcal{D}(J).$
\end{enumerate}

\end{assumption}
\begin{assumption}[Derivative extension]\label{assumption on Derivative extension}
Assume that there exist (extension) mappings  $\nabla E, \nabla P \in C^0(\mathcal{O}\cap \mathbb{V};\mathbb{X}^*)$ of $\T{D}E(u)$ and $\T{D}P(u)$ respectively for all $u \in \mathcal{O} \cap \mathbb{V}.$ 
\end{assumption}

Suppose also that there exists a family of affine maps $T(s):\mathbb{X} \to \mathbb{X}$ parameterized by $s$ where the linear part $dT(s):=T(s)u-T(s)0$ enjoys a number of properties.
\begin{assumption}[Symmetry group]\label{assumption on symmetry group}
The symmetry group $T(\cdot)$ satisfies
\begin{enumerate}
    \item (Invariance) The neighborhood $\mathcal{O}$, and the subspaces $\mathbb{V}$ and $\mathbb{W}$, are all invariant under the symmetry group. Moreover $I^{-1} \mathcal{D}(J)$ is invariant under the linear symmetry group i.e. $\mathcal{D}(J)$ is invariant under the adjoint $dT^*(s):\mathbb{X}^* \to \mathbb{X}^*$.
    \item (Flow property) Assume $T(0)=dT(0)=\text{Id}_{X}$, for any $r,s \in \mathbb{R}$, we have
    \begin{equation}
        T(s+r)=T(s)T(r), \qquad dT(s+r)=dT(s)dT(r).
    \end{equation}
    \item (Unitary) The linear part $dT(s)$ is a unitary operator on $\mathbb{X}$ and an isometry on $\mathbb{V}$ and $\mathbb{W}$ for each $s\in \mathbb{R}$.
    \item (Strong continuity) The symmetry group is strongly continuous on $\mathbb{X},\mathbb{V}$, and $\mathbb{W}$.
    \item (Affine part) The function $T(\cdot)0$ belongs to $C^3(\mathbb{R}; \mathbb{W})$ and there exists an increasing function $\iota:[0,\infty) \to [0,\infty)$ such that 
    \begin{equation}
        \lVert T(s)0\rVert_{\mathbb{W}} \neq \iota (\lVert T(s)0\rVert_{\mathbb{X}}), \quad \T{for all } s\in \mathbb{R}.
    \end{equation}
     \item (Commutativity with  $J$) For all $s \in \mathbb{R},$
     \begin{equation}
         JI T(s)= T(s) JI.
     \end{equation}
     \item (Infinitesimal generator) The infinitesimal generator of $T$ is the affine mapping
     \begin{equation}
     T'(0)u=\lim_{s\to 0}\dfrac{T(s)u-u}{s}=dT'(0) +T'(0)0,
     \end{equation}
     with dense domain $\mathcal{D}(T'(0)) \subset \mathbb{X}$ consisting of all $u \in\mathbb{X}$ such that the limit above exists in $\mathbb{X}$ (similarly for the spaces $\mathbb{V}$ and $\mathbb{W}$). Assume also that $\nabla P(u) \in \mathcal{D} (J)$ and that 
     \begin{equation}
         T'(0)u=J(u)\nabla P(u),
     \end{equation}
     for all $u \in \mathcal{D}(T'(0)|_{\mathbb{V}}) \cap \mathcal{O}.$
     \item (Density) The subspace
     \begin{equation}
     \mathcal{D}(T'(0))|_{\mathbb{W}} \cap \text{Rng} J
     \end{equation}
     is dense in $\mathbb{X}$.
     \item (Conservation) For all $u \in \mathcal{O} \cap \mathbb{V}$, the energy is conserved by the flow of the symmetry group, meaning
     \begin{equation}
         E(u)=E(T(s)u), \qquad \T{for all } s\in \mathbb{R}.
     \end{equation}
\end{enumerate}
We call $u\in C^1(\mathbb{R}; \mathcal{O}\cap \mathbb{W})$ to be a \textit{bound state} of the abstract Hamiltonian \eqref{abstract hamiltonian}, if $u$ is a solution of the form
\begin{equation}
    u(t)=T(ct)U,
\end{equation}
for some $c\in \mathbb{R}$ and $U\in \mathcal{O} \cap \mathbb{W}$.
\end{assumption}
We would like to point out that the water wave problem in the present setting satisfies Assumption \ref{assumption on symmetry group}. In order to see this, one has to go through a series of lengthy, yet elementary, computations. Therefore, we will avoid checking the majority of  the requirements in Assumption \ref{assumption on symmetry group} here. Instead, we will focus more on showing that Assumption \ref{assumption on symmetry group}(8) holds: this is precisely one of requirements from the general theory in \cite{GSS11990} that has been weakened and modified in  \cite{VarholWahlenWalsh2020}. 
\begin{assumption}[Bound states]\label{assumption on bound states} There exists a one-parameter family of bound state solutions $\{U_c: c \in \mathcal{I}\}$ to the Hamiltonian system \eqref{abstract hamiltonian}.
\begin{enumerate}
    \item The mapping $c\in \mathcal{I} \mapsto U_c \in \mathcal{O} \cap\mathbb{W}$ is of class $C^1$.
    \item The non-degeneracy condition $T'(0)U_c \neq 0$ holds for every $c \in \mathcal{I}$. Equivalently, $U_c$ is never a critical point of the momentum.
    \item For all $c\in \mathcal{I}$,
    \begin{equation}
        U_c \in \mathcal{D}(T'''(0)) \cap \mathcal{D}(JIT'(0)),
    \end{equation}
    and 
    \begin{equation}
        JIT'(0)U_c \in \mathcal{D}(T'(0)|_{\mathbb{W}}).
    \end{equation}
    \item It holds that 
    $\lim \inf_{|s|\to \infty} \lVert T(s) U_c-U_c\rVert_{\mathbb{X}}>0$.
\end{enumerate}

\end{assumption}
\begin{assumption}[Spectrum]\label{assumption on spectrum}
The operator $\T{D}^2E_c(U_c) \in \T{Lin}(\mathbb{V},\mathbb{V}^*)$ extends uniquely to a bounded linear operator $H_c:\mathbb{X}\to \mathbb{X}^*$ such that:
\begin{enumerate}
    \item $I^{-1}H_c$ is a self-adjoint operator on $\mathbb{X}$.
    \item The spectrum of $I^{-1} H_c$ satisfies
    \begin{equation}
        \text{spec}(I^{-1} H_c)=\{-\mu_c^2\} \cup\{0\} \cup \Sigma_c,
    \end{equation}
    where $-\mu_c^2<0$ is a simple eigenvalue that correspond to a unit eigenvector $\chi_c$, $0$ is a simple eigenvalue generated by $T$, and $\Sigma_C \subset (0,\infty)$ is bounded away from 0.
\end{enumerate}
\end{assumption}
\subsection{Notion on stability/instability}

After outlining required assumptions for the general theory, we now proceed to define the notion on stability/instability concerned here. At this point, the functions spaces that we are working with are still abstract and will be specified soon in the next subsection. Fix a bound state $U_c$ and radius $r>0$, we define the following sets
\begin{equation}
\begin{split}
    &\mathcal{U}^{\mathbb{X}}_r:=\{u \in \mathcal{O}: \inf_{s \in \mathbb{R}} 
    \norm{u-T(s)U_c}_{\mathbb{X}}<r\},\\&
    \mathcal{U}^{\mathbb{W}}_r:=\{u \in\mathcal{O} \cap \mathbb{W}: \inf_{s \in \mathbb{R}} 
    \norm{u-T(s)U_c}_{\mathbb{W}}<r\}.
\end{split}
\end{equation}
Fix $R>0$, let $\mathcal{B}^{\mathbb{W}}_{R}$ denote the intersection between the ball of radius $R$ centered at the origin in $\mathbb{W}$ and the set $\mathcal{O}$. 

\begin{definition}
The bound state $U_C$ is conditionally orbitally stable provided that for any $r>0$ and $R>0$, there exists $r_0>0$ such that if $u:[0,t_0] \to \mathcal{B}^{\mathbb{W}}_{R}$ is a solution to \eqref{Euler-Incompressible in terms of Harmonic Conj} where $u(0) \in \mathcal{U}^{\mathbb{X}}_{r_0}$ then $u(t) \in \mathcal{U}^{\mathbb{X}}_r$ for all $t\in[0,t_0)$.
\end{definition}

Note that the definition of orbital stability above is the same as the definition of stability introduced earlier in Section~\ref{Introduction}. Here, we are presenting it in a more general and abstract manner.

From the general theory \cite{GSS11990,GSS21990,VarholWahlenWalsh2020} , the conclusion on stability/instability can be determined by looking at the sign of the second derivative of a scalar-valued function known as \textit{moment instability} 
\begin{equation}
    d(c):=E_c(U_c)=E(U_c)-cP(U_c). 
\end{equation}
This leads us to state the following theorem.

\begin{theorem}[Stability/Instability]\label{stability theorem} Let all Assumptions \ref{assumption on spaces}-\ref{assumption on spectrum} be satisfied. The bound state $U_c$ under consideration is conditionally orbitally stable (unstable), provided $d''(c)>0\;(<0)$. 
\end{theorem}


\section{Hamiltonian Formulation}\label{hamiltonian formulation }
\subsection{Nonlocal operators}
 We begin this section by reformulating the governing equations \eqref{Euler-Incompressible in terms of Harmonic Conj} in terms of variables restricted to the interface $\eta(x,t)$ in the spirit of Zakharov--Craig--Sulem. Although this way of formulating the problem forces us to work with some complicated non-local operators (pseudo-differential operators), it simplifies the problem by pushing all the unknowns to the boundary, in this case, the internal interface. The idea was then adopted by a number of authors studying internal waves, for instance, \cite{BenjaminBridges1997}, \cite{CraigGroves2000}.
 
 Let $\xi_\pm$ be defined as the trace of $\phi_\pm$ on the interface $y=\eta(x,t)$ for the upper and lower regions of the fluid.
 It is clear that
 \begin{equation}
     \xi_\pm'=(\partial_x\phi_\pm)|_{y=\eta}+\eta'(\partial_y\phi_\pm)|_{y=\eta}.
 \end{equation}
 Additionally, we  denote $\mathscr{H}_\pm$ as the Hilbert transform acting on $\xi_\pm$:
\begin{equation}\label{Hilbert Transform}
    \mathscr{H}_\pm(\eta) \xi_\pm= \tilde{\psi}_\pm(t,x,\eta)=\psi_\pm(t,x,\eta) + \dfrac{\omega_\pm }{2}\eta^2.
\end{equation} 
For later use, let us also introduce the Dirichlet--Neumann operator in $\Omega_+$ and $\Omega_-$ (for a fixed $\eta$):
\[\mathcal{G}_\pm(\eta)\xi_\pm:=\langle\eta'\rangle\big(N_\pm \cdot \nabla \mathcal{H}_\pm(\eta)\xi_\pm\big),\]
where $N_\pm$ is the outward unit normal relative to the domain $\Omega_\pm$ along the internal interface $\mathscr{S}$ and the Japanese bracket $\langle \cdot \rangle:=\sqrt{1+|\cdot|^2}$. Further, the notation $\mathcal{H}_\pm(\eta) \xi_\pm$ denotes the harmonic extension of $\xi_\pm$ to $\Omega_\pm$ and uniquely solves
\begin{equation}\label{harmonic extension}
\left\{ \begin{aligned}
    \Delta \mathcal{H}_\pm(\eta)\xi_\pm&=0  &\T{in}\; \Omega_\pm, \\
    \mathcal{H}_\pm(\eta)\xi_\pm&=\xi_\pm \qquad& \T{on}\; y=\eta(x,t), \\\partial_y \mathcal{H}_\pm(\eta)\xi_\pm&=0\qquad& \T{on}\;y=\pm d_\pm.
\end{aligned} \right. 
\end{equation}

It is known that for $\eta \in H^{k_0+1/2}(\mathbb{R})$, the Dirichlet--Neumann operator $\mathcal{G}_{\pm}(\eta)$ is an isomorphism $\dot{H}^k(\mathbb{R}) \to \dot{H}^{k-1}(\mathbb{R})$, where $\dot{H}^k$ denotes the usual homogeneous Sobolev space of order $k$ and $k\in[1/2-k_0,1/2+k_0]$ for any real number $k_0>1/2$. The operator $\mathcal{H}_{\pm}(\eta)$ is a bounded mapping from $H^k(\mathbb{R})$ to $H^{k+1/2}(\Omega_{\pm})$ and from $\dot{H}^{k}(\mathbb{R})$ to $\dot{H}^{k+1/2}(\Omega_{\pm})$.

\begin{remark}
The spaces $H^k$ and $\dot{H}^k$ are the Sobolev and homogeneous Sobolev spaces, respectively.
\end{remark}

Using the operators $\mathcal{G}_\pm(\eta)$, $\mathcal{H}_\pm(\eta)$, and the new unknown $\xi_\pm$, the water wave problem can be pushed to the boundary $y=\eta(x,t)$. Now, it reads

\begin{equation}
\left\{ \begin{aligned}\label{kinematic dynamic}
    \eta_t&=\mp \mathcal{G}_\pm(\eta)\xi_\pm + \omega_\pm \eta \eta_x \qquad &\T{on}\; y=\eta(x,t), \\\jump{\rho \xi_t}&=-g \eta \jump{\rho}-\sigma \left(\dfrac{\eta_x}{\sqrt{1+\eta_x^2}}\right)_x&\\&\quad-\sum_{\pm}\pm\dfrac{\rho_\pm}{2}\Gamma_\pm(\eta,\xi_\pm) \mp \rho_\pm \omega_\pm \eta \xi_{x_\pm} \pm \rho \omega_\pm \mathscr{H}_\pm(\eta)\xi_\pm \qquad &\T{on}\; y=\eta(x,t),
\end{aligned} \right. 
\end{equation} where
\begin{equation}\label{definition of Gamma}
    \Gamma_\pm(\eta,\xi_\pm)=\left(\xi^2_{x_\pm}-(\mathcal{G}_\pm(\eta)\xi_\pm)^2\pm2\eta_x\xi_{x_\pm}\mathcal{G}_\pm(\eta)\xi_\pm\right)/(1+\eta_x^2).
\end{equation}
Next, in order to reformulate the problem in a Hamiltonian language, we introduce the variable $\tilde{\xi}:=-\jump{\rho \xi}$. Recall, from the kinematic boundary condition in \eqref{kinematic dynamic} we have 
\begin{equation}\label{G_- and G_+}
    \mathcal{G}_-(\eta)\xi_-+\mathcal{G}_+(\eta)\xi_+=\jump{\omega}\eta \eta_x.
\end{equation} This, together with the definition of $\tilde{\xi}$, yields
\begin{equation}\label{Operator G and B}
    \pm\mathcal{G}_\pm \tilde{\xi}=\mathcal{B}(\eta)\xi_\mp-\rho_\pm\jump{\omega}\eta \eta_x,
\end{equation} where $\mathcal{B}(\eta):=\rho_+\mathcal{G}_-(\eta)+\rho_-\mathcal{G}_+(\eta).$ Following the property of $\mathcal{G}_{\pm}(\eta)$, the operator $\mathcal{B}(\eta)$ is also bounded and linear  
from $H^k(\mathbb{R})$ to $H^{k-1}(\mathbb{R})$ and from $\dot{H}^k(\mathbb{R}) \to \dot{H}^{k-1}(\mathbb{R}) $ for any $\eta \in H^{k_0+1/2}(\mathbb{R})$ with $k_0,k$ given as before. Moreover, $\mathcal{B}(\eta)$ is an isomorphism from $\dot{H}^k(\mathbb{R})$ to $\dot{H}^{k-1}(\mathbb{R})$.
Therefore, using \eqref{Operator G and B} and solving for $\xi_\pm$ we obtain
\begin{equation}\label{Equation for Xi}
    \xi_\pm=\mp \mathcal{B}(\eta)^{-1} \mathcal{G}_\mp(\eta)\tilde{\xi}+\rho_\mp \mathcal{B}(\eta)^{-1}\jump{\omega}\eta \eta_x.
\end{equation} The kinematic and dynamic conditions now read
\begin{equation}\label{Kinematic and dynamic condition in terms of Gamma and xi}
\left\{ \begin{aligned}
    \eta_t&=\mathcal{A}(\eta)\tilde{\xi} \mp\rho_\mp\mathcal{G}_\pm(\eta)\mathcal{B}(\eta)^{-1}\jump{\omega}\eta \eta_x+\omega_\pm \eta \eta_x\qquad &\T{on}\; y=\eta(x,t), \\\tilde{\xi}_t&=g \eta \jump{\rho}+\sigma \left(\dfrac{\eta_x}{\sqrt{1+\eta_x^2}}\right)_x&\\&\quad+\sum_{\pm}\left[\pm\dfrac{\rho_\pm}{2}\Gamma_\pm(\eta,\xi_\pm) \mp \rho_\pm \omega_\pm \eta \xi_{x_\pm} \pm \rho_\pm \omega_\pm \mathscr{H}_\pm(\eta)\xi_\pm\right] \qquad &\T{on}\; y=\eta(x,t),
\end{aligned} \right. 
\end{equation}
where $\mathcal{A}(\eta):=\mathcal{G}_{\pm}(\eta) \mathcal{B}(\eta)^{-1}\mathcal{G}_{\mp}(\eta)$.

The above equations can alternatively be written as 
\begin{equation}\label{kinematic and bernoulli}
\left\{ \begin{aligned}
    \eta_t&=\mathcal{A}(\eta)\tilde{\xi} + \mathcal{G}_-(\eta)\mathcal{B}(\eta)^{-1}\rho_+\jump{\omega}\eta \eta_x+\omega_- \eta \eta_x\qquad &\T{on}\; y=\eta(x,t), \\\tilde{\xi}_t&=g \eta \jump{\rho}+\sigma \left(\dfrac{\eta_x}{\sqrt{1+\eta_x^2}}\right)_x+ \dfrac{\jump{\rho|\nabla \psi|^2}}{2} -\eta_t \jump{\rho \phi_y} + \jump{\rho \omega \psi}
\qquad &\T{on}\; y=\eta(x,t).
\end{aligned} \right. 
\end{equation}
We would like to mention that similar formulation involving constant (non-vanishing) vorticity for a one-fluid and two-fluid case can be found, for instance, in \cite{Wahlen2007} and \cite{Compelli2016, Constantin2015} respectively
\begin{remark}
The spaces $H^k$ and $\dot{H}^k$ are the Sobolev and homogeneous Sobolev spaces, respectively.
\end{remark}
\subsection{Function spaces}
The formulation stated in \eqref{kinematic and bernoulli} is crucial in helping us exploit the Hamiltonian structure of the problem. In light of that, let us informally introduce the required function spaces where the internal water wave problem will be posed. Fix $k\geq1/2$, we define a product space for $u$:
\begin{equation}
    \mathbb{X}^k:=\mathbb{X}_1^k \times \mathbb{X}_2^k:=H^{k+1/2}(\mathbb{R}) \times \left(\dot{H}^k(\mathbb{R}) \cap \dot{H}^{1/2}(\mathbb{R})\right).
\end{equation}
For future reference, we will denote $\mathbb{X}^{k+}$ to mean $\mathbb{X}^{k+\epsilon}$ for any $0<\epsilon \ll 1$, similarly for $H^{k+}$.
\begin{remark}\label{Dense space: intersectiom Hp amd dotHq}
Note that the space $H^p(\mathbb{R}) \cap \dot{H}^q(\mathbb{R})$ is dense in $H^p(\mathbb{R})$ and $\dot{H}^q(\mathbb{R})$ for all $p,q \in \mathbb{R}$. We will use this fact to verify Assumption \ref{assumption on spaces} in the general theory.
\end{remark}

Consider the following sequence of (continuously) embedded spaces
\begin{equation}
    \mathbb{W}\hookrightarrow\mathbb{V}\hookrightarrow\mathbb{X},
\end{equation}
where $\mathbb{X}$ is a Hilbert space, while $\mathbb{W}$ and $\mathbb{V}$ are reflexive Banach spaces. In practice, $\mathbb{W}$ will be the local well-posedness space of the internal water wave problem. Its regularity follows from the available results on local well-posedness. The space $\mathbb{X}$ is the natural energy space of the problem with $\mathbb{X}^*$ being its continuous dual:
\begin{equation} \label{space X}
    \mathbb{X}:=H^1(\mathbb{R}) \times \dot{H}^{1/2}(\mathbb{R}),\qquad \mathbb{X}^{*}:=H^{-1}(\mathbb{R}) \times \dot{H}^{-1/2}(\mathbb{R}).
\end{equation} Indeed, if $u\in \mathbb{X}$, then $\nabla \phi_{\pm} \in L^{2}(\Omega_\pm)$. Furthermore, it also informs us that $\eta \in H^1(\mathbb{R})$ which ensures the finiteness of the potential energy. Both combined confirms that the energy is finite on $\mathbb{X}$. The space $\mathbb{V}$ is an intermediate space that lies between the spaces $\mathbb{W}$ and $\mathbb{X}$ where all the conserved quantities are smooth there. 


The regularity of $\mathbb{X}$ turns out to be slightly insufficient for our analysis. This is because the map $u \mapsto \mathcal{G}_\pm(\eta)$ is not smooth with $\mathbb{X}$ being its domain. To fix this, the profile $\eta$ has to be, at least, Lipschitz continuous and bounded away from the rigid walls $\{y=\pm d_\pm\}$. In order to satisfy this level of smoothness condition, we define an intermediate space
\begin{equation} \label{space V}
    \mathbb{V}:=\mathbb{X}^{1+}=H^{3/2+}(\mathbb{R}) \times\left(\dot{H}^{1+}(\mathbb{R}) \cap \dot{H}^{1/2}(\mathbb{R})\right)
\end{equation} along with a neighborhood
\begin{equation}
    \mathcal{O}:=\{(\eta,\tilde{\xi})\in \mathbb{X}:-d_-<\eta<d_+\},
\end{equation}
which makes sure that $\eta$ is bounded away from the rigid walls.
Moreover, observe that $H^{3/2+}(\mathbb{R})\hookrightarrow W^{1,\infty}(\mathbb{R})$, meaning
if $(\eta,\tilde{\xi})\in \mathbb{V}$, then $\eta$ is Lipschitz continuous.

Finally, since the current result on the Cauchy problem for water wave is not yet available with the level of regularity in $\mathbb{V}$, thus we define a smoother space
\begin{equation}\label{space W}
    \mathbb{W}:=\mathbb{X}^{5/2+}=H^{3+}(\mathbb{R}) \times \left(\dot{H}^{5/2+}(\mathbb{R}) \cap \dot{H}^{1/2}(\mathbb{R})\right).
\end{equation} It is worth mentioning that the work of Shatah and Zeng \cite{ShatahZheng2011} proves the local-wellposedness of the water wave problem at the same regularity as in $\mathbb{W}$. Additionally, we would like to point out that the regularity of these spaces is the same and largely follows from the recent work of Chen and Walsh \cite{Ming--WalshOrbital2022}.

Having specified the function spaces, we express the relationship between the trio function spaces $\mathbb{X}, \mathbb{V},$ and $\mathbb{W}$ via an inequality recorded in the next lemma. Moreover, the content of the lemma shows that Assumption \ref{assumption on spaces} in the general theory is satisfied. 
\begin{lemma}[Interpolation] Consider the following spaces: $\mathbb{X}, \mathbb{V}$, and $\mathbb{W}$ defined by \eqref{space X},\eqref{space V}, and \eqref{space W}, respectively. Then there exists a constant $C>0$ and $\theta \in (0,\frac{1}{4})$ such that 
\[
\norm{u}_{\mathbb{V}}^3 \leq C \norm{u}_{\mathbb{X}}^{2+\theta}\norm{u}_{\mathbb{W}}^{1-\theta}, 
\]
for all $u\in \mathbb{W}$.
\end{lemma}
\begin{proof}
This follows from the Gagliardo--Nirenberg interpolation theorem.
\end{proof}
The above lemma shows that a small cubic term in $\mathbb{V}$ norm can be bounded using a quadratic term in $\mathbb{X}$. This fact is needed in the general theory when bounding some of the terms resulted from Taylor expanding  functionals whose domain is $\mathbb{V} \cap \mathcal{O}$.

\subsection{Hamiltonian Structure}
In \cite{BenjaminBridges1997} Benjamin and Bridges formulated the internal water wave problem as a Hamiltonian system in the style of Zakharov--Craig--Sulem. Inspired by the aforementioned paper, we show that the water wave problem also exhibits a Hamiltonian structure (with a non-canonical Poisson map) that can be exploited for the stability analysis. Following the same idea as in \cite[Section 3.3]{Ming--WalshOrbital2022}, we derive the energy functional of \eqref{kinematic and bernoulli}
\begin{equation}\label{Hamiltonian}
  \begin{aligned}
E(\eta,\tilde{\xi})&=\dfrac{1}{2}\int_{\mathbb{R}}\tilde{\xi}\mathcal{A}(\eta) \tilde{\xi}+2\rho_+\jump{\omega}\eta \eta_x\mathcal{G}_-(\eta)\mathcal{B}^{-1} \tilde{\xi}-\rho_+\rho_-\jump{\omega}\eta \eta_x \mathcal{B}^{-1}\jump{\omega}\eta \eta_x&\\&\quad -g \jump{\rho}\eta^2+2\tilde{\xi}\omega_-\eta \eta_x -\dfrac{\eta^3 \jump{\rho \omega^2}}{3}+2\sigma(\sqrt{1+\eta_x^2}-1)\;dx.
\end{aligned}  
\end{equation}
Observe that $E \in C^{\infty}(\mathcal{O} \cap \mathbb{V};\mathbb{R}).$ Moreover, we will show that there is an extension mapping $\nabla E(u)$ of $\T{D}E(u)$ defined on the dual space $\mathbb{X}^*$ which is the content of the lemma below.

\begin{lemma}[Energy Extension]\label{Energy Extension}
There exists a mapping $\nabla E \in C^{\infty}(O \cap \mathbb{V}; \mathbb{X}^*)$ such that
\begin{equation}
    \langle\nabla E(u),v \rangle_{\mathbb{X}^* \times \mathbb{X}}=\T{D}E(u)v,\; \text{for all } u\in \mathcal{O} \cap \mathbb{V} ,\; v\in \mathbb{V}.
\end{equation}
\end{lemma}
\begin{proof}
Fix  $u=(\eta,\tilde{\xi})\in O \cap \mathbb{V}$ and let $\dot{u}=(\dot{\eta},\dot{\tilde{\xi}})\in \mathbb{V}$ be given. Using the definition of the energy \eqref{Hamiltonian} together with the self-adjointness properties of $\mathcal{A}(\eta), \mathcal{G}_\pm(\eta)$, and $\mathcal{B}^{-1}(\eta)$, one can show that
\begin{equation}\label{D eta E}
 \begin{split}
     &\T{D}_{\eta}E(u)\dot{u}=
     \\&\; \dfrac{1}{2}\int_{\mathbb{R}} \tilde{\xi}\langle \T{D} \mathcal{A}(\eta)\dot{\eta},\tilde{\xi} \rangle\; dx -\int_{\mathbb{R}}\left(g \jump{\rho}\eta +\sigma \left(\dfrac{\eta'}{\langle\eta'\rangle}\right)'\right) \dot{\eta}\; dx \\&  \; +\int_{\mathbb{R}}\rho_+ \tilde{\xi} \langle \T{D} \mathcal{G}_-(\eta)\dot{\eta}, \mathcal{B}^{-1}(\eta) \jump{\omega} \eta \eta_x \rangle +\rho_+\mathcal{G}_-(\eta) \tilde{\xi} \langle \T{D}\mathcal{B}^{-1}(\eta)\dot{\eta},\jump{\omega}\eta\eta_x\rangle \;dx \\&\; +\int_{\mathbb{R}}\rho_+ \tilde{\xi} \mathcal{G}_-(\eta)\mathcal{B}^{-1}(\eta) \jump{\omega}(\dot{\eta}\eta_x+\eta\dot{\eta}_x)-\rho_+\rho_- \jump{\omega}(\dot{\eta}\eta_x+\eta\dot{\eta}_x)\mathcal{B}^{-1}(\eta) \jump{\omega}(\eta\eta_x)\; dx\\& \; - \dfrac{1}{2}\int_{\mathbb{R}}\rho_+ \rho_- \jump{\omega}\eta \eta_x \langle \T{D}\mathcal{B}^{-1}(\eta)\dot{\eta},\jump{\omega}\eta \eta_x \rangle + \left(\tilde{\xi} \omega_-(\dot{\eta}\eta_x +\eta \dot{\eta}_x)-\dfrac{\eta^2 \dot{\eta} \jump{\rho \omega^2}}{2}\right)\; dx,
 \end{split}
\end{equation}
and 
\begin{equation}\label{D xi E}
\T{D}_{\tilde{\xi}}E(u)\dot{u}=\int_{\mathbb{R}} \left( \mathcal{A}(\eta)\tilde{\xi}+ \omega_-\eta \eta_x \right)\dot{\tilde{\xi}}\; dx + \int_{\mathbb{R}} \rho_+ \dot{\tilde{\xi}} \mathcal{G}_-(\eta)\mathcal{B}^{-1}(\eta) \jump{\omega} \eta \eta_x \;dx.
\end{equation}
First of all, let us look at the expression in \eqref{D xi E}, it is easy to see that 
\begin{equation*}
    \mathcal{A}(\eta)\tilde{\xi},\; \omega_-\eta \eta_x,\; \mathcal{G}_-(\eta)\mathcal{B}^{-1}(\eta) \jump{\omega} \eta \eta_x \in L^2 (\mathbb{R}).
\end{equation*} Moreover, it is clear that the last integral in \eqref{D xi E} is an element of the dual space $\mathbb{X}^*$ acting on $\dot{u}$.

In \eqref{D eta E}, the first integral can be written as
\begin{equation*}
    \int_{\mathbb{R}} \tilde{\xi}\langle \T{D} \mathcal{A}(\eta)\dot{\eta},\tilde{\xi} \rangle\; dx=\sum_{\pm}\rho_{\pm} \int_{\mathbb{R}} a_1^{\pm}(\eta, \theta_{\pm}) \theta_{\pm}' \dot{\eta}\; dx +\sum_{\pm} \rho_{\pm} \int_{\mathbb{R}} (a_2^{\pm}(\eta_{\pm},\theta_{\pm}) \mathcal{A}(\eta) \tilde{\xi})\dot{\eta}\; dx.
\end{equation*}
Since $u\in \mathcal{O}\cap \mathbb{V}$, we obtain 
\begin{equation*}
    a_1^{\pm}(\eta, \theta_{\pm}),\; a_2^{\pm}(\eta_{\pm},\theta_{\pm}) \in L^2(\mathbb{R}), \qquad \theta_{\pm}\in H^1(\mathbb{R}),
\end{equation*}
where $\theta_{\pm}=\mathcal{A}(\eta) \mathcal{G}_{\pm}(\eta)^{-1} \tilde{\xi}$. Further, one can directly see that the second and last integral in \eqref{D eta E} is an element of the dual space $\mathbb{X}^*$ acting on $\dot{u}$.

Hence the extension $\nabla E(u)$ can be thought to have an $L^2$ gradient structure, that is, 
\begin{equation}
    \langle \nabla E(u), v \rangle_{{\mathbb{X}}^* \times \mathbb{X}}= (E'(u),v)_{L^2},
\end{equation}
where the $L^2$ gradient $E'(u)=(E_{\eta}'(u), E_{\tilde{\xi}}'(u))$ takes the form
\begin{equation}
    \begin{aligned}
        &E_{\eta}'(u):=\\&\dfrac{1}{2}\sum_{\pm}\rho_{\pm} a_1^{\pm}(\eta, \theta_{\pm}) \theta_{\pm}' +\sum_{\pm} \rho_{\pm}  (a_2^{\pm}(\eta_{\pm},\theta_{\pm}) \mathcal{A}(\eta) \tilde{\xi})
        -\left(g \jump{\rho}\eta +\sigma \left(\dfrac{\eta'}{\langle\eta'\rangle}\right)'\right) \\&
        +\rho_+\left(a_1^{-}(\eta, \mathcal{B}^{-1}(\eta) \jump{\omega} \eta \eta_x) \tilde{\xi}' +a_2^{-}(\eta,\mathcal{B}^{-1}(\eta) \jump{\omega} \eta \eta_x) \mathcal{G}_{-}(\eta) \tilde{\xi} \right)\\&
        - \rho_+ \sum_{\pm}\rho_{\pm}\left(a_1^{\mp}(\eta, \mathcal{B}^{-1}(\eta)\jump{\omega}\eta \eta_x)(\mathcal{B}^{-1}(\eta)\mathcal{G}_-(\eta)\tilde{\xi})_x \right. \\& \qquad  \left.
        \phantom{a_1^{\mp}(\eta, \mathcal{B}^{-1}(\eta)\jump{\omega}\eta \eta_x)(\mathcal{B}^{-1})}
         +\mathcal{B}^{-1}(\eta)\mathcal{G}_{\mp}(\eta)(a_2^{\mp}(\eta, \mathcal{B}^{-1}(\eta)\jump{\omega}\eta \eta_x))\mathcal{G}_-(\eta) \tilde{\xi} \right)\\& 
        -\rho_+\jump{\omega}\eta_x \left(\mathcal{B}^{-1}(\eta)\mathcal{G}_-(\eta) \tilde{\xi}\right)_x +\rho_+ \rho_-\jump{\omega} \eta \left(\mathcal{B}^{-1}(\eta) \jump{\omega} \eta \eta_x \right)_x \\&
        +\dfrac{1}{2}\rho_+\rho_-\sum_{\pm} \rho_{\pm} \left(a_1^{\mp}\left(\eta,\mathcal{B}^{-1}(\eta) \jump{\omega} \eta \eta_x\right) \left(\mathcal{B}^{-1}(\eta)\jump{\omega} \eta \eta_x\right)_x \right.\\&  \qquad \left.
     \phantom{a_1^{\mp}(\eta, \mathcal{B}^{-1}(\eta)\jump{\omega}\eta \eta_x)(\mathcal{B}^{-1})}
         + a_2^{\pm}\left(\eta,\mathcal{B}^{-1}(\eta) \jump{\omega} \eta \eta_x\right) \mathcal{B}^{-1}(\eta) \mathcal{G}_{\mp}(\eta) \jump{\omega} \eta \eta_x \right)\\& 
        -\omega_-\tilde{\xi}'\eta -\dfrac{\jump{\rho \omega^2} \eta^2}{2},\\ \\&
        E_{\tilde{\xi}}'(u):= \mathcal{A}(\eta) \tilde{\xi} +\omega_- \eta \eta_x +\rho_+ \mathcal{G}_-(\eta)\mathcal{B}^{-1}(\eta)\jump{\omega} \eta \eta_x. 
    \end{aligned}
\end{equation}
Hence, the proof is now complete.
\end{proof}

The function space $\mathbb{X}$ (energy space) will be equipped with a symplectic form via a Poisson map
\begin{equation}\label{Poisson Map}
    J= \begin{pmatrix}
    0 & 1 \\ -1 & -\sum_{\pm}\pm\rho_{\pm}\omega_{\pm}\partial_x^{-1}
    \end{pmatrix} : \T{Dom } J \subset\mathbb{X}^* \to \mathbb{X},
\end{equation}
 where
\begin{equation}\label{Domain J}
    \T{Dom } J:=(H^{-1}(\mathbb{R}) \cap \dot{H}^{1/2}(\mathbb{R})) \times (H^1(\mathbb{R}) \cap \dot{\mathbb{H}}^{-1/2}(\mathbb{R})).
\end{equation}

Observe that from the structure of $J$ in \eqref{Poisson Map}, the difference in regularity, and homogeneity of the the spaces in \eqref{Domain J}, one can conclude that $J$  is not a bijection. This fact clearly violates one of the assumption in the classical GSS approach \cite{GSS1987}. This is one of the reasons why we instead use the relaxed GSS \cite{VarholWahlenWalsh2020} as it only requires that $J$ to be an injection, as stated in Assumption \ref{assumption of poisson map}. Due to the presence of vorticity, the Poisson map is not canonical anymore. In other words, it is different with the Poisson map obtained in \cite{Ming--WalshOrbital2022}. When $\omega_+=0$ (one-layer fluid), we recover the map presented in \cite{Wahlen2007}.

\begin{lemma}[Poisson map ]
The Poisson map J \eqref{Poisson Map} satisfies Assumption \ref{assumption of poisson map} in the general theory.
\end{lemma}
\begin{proof}
The density of Domain $J$ \eqref{Domain J} in $\mathbb{X}^*$ is a direct consequence of Remark \ref{Dense space: intersectiom Hp amd dotHq}. Further, the injectiveness and skew-adjointness of $J$  follows easily from the definition of $J$ in \eqref{Poisson Map}.
\end{proof}
\begin{theorem}[Hamiltonian formulation]
Let $u=(\eta,\tilde{\xi}) \in  \mathcal{O} \cap \mathbb{W}$. Consider the following Hamiltonian system given by 
\begin{equation}\label{abstract Hamiltonian}
    \partial_tu= J \T{D}E(u), \quad u|_{t=0}=u_0,
\end{equation}
with $u_0 \in \mathcal{O} \cap \mathbb{W}$, where $E$ is given by \eqref{Hamiltonian} and $J$ is the Poisson map in \eqref{Poisson Map}. Then $u\in C^0([0,t_0); \mathcal{O}\cap\mathbb{W})$ is a (weak) solution to \eqref{abstract Hamiltonian} if and only if the profiles $(\eta, \phi_\pm)$ solve the governing equations \eqref{Euler-Incompressible in terms of Harmonic Conj}.
\end{theorem}
\begin{proof}
Suppose that $u(t)=(\eta(t),\tilde{\xi}(t))\in C^0([0,t_0); \mathcal{O} \cap \mathbb{W})$ is a weak solution to the abstract Hamiltonian \eqref{abstract hamiltonian}. Using \eqref{Equation for Xi}, we define 
\[\phi_\pm:= \mathcal{H}(\eta)\left(\mp \mathcal{B}(\eta)^{-1} \mathcal{G}_\mp(\eta)\tilde{\xi}+\rho_\mp \mathcal{B}(\eta)^{-1}\jump{\omega}\eta \eta_x\right).
\] Clearly, this satisfies the Laplace equation in \eqref{Euler-Incompressible in terms of Harmonic Conj-a} and the boundary conditions on the walls $\{y=\pm d_\pm\}$ by definition of $\mathcal{H}_{\pm}$ in \eqref{harmonic extension}. Further, recalling the expression for $E'(u)$ in Lemma \ref{Energy Extension}, we have
\begin{equation*}
    \eta_t=E_{\tilde{\xi}}'(u)=\mathcal{A}(\eta) \tilde{\xi} +\rho_+ \mathcal{G}_-(\eta)\mathcal{B}^{-1}(\eta)\jump{\omega} \eta \eta_x +\omega_- \eta \eta_x,
\end{equation*}
which holds in the distributional sense and it is equivalent to the kinematic boundary conditions in \eqref{kinematic and bernoulli}, therefore leads to the kinematic condition in \eqref{Euler-Incompressible in terms of Harmonic Conj-b}. 
Next, we claim that the Bernoulli condition is equivalent to
\begin{equation*}
   \tilde{\xi}_t= - E_{\eta}'(u)-\sum_{\pm}\pm\rho_{\pm}\omega_{\pm}\partial_x^{-1}E_{\tilde{\xi}}'(u).
\end{equation*}
We begin by writing the integrand in \eqref{Hamiltonian} in terms of $\xi_\pm$ instead of $\tilde{\xi}$:

\[
\begin{aligned}
&\dfrac{1}{2}\bigg[\tilde{\xi}\mathcal{A}(\eta) \tilde{\xi}+2\rho_+\jump{\omega}\eta \eta_x\mathcal{G}_-(\eta)\mathcal{B}^{-1} \tilde{\xi}-\rho_+\rho_-\jump{\omega}\eta \eta_x \mathcal{B}^{-1}\jump{\omega}\eta \eta_x\quad\\& -g \jump{\rho}\eta^2+2\tilde{\xi}\omega_-\eta \eta_x -\dfrac{\eta^3 \jump{\rho \omega^2}}{3}+2\sigma(\sqrt{1+\eta_x^2}-1)\bigg]\\&=\dfrac{1}{2}\bigg[\rho_-\xi_-\mathcal{G}_-(\eta)\xi_-+\rho_+\xi_+\mathcal{G}_+(\eta)\xi_+ -2\xi_+\rho_+\jump{\omega}\eta\eta_x\\&\qquad -g \jump{\rho}\eta^2+2\tilde{\xi}\omega_-\eta \eta_x -\dfrac{\eta^3 \jump{\rho \omega^2}}{3}+2\sigma(\sqrt{1+\eta_x^2}-1)\bigg]\\&=\dfrac{1}{2}\bigg[\rho_-\xi_-\mathcal{G}_-(\eta)\xi_-+\rho_+\xi_+\mathcal{G}_+(\eta)\xi_+ -2\jump{\rho \xi\omega}\eta\eta_x\\&\qquad-g \jump{\rho}\eta^2-\dfrac{\eta^3 \jump{\rho \omega^2}}{3}+2\sigma(\sqrt{1+\eta_x^2}-1)\bigg]
\end{aligned}.\]
Therefore, equivalently, the Hamiltonian \eqref{Hamiltonian} can be written as

\begin{equation}\label{alternative energy}
  \begin{aligned}
E(\xi_\pm,\eta)=\dfrac{1}{2}\int_{\mathbb{R}}&\rho_-\xi_-\mathcal{G}_-(\eta)\xi_-+\rho_+\xi_+\mathcal{G}_+(\eta)\xi_+ -2\jump{\rho \xi\omega}\eta\eta_x\\&-g \jump{\rho}\eta^2-\dfrac{\eta^3 \jump{\rho \omega^2}}{3}+2\sigma(\sqrt{1+\eta_x^2}-1)\; dx.
\end{aligned}  
\end{equation} 

Thanks to the derivative formula in the  Appendix~\ref{formulas appendix} for the operator $\mathcal{G}_\pm(\eta)$ with respect to $\eta$:
\begin{equation}\label{Frechet Deritavative G}
    \int_{\mathbb{R}}\xi_\pm\langle D\mathcal{G}_\pm(\eta)\dot{\eta}, \xi_\pm\rangle\; dx= \int_{\mathbb{R}}\dot{\eta}(\mp \Gamma_\pm(\eta,\xi_\pm))\;dx,
\end{equation} where $\Gamma_\pm$ is defined in \eqref{definition of Gamma}.
 Due to formula \eqref{Frechet Deritavative G}, it follows that the Bernoulli equation is satisfied.
\end{proof}

\subsection{The symmetry group and momentum}\label{symmetry group and momentum}
It is well known that the internal water wave problem is invariant under the horizontal translations. For this reason, we define a one-parameter symmetry group:
\[
T(s)u:=u(\cdot-s) \qquad \T{for } u\in \mathbb{X}.
\]
In addition to that, this invariance also gives rise to another conserved quantity known as the momentum, $P_{\pm}$ in each layer:
\[
P_{\pm}=\pm \int_{\mathbb{R}}\left(\rho_\pm \eta_x \xi_{\pm} + \dfrac{1}{2}\rho_{\pm}\omega_{\pm}\eta^2\right)\; dx.
\]

Summing both momentum in each layers leads to the total momentum equation given by
\begin{equation}\label{Total Momentum}
    P(\eta, \tilde{\xi})=-\int_{\mathbb{R}}\left(\eta_x \tilde{\xi} -\dfrac{1}{2}\jump{\rho \omega}\eta^2\right) \; dx.
\end{equation}
Observe that $P$ is a smooth functional in $\mathcal{O}\cap \mathbb{V}$.

The following lemma shows that $T$ and $P$ satisfy a number of properties as required by Assumption \ref{assumption on symmetry group}.
\begin{lemma}[Conserved quantities and symmetry] \label{conserved quantity}
The energy E, momentum P, and the translation symmetry group T given above satisfy Assumptions \ref{assumption on Derivative extension} and \ref{assumption on symmetry group}. Specifically, the infinitesimal generator of $T_{\mathbb{X}^k}$ is the unbounded linear operator
\begin{equation}\label{definition of T'(0)}
\begin{split}
    T'(0)|_{\mathbb{X}^k}&:\T{Dom} T'(0) \subset\mathbb{X}^k \to \mathbb{X}^k\\& u\mapsto -\partial_{x}u
\end{split}
\end{equation}
with dense domain $\T{Dom}T'(0)|_{\mathbb{X}^k}:=\mathbb{X}^{k+1}$, and 
\begin{equation}\label{T'(0)=J nabla P}
    T'(0)u=J\nabla P(u) \qquad \T{for all } u\in O \cap\T{Dom}T'(0).
\end{equation}
\end{lemma}
\begin{proof}
In light of Assumption \ref{assumption on Derivative extension}, we have shown that the energy has an extension as stated by Lemma \ref{Energy Extension}. Here, we will show that the momentum can also be extended. Let $u=(\eta,\tilde{\xi})\in \mathcal{O}\cap \mathbb{V}$ and $\dot{u}=(\dot{\eta},\dot{\tilde{\xi}}) \in \mathbb{V}$. Recalling the definition of $P$ in \eqref{Total Momentum} and computing its first variation yield
\begin{equation*}
    \T{D}P(u)\dot{u}=\int_{\mathbb{R}}\left(\tilde{\xi}' +\jump{\rho \omega}\eta\right) \dot{\eta} \;dx-\int_{\mathbb{R}}\eta_x \dot{\tilde{\xi}}\;dx=:\langle \nabla P(u),\dot{u}\rangle_{\mathbb{X}^* \times \mathbb{X}}.
\end{equation*}
The above expression has an $L^2$ gradient
\[
P'(u)=(P_\eta'(u),P_{\tilde{\xi}}'(u))=(\tilde{\xi}+\jump{\rho \omega}\eta, -\eta').
\]
Further, it is obvious that $\nabla P(u)\in \T{Dom } J$ for $u \in \mathcal{O}\cap \mathbb{V}.$ Observe that $\T{Dom } T'(0) =\mathbb{X}^{3/2} \subset \mathbb{X}$. The expression \eqref{T'(0)=J nabla P} follows easily from the definition of $J$ in \eqref{Poisson Map} and $T'(0)$ in \eqref{definition of T'(0)}.
\end{proof}
As it is mentioned previously, most of the requirements in Assumption \ref{assumption on symmetry group} can be shown to hold in a straightforward manner. Hence, we will omit the details here. However, we want to focus more on Assumption \ref{assumption on symmetry group}(8). First, recall that
\[
\T{Rng } J =\left(H^1(\mathbb{R}) \cap \dot{H}^{-1/2}(\mathbb{R})\right) \times \left(H^{-1} (\mathbb{R} \cap \dot{H}^{1/2}(\mathbb{R})\right).
\]
Therefore, by Lemma \ref{conserved quantity}, we obtain
\[
\T{Dom } T'(0)|_{\mathbb{W}} \cap \T{Rng }J=\left(H^{4+}(\mathbb{R}) \cap \dot{H}^{-1/2}(\mathbb{R}\right) \times \left(H^{-1}(\mathbb{R})\cap \dot{H}^{1/2}(\mathbb{R})\cap \dot{H}^{7/2+}(\mathbb{R})\right).
\]
Hence, by Remark \ref{Dense space: intersectiom Hp amd dotHq}, $\T{Dom } T'(0)|_{\mathbb{W}} \cap \T{Rng }J$ is dense in $\mathbb{X}$.

\subsection{Bound states}
The existence result in Theorem~\ref{existence theorem} was obtained by fixing the value of $\beta$ and letting the rest of the parameters vary. Unfortunately, when using the general theory, this choice is not ideal: essentially, one might study the stability of two waves that solve two different internal water waves problem. To avoid such degenerate and nonphysical scenario, instead, we require a family of solutions parameterized only by the variable $c$, known as bound states, while fixing the rest of the physical parameters. 

Let the parameters $\left(\rho_{\pm*},d_{\pm*}, \omega_{\pm *}, \sigma_*,c_*\right)$ be fixed. We define 
\begin{equation}\label{definition of betac,alphac,epsilonc}
    (\beta_c,\alpha_c):=\left(\dfrac{\sigma_*}{d_{+*}\rho_{-*}c^2}, -\dfrac{g\jump{\rho_*}d_{+*}}{\rho_{{-}_*}c^2}\right), \qquad \epsilon_c:=\sqrt{\alpha_c-\alpha_0} \quad \T{for } |c-c_*| \ll 1. 
\end{equation}
The pair $(\beta_c,\alpha_c)$ parameterizes a line segment joining the fixed point $(\beta_*,\alpha_*)$ to the origin in the $(\beta,\alpha)$ plane. Meanwhile, $\epsilon_c$ plays a role as a bifurcation parameter in terms of $c$. Throughout this week, $\epsilon_c$ will be kept sufficiently small.

\begin{corollary}[Bound states] \label{Bound States Corollary}
Fix $\left(\rho_{\pm*}\,d_{\pm*},\omega_{\pm *}, \sigma_*,c_*\right)$ such that 
\[\varrho_*-\dfrac{1}{d_*^2}+\dfrac{\omega_{+*}d_{+*}\varrho_*}{c_*}+\dfrac{\omega_{-*}d_{+*}}{c_*d_*}+\dfrac{\omega_{+*}^2d_{+*}^2\varrho_*}{3c_{*}^2}-\dfrac{\omega_{-*}^2d_{+*}^2}{3c_{*}^2}\neq 0\] and the fixed non-dimensional parameters $(\beta_*,\alpha_*)$ satisfies the condition $\beta_*>\beta_0$ and $\alpha_* =\alpha_0 +\epsilon^2.$ Then there exists an open interval $\mathcal{I}\ni c_*$ and a family of bound states $\{U_c\}_{c \in \mathcal{I}} \subset \mathcal{O}\cap \mathbb{W}$ having the non-dimensional parameter values $(\beta_c,\alpha_c)$. The free surface of the bound states is 
.\[
\eta_c:=\eta_{\epsilon_c, \beta_c} \quad \T{for } c\in \mathcal{I}. 
\]
Moreover, the family of bound states $\{U_c\}$ satisfies Assumption \ref{assumption on bound states}.
\end{corollary}
\begin{proof}
Let $\left(\rho_{\pm*}\,d_{\pm*},\omega_{\pm *}, \sigma_{\star},c_{\star}\right)$ be given. Suppose that $\beta_*>\beta_0$ and $0<\alpha_*-\alpha_0 \ll 1$. For any $c$ such that $0<c-c_*\ll 1$, then by Theorem~\ref{existence theorem}, the bound states can be taken to be $U_c:=u_{\epsilon_c,\beta_c}$ where the surface profile depends on the parameter $\beta_c,\alpha_c$. Observe that from the explicit expression of the profile in Theorem~\ref{existence theorem}, it is exponentially localized. Further, due to the translation invariance of the problem, the profile $\eta_c$ is of class $C^{\infty}$. Thus, $\eta_c \in \mathbb{X}_1^k$ for all $k\geq 1/2$. Similarly, via the kinematic condition, $\tilde{\xi}_c$ is also smooth and belongs to $\mathbb{X}_2^{k}$ for all $k\geq 1/2$. The first part of Assumption \ref{assumption on bound states} now follows. Assumption \ref{assumption on bound states}(3) also follows from the regularity considered here. From the knowledge on the expression of $P$  and the profile $\eta_c$, Assumption \ref{assumption on bound states}(2),(4) clearly hold.
\end{proof}

\section{Spectral Analysis} \label{Spectral Analysis}
If $u(t)=T(ct)U$ is a traveling wave solutions for any bound state solution $U \in \mathcal{O} \cap \mathbb{W}$ with a wave speed $c\in \mathbb{R}$, then by the Hamiltonian structure, Lemma \ref{Energy Extension}, and Assumption \ref{assumption on bound states}(6), we have
\begin{equation}\label{cT'(0)U=JDE(U)}
    \dfrac{du}{dt}=cT'(0)U=J\T{D}E(U).
\end{equation}
Furthermore, recall that via \eqref{T'(0)=J nabla P},
the infinetesimal generator of $T$ satisfies the following relation
\begin{equation*}
    T'(0)(u)=J\nabla P(u), 
\end{equation*}
where the operator $T'(0)$ maps $u\mapsto-\partial_x u$. In concert with \eqref{cT'(0)U=JDE(U)}, they yield
\begin{equation*}
    \dfrac{du}{dt}=cT'(0)U=\T{D}E(U)=c\T{D}P(U)=\T{D}E(U).
\end{equation*}
The above equation leads us to define the following functional known as the \textit{augmented Hamiltonian} for a fixed speed $c$:
\begin{equation*}
E_c(u)=E(u)-cP(u).
\end{equation*}
Let $u_*=(\eta_*, \tilde{\xi}_*)$ be the critical point of the functional $E_c$, then $\T{D}_{\tilde{\xi}}E(u_*)=\T{D}_{\tilde{\xi}}P(u_*)$. From the Kinematic conditions in \eqref{kinematic dynamic} and \eqref{Kinematic and dynamic condition in terms of Gamma and xi}, for traveling waves, we have useful expressions for $\xi_{{\pm}_*}$ and $\tilde{\xi}_*$

\[
\xi_{{\pm}_*}=\pm c \mathcal{G}_{\pm}^{-1}(\eta)\eta_x \pm \omega_{\pm} \mathcal{G}_{\pm}^{-1}(\eta)\eta\eta_x=\pm\mathcal{G}_{\pm}^{-1}(\eta)(c\eta_x+\omega_{\pm}\eta\eta_x),
\]
\[
\tilde{\xi}_*=-c\mathcal{A}^{-1}(\eta)\eta_x-\rho_+\jump{\omega}\mathcal{A}^{-1}(\eta)\mathcal{G}_-(\eta)\mathcal{B}^{-1}(\eta)\eta \eta_x-\omega_-\mathcal{A}^{-1}(\eta)\eta\eta_x.
\]

Inspired by the notation used in \cite{Ming--WalshOrbital2022}, we define the $a_{1}^{\pm}(\eta,\xi)$ and $a_{2}^{\pm}(\eta,\xi)$ as follows:
\[a_{1}^{\pm}(\eta,\phi):=\mp(\partial_x\mathcal{H}_{\pm}(\eta)\phi)|_{y=\eta}\qquad a_{2}^{\pm}(\eta,\phi):=-(\partial_y\mathcal{H}_{\pm}(\eta)\phi)|_{y=\eta}.\] Note that these two quantities represent the horizontal and vertical velocities respectively when $\phi$ is being replaced with $\xi_{\pm}$.
Further, in relations to $a_{1}^{\pm}$ and $a_{2}^{\pm}$, we define the following functions which represent the relative velocities in horizontal and vertical directions:
\begin{equation}\label{ b_1 and b_2}
    b_{1}^{\pm}:=\mp a_{1}^{\pm}(\eta,\xi_{\pm})-c-\omega_{\pm} \eta, \quad b_{2}^{\pm}:= - a_{2}^{\pm}(\eta,\xi_{\pm}).
\end{equation}
For traveling waves solutions, the Kinematic condition can now be recast as 
\[b_2^{\pm}=\left(b_1^{\pm} \right)\eta_x. \]
Via some standard computation (in the notes), we obtain
\begin{equation*}
     \T{D}\xi_{\pm}(\eta)\dot{\eta}
  =\mp \mathcal{G}_{\pm}(\eta)^{-1}(b_{1}^{\pm} \dot{\eta})_x + b_{2}^{\pm} \dot{\eta} 
\end{equation*}
Using the definition of $\tilde{\xi}$ and the formula for $\T{D}\xi_{\pm}(\eta)$, we can infer
\begin{equation}\label{S and T}
   \T{D}\tilde{\xi}_*(\eta)\dot{\eta}=\sum_{\pm}\rho_{\pm}\mathcal{G}_{\pm}^{-1}(\eta)\left(b_{1}^{\pm}\dot{\eta}\right)_x -\sum_{\pm} \pm \rho_{\pm}b_{2}^{\pm} \dot{\eta}=:\mathcal{S}\dot{\eta}-\mathcal{T}\dot{\eta}. 
\end{equation}

Let us now define a smooth functional known as the \textit{augmented potential} $\mathcal{V}_c^{\text{aug}}$ as follows:
\[
\mathcal{V}_c^{\text{aug}}:=E_c(\eta,\tilde{\xi}_*(\eta))=\min_{\tilde{\xi}}E_c(\eta, \tilde{\xi}).
\]
In the rest of this article, we shall compute the spectrum of $\T{D}^2\mathcal{V}_c^{\text{aug}}$, which will determine the spectrum of $\T{D}^2E_c$.

\begin{lemma}[Second derivative of $\mathcal{V}_c^{\text{aug}}$]\label{second derivative of Aug Potential}
For all $(\eta, \tilde{\xi}_*(\eta)) \in \mathcal{O} \cap \mathbb{V}$ and $\dot{\eta}\in H^{3/2 +}$, we have the following formula
\begin{equation}\label{second derivative of augmented potential}
    \T{D}^2\mathcal{V}_c^{\textnormal{aug}}(\eta)[\dot{\eta},\dot{\eta}]=\T{D}^2_{\eta}E_c(\eta,\tilde{\xi}_*)[\dot{\eta},\dot{\eta}]-\int_{\mathbb{R}} (\mathcal{S}-\mathcal{T})\dot{\eta}\mathcal{A}(\eta)(\mathcal{S}-\mathcal{T})\dot{\eta}\; dx.
\end{equation}
\end{lemma}

\begin{proof}
We begin by differentiating in the direction of $\dot{\eta}$,
\begin{equation}
    \T{D}\mathcal{V}_c^{\text{aug}}(\eta)\dot{\eta}=\T{D}_\eta E_c(u_*)\dot{\eta}+ \T{D}_{\tilde{\xi}}E_c(u_*)\T{D}{\tilde{\xi}_*}(\eta)\dot{\eta}=D_\eta E_c(u_*)\dot{\eta},
\end{equation}
where $u_*=(\eta, \tilde{\xi_*})$. Observe that the second term in the summation above has vanished due to its evaluation at $u_*=(\eta,\tilde{\xi}_*(\eta))$ which is a critical point of $E_c$. Differentiating again in the direction of $\dot{\eta}$ gives us
\begin{align*}
\T{D}^2\mathcal{V}_c^{\text{aug}}(\eta)[\dot{\eta},\dot{\eta}]&=\T{D}^2_{\eta}E_c(\eta)[\dot{\eta},\dot{\eta}]+\T{D}_{\tilde{\xi}}\T{D}_{\eta}E_c(u_*)[\T{D}\tilde{\xi}_*(\eta)\dot{\eta},\dot{\eta}]\\&=\T{D}^2_{\eta}E_c(\eta)[\dot{\eta},\dot{\eta}]-\T{D}^2_{\tilde{\xi}}E_c(\eta)[\T{D}\tilde{\xi}_*(\eta)\dot{\eta},\T{D}\tilde{\xi}_*(\eta)\dot{\eta}].
\end{align*}
From the definition of $E_c$ and the fact that the momentum is linear in $\tilde{\xi}$, we obtain
\begin{equation}\label{second derivative of Hamiltonian in the direction of D}
    \T{D}^2_{\tilde{\xi}}E_c(\eta)[\T{D}\tilde{\xi}_*(\eta)\dot{\eta},\T{D}\tilde{\xi}_*(\eta)\dot{\eta}]=\int_{\mathbb{R}} \T{D}\tilde{\xi}_*(\eta)\dot{\eta}\mathcal{A}(\eta)\T{D}\tilde{\xi}_*(\eta)\dot{\eta}\; dx.
\end{equation}
Combining \eqref{second derivative of Hamiltonian in the direction of D} and \eqref{S and T} leads us to the formula \eqref{second derivative of augmented potential}.
\end{proof}

\begin{lemma}[Quadratic form]\label{quadratic form}
For all $(\eta, \tilde{\xi}_*(\eta))\in \mathcal{O} \cap \mathbb{V}$ and $c\in \mathbb{R}$, there exists a self-adjoint operator $\mathcal{Q}_c(\eta) \in \text{Lin}(\mathbb{X}_1; \mathbb{X}_1^*)$ such that
\begin{equation}
    \T{D}^2\mathcal{V}_c^{\textnormal{aug}}(\eta)[\dot{\eta},\dot{\zeta}]=\langle\mathcal{Q}_c(\eta)\dot{\eta},\dot{\zeta}\rangle_{\mathbb{X}_1^* \times \mathbb{X}_1},
\end{equation}
for all $\dot{\eta}, \dot{\zeta} \in \mathbb{V}_1$
and 
\begin{equation}\label{Operator Q_c}
    \begin{split}
        \mathcal{Q}_c(\eta)\dot{\eta}&=-\left(\sigma \dfrac{\dot{\eta}_x}{\langle \eta_x \rangle^3}\right)'-\left(g\jump{\rho} + \sum_{\pm}\rho_{\pm} \left(\pm b_1^{\pm}(b_2^{\pm})'\mp \omega_{\pm}\eta (b_2^{\pm})'-a_1^{\pm}(\eta,\Upsilon_{\mp})(b_2^{\pm})'\right)\right)\dot{\eta} \\&\qquad - \left(\jump{\rho \xi_x \omega} - \eta \jump{\rho \omega^2}\right)\dot{\eta}\\&\qquad  - c \jump{\rho \omega} \dot{\eta} +\sum_{\pm} \rho_{\pm} b_1^{\pm}\left(\mathcal{G}_{\pm}(\eta)^{-1}\left(\left(b_1^{\pm} \right)\dot{\eta}\right)'\right)'+2 \rho_+\rho_-\jump{\omega}\dot{\eta}'\mathcal{B}(\eta)^{-1} \jump{\omega} \eta \eta_x \\&\qquad +2 \rho_+ \rho_- \jump{\omega} \eta \partial_x\left(\mathcal{B}(\eta)^{-1} \jump{\omega} (\dot{\eta} \eta)_x\right) \\&\qquad- 4 \rho_-\rho_+ \sum_{\pm} \left(a_1^{\mp}(\eta, \Upsilon_\pm)\mathcal{B}^{-1}\jump{\omega}(\eta \dot{\eta})_x\right) \\& \qquad - 2 \rho_+ \rho_-\sum_{\pm} a_1^{\mp}(\eta,\Upsilon_{\pm}) \partial_x\left(\mathcal{B}^{-1}(\eta) \left(\sum_{\pm} a_1^{\mp}(\eta,\Upsilon_{\pm})\dot{\eta}\right)'\right).
    \end{split}
\end{equation}
\end{lemma}
\begin{proof}
First, it is clear that the Hamiltonian can alternatively be written in the following way
\begin{equation} \label{new expression of Hamiltonian}
\begin{split}
     E(\eta, \tilde{\xi})&=\dfrac{1}{2}\int_{\mathbb{R}}\tilde{\xi}\mathcal{A}(\eta) \tilde{\xi}+\rho_-\rho_+ \jump{\omega} \eta \eta_x \mathcal{B}^{-1}(\eta)\jump{\omega} \eta \eta_x-2\rho_+\xi_+ \jump{\omega}\eta \eta_x+2\tilde{\xi}\omega_-\eta \eta_x\\&\quad -g \jump{\rho}\eta^2 -\dfrac{\eta^3 \jump{\rho \omega^2}}{3}+2\sigma(\sqrt{1+\eta_x^2}-1) \; dx.
\end{split}
\end{equation}
From \ref{second derivative of augmented potential}, it is therefore useful to start doing an expansion on $\T{D}^2_{\eta}E_c(\eta,\tilde{\xi}_*)[\dot{\eta},\dot{\eta}]$. Using the expression of the Hamiltonian \eqref{new expression of Hamiltonian}, we can see that 
\begin{equation}\label{second derivative of hamiltonian}
\begin{split}
    \T{D}^2_{\eta}E_c(\eta,\tilde{\xi}_*)[\dot{\eta},\dot{\eta}]&=\dfrac{1}{2}\int_{\mathbb{R}}\tilde{\xi}_*\langle \T{D}^2\mathcal{A}(\eta)[\dot{\eta},\dot{\eta}] \tilde{\xi}_*\rangle \; dx-\int_{\mathbb{R}}g \jump{\rho}\dot\eta^2 \; dx +\int_{\mathbb{R}}\left(\sigma \dfrac{\dot{\eta}_x}{\langle\eta_x\rangle^3}\right) \; dx\\& \quad -\int_{\mathbb{R}}c\jump{\rho\omega}\dot{\eta}^2 \; dx - \int_{\mathbb{R}} \jump{\omega \xi_x \rho} \dot{\eta}^2 \; dx- \int_{\mathbb{R}} \jump{\rho \omega^2} \eta \dot{\eta}^2 \; dx \\& \quad -2 \int_{\mathbb{R}} \rho_- \rho_+ \jump{\omega}\partial_x\left(B^{-1}(\eta)\jump{\omega}\eta \eta_x\right) \dot{\eta}^2 \; dx\\& \quad -2 \int_{\mathbb{R}} \rho_- \rho_+ \jump{\omega} \eta \partial_x\left(B^{-1}(\eta) \jump{\omega} (\dot{\eta}\eta)_x \right) \dot{\eta} \; dx\\&\quad -4\int_{\mathbb{R}} \rho_- \rho_+ \jump{\omega} \eta \partial_x \langle \T{D}B^{-1}(\eta) \dot{\eta},\jump{\omega} \eta \eta_x \rangle \dot{\eta} \;dx\\&\quad + \int_{\mathbb{R}} \rho_- \rho_+ \jump{\omega} \eta \eta_x \langle \T{D}^2B^{-1}(\eta)[\dot{\eta},\dot{\eta}],\jump{\omega} \eta \eta_x \rangle \; dx.
\end{split}
\end{equation}
It is important to note that when we let $\omega_{\pm}=0$, we recover back the expression found in \cite{Ming--WalshOrbital2022}. To arrive at the expression stated in \eqref{Operator Q_c}, we first begin by looking at the term involving $\tilde{\xi_*}\langle \T{D}^2\mathcal{A}(\eta)[\dot{\eta},\dot{\eta}] \tilde{\xi}_*\rangle$ in \eqref{second derivative of hamiltonian}. 
Since we are going to use some formula in \cite{Ming--WalshOrbital2022}, let us define the following variable
\[\theta_{\pm}(u_*)=\mathcal{G}_{\pm}(\eta)^{-1}\mathcal{A}(\eta)\tilde{\xi}=\mp\xi_{\pm}\pm \rho_{\mp}\mathcal{B}(\eta)^{-1} \jump{\omega} \eta \eta_x=:\mp \xi_{\pm}\pm \Upsilon_{\mp}.\] 
This implies 
\[a_1^{\pm}(\eta,\theta_{\pm})=b_1^{\pm}+c+\omega_{\pm}\eta + a_1^{\pm}(\eta,\pm \Upsilon_{\mp}), \quad a_2^{\pm}(\eta,\theta_{\pm})=\pm b_2^{\pm}.\] We further define the following expressions for later use:
\[\mathcal{S}_{\pm}(\eta)\zeta:=\mathcal{G}_{\pm}(\eta)^{-1}(b_1^{\pm}\zeta)', \qquad \mathcal{T}_{\pm}(\eta)\zeta:= \pm b_2^{\pm} \zeta.\]

Upon using the formula in \cite{Ming--WalshOrbital2022}, we can infer that
\begin{equation}\label{formula second derivative A}
    \begin{split}
        \dfrac{1}{2}\int_{\mathbb{R}}\tilde{\xi}_*\langle \T{D}^2\mathcal{A}(\eta)[\dot{\eta},\dot{\eta}] \tilde{\xi}_*\rangle \; dx&= \sum_{\pm}\mp\rho_{\pm} \int_{\mathbb{R}}\Biggl[\left((b_1^{\pm})'b_2^{\pm}+\omega_{\pm}\eta'b_2^{\pm}+ (a_1^{\pm}(\eta,\pm\Upsilon_{\mp}))'b_2^{\pm}\right)\dot{\eta}^2 \Biggr.\\ \Biggl. 
        &\phantom{\sum_{\pm}\mp\rho_{\pm}\int_{\mathbb{R}}    \left((b_1^{\pm})'b_2^{\pm}+\omega_{\pm}\eta'b_2^{\pm}+ (a_1^{\pm})\right)}   
         +\mathcal{T}\dot{\eta}\mathcal{G}_{\pm}(\eta)\mathcal{T}\dot{\eta}\Biggr] \; dx \\& \qquad +\int_{\mathbb{R}} (-\dot{\eta} \mathcal{M}(u_*)\dot{\eta} +\dot{\eta} \mathcal{N}(u_*) \dot{\eta}) \; dx.
    \end{split}
\end{equation}
To compute the terms on the second row in \eqref{formula second derivative A}, we would need to define the following expressions:
\begin{equation}\label{operator L}
    \begin{split}
    &\mathcal{L}_{\pm}(u_*) \dot{\eta}=\mathcal{T}\dot{\eta}-\mathcal{S}\dot{\eta}-\mathcal{G}_{\pm}(\eta)^{-1} \partial_x(c\dot{\eta}+\omega_{\pm}\eta \dot{\eta}+ a_1^{\pm}(\eta,\pm\Upsilon_{\mp})\dot{\eta})\\&
    \mathcal{L}(u_*) \dot{\eta}=\mathcal{T}\dot{\eta}-\mathcal{S}\dot{\eta}-\mathcal{A}(\eta)^{-1}
\partial_x (c\dot{\eta}+\omega_{\pm}\eta \dot{\eta} + a_1^{\pm}(\eta,\pm\Upsilon_{\mp})\dot{\eta}).    \end{split}
\end{equation}
Furthermore, 
\begin{align*}
        \int_{\mathbb{R}} \dot{\eta} \mathcal{M}\dot{\eta} \; dx&= \sum_{\pm} \rho_{\pm} \int_{\mathbb{R}} \left((b_1^{\pm}+c+\omega_{\pm}\eta + a_1^{\pm}(\eta,\pm\Upsilon_{\mp}))(\mathcal{L}_{\pm} \dot{\eta})' \pm b_2^{\pm}\mathcal{G}_{\pm}(\eta)\mathcal{L}_{\pm} \dot{\eta} \right)\dot{\eta} \; dx\\&= \sum_{\pm} \rho_{\pm} \int_{\mathbb{R}} \mathcal{L}_{\pm}\dot{\eta}\mathcal{G}_{\pm}(\eta)\mathcal{L}_{\pm}\dot{\eta} \;dx.
\end{align*}
Via the definition of $\mathcal{L}_{\pm}$ in \eqref{operator L}, we can infer
\begin{align*}
        \int_{\mathbb{R}} \dot{\eta} \mathcal{M}\dot{\eta} \; dx&= \sum_{\pm} \rho_{\pm} \int_{\mathbb{R}} \left(\mathcal{S}_{\pm}\dot{\eta}\mathcal{G}_{\pm}(\eta)\mathcal{S}_{\pm}\dot{\eta}-2\mathcal{S}_{\pm}\dot{\eta}\mathcal{G}_{\pm}(\eta)\mathcal{T}_{\pm}\dot{\eta}+\mathcal{T}_{\pm}\dot{\eta}\mathcal{G}_{\pm}(\eta)\mathcal{T}_{\pm}\dot{\eta}\right) \;dx \\& + \int_{\mathbb{R}} \left(\partial_x (c\dot{\eta}+\omega_{\pm}\eta \dot{\eta}+ a_1^{\pm}(\eta,\pm\Upsilon_{\mp})\dot{\eta})\mathcal{A}(\eta)^{-1}\partial_x (c\dot{\eta}+\omega_{\pm}\eta \dot{\eta}+ a_1^{\pm}(\eta,\pm\Upsilon_{\mp})\dot{\eta})\right) \; dx \\&+\int_{\mathcal{R}} 2\partial_x (c\dot{\eta}+\omega_{\pm}\eta \dot{\eta}+ a_1^{\pm}(\eta,\pm\Upsilon_{\mp})\dot{\eta}) (\mathcal{S}-\mathcal{T})\dot{\eta} \;dx.
\end{align*}
For the importance of simplification later, we display the following formula
\begin{align*}
        \int_{\mathbb{R}} \mathcal{S}_{\pm}\dot{\eta}\mathcal{G}_{\pm}(\eta)\mathcal{T}_{\pm}\dot{\eta} \; dx&=\pm \int_{\mathbb{R}} \mathcal{G}_{\pm}(\eta)^{-1}(b_1^{\pm}\dot{\eta})'\mathcal{G}_{\pm}(\eta)(b_2^{\pm} \dot{\eta}) \; dx\\&=\pm \int_{\mathbb{R}} (b_1^{\pm}\dot{\eta})'(b_2^{\pm} \dot{\eta}) \; dx\\&=\pm \dfrac{1}{2} \int_{\mathbb{R}} \left((b_1^{\pm})'(b_2^{\pm})- (b_1^{\pm})(b_2^{\pm})' \right)\dot{\eta}^2 \; dx.
\end{align*}
At last, following the formula in Appendix~\ref{formulas appendix}, we obtain
\begin{align*}
        \int_{\mathbb{R}} \dot{\eta} \mathcal{N}\dot{\eta} \; dx&= \int_{\mathbb{R}} (\mathcal{T}\dot{\eta}-\mathcal{S}\dot{\eta} -\mathcal{A}(\eta)^{-1}\partial_x (c\dot{\eta}+\omega_{\pm}\eta \dot{\eta}+ a_1^{\pm}(\eta,\pm\Upsilon_{\mp})\dot{\eta}))\mathcal{A}(\eta)\\&\qquad(\mathcal{T}\dot{\eta}-\mathcal{S}\dot{\eta} -\mathcal{A}(\eta)^{-1}\partial_x (c\dot{\eta}+\omega_{\pm}\eta \dot{\eta}+ a_1^{\pm}(\eta,\pm\Upsilon_{\mp})\dot{\eta})) \; dx\\&=\int_{\mathbb{R}}(D\tilde{\xi}_* (\eta)\dot{\eta} \mathcal{A}(\eta)D\tilde{\xi}_* (\eta)\dot{\eta}+2\partial_x (c\dot{\eta}+\omega_{\pm}\eta \dot{\eta}+ a_1^{\pm}(\eta,\pm\Upsilon_{\mp})\dot{\eta}) (\mathcal{S}-\mathcal{T})\dot{\eta}\\& \quad + \left(\partial_x (c\dot{\eta}+\omega_{\pm}\eta \dot{\eta}\pm a_1^{\pm}(\eta,\Upsilon_{\mp})\dot{\eta})\mathcal{A}(\eta)^{-1}\partial_x (c\dot{\eta}+\omega_{\pm}\eta \dot{\eta}+ a_1^{\pm}(\eta,\pm\Upsilon_{\mp})\dot{\eta})\right)\; dx.
\end{align*}
Combining together all the computations above and the formula in \eqref{second derivative of augmented potential}, we get
\begin{equation}\label{Second derivative Vaug}
    \begin{split}
         &\T{D}^2\mathcal{V}_c^{\textnormal{aug}}(\eta)[\dot{\eta},\dot{\eta}]\\&=\T{D}^2_{\eta}E_c(\eta,\tilde{\xi}_*)[\dot{\eta},\dot{\eta}]-\int_{\mathbb{R}} (\mathcal{S}-\mathcal{T})\dot{\eta}\mathcal{A}(\eta)(\mathcal{S}-\mathcal{T})\dot{\eta}\; dx\\&=
        \int_{\mathbb{R}} \left(\sigma \dfrac{(\dot{\eta}_x)^2}{\langle\eta_x\rangle^3}-\left(g\jump{\rho}+\sum_{\pm}\rho_{\pm}\left(\pm b^{\pm}_1 (b^{\pm}_2)'\mp \omega_\pm \eta (b^{\pm}_2)'\mp a^{\pm}_1(\eta,\pm\Upsilon_{\mp}(b_2^{\pm})'\right)\right) \dot{\eta}^2\right.\\& \left. \qquad -\left(\jump{\rho \xi_x \omega}-\eta \jump{\rho \omega^2}-c\jump{\rho \omega}\right)\dot{\eta}^2-\sum_{\pm} \rho_{\pm}  \mathcal{S}_{\pm}\dot{\eta}\mathcal{G}_{\pm}(\eta)\mathcal{S}_{\pm}\dot{\eta}\right) \; dx\\& \qquad
        -2 \int_{\mathbb{R}} \rho_- \rho_+ \jump{\omega}\partial_x\left(B^{-1}(\eta)\jump{\omega}\eta \eta_x\right) \dot{\eta}^2 \; dx\\& \qquad -2 \int_{\mathbb{R}} \rho_- \rho_+ \jump{\omega} \eta \partial_x\left(B^{-1}(\eta) \jump{\omega} (\dot{\eta}\eta)_x \right) \dot{\eta} \; dx\\&\qquad -4\int_{\mathbb{R}} \rho_- \rho_+ \jump{\omega} \eta \partial_x \langle \T{D}B^{-1}(\eta) \dot{\eta},\jump{\omega} \eta \eta_x \rangle \dot{\eta} \;dx\\&\qquad + \int_{\mathbb{R}} \rho_- \rho_+ \jump{\omega} \eta \eta_x \langle \T{D}^2B^{-1}(\eta)[\dot{\eta},\dot{\eta}],\jump{\omega} \eta \eta_x \rangle \; dx.
    \end{split}
\end{equation}

To compute the last two integrals in \eqref{Second derivative Vaug}, we need to derive the formula for 
\[\langle \T{D}\mathcal{B}^{-1}(\eta)\dot{\eta}\jump{\omega}\eta \eta_x \rangle \; \text{and } \langle \T{D}^2B^{-1}(\eta)[\dot{\eta},\dot{\eta}],\jump{\omega} \eta \eta_x \rangle.\] 
We begin with the expansion of the formula below
\begin{equation} 
\begin{aligned}\label{derivative of B}
\langle \T{D}\mathcal{B}(\eta)\dot{\eta},\tilde{\xi}\rangle&=\rho_+ \langle \T{D}\mathcal{G}_-(\eta)\dot{\eta},\tilde{\xi} \rangle + \rho_- \langle \T{D}\mathcal{G}_+(\eta)\dot{\eta}, \tilde{\xi}\rangle\\&=\rho_+\left(-\partial_x(a_1^-(\eta,\tilde{\xi})\dot{\eta}) + \mathcal{G}_-(\eta) a_2^-(\eta,\tilde{\xi})\dot{\eta} \right)\\&\qquad + \rho_-\left(-\partial_x(a_1^+(\eta,\tilde{\xi})\dot{\eta}) + \mathcal{G}_+(\eta) a_2^+(\eta,\tilde{\xi}) \dot{\eta}\right).
\end{aligned}
\end{equation}
Using the equation
\[
\langle \T{D}\mathcal{B}^{-1}(\eta)\dot{\eta},\tilde{\xi}\rangle=-\mathcal{B}^{-1}\langle \T{D}\mathcal{B}(\eta)\dot{\eta},\mathcal{B}^{-1}\tilde{\xi}\rangle
\] 
in combination with \eqref{derivative of B}, we derive the representation formula for the Fr\'echet derivative of $\mathcal{B}^{-1}(\eta)$ (for any given $\zeta$):

\begin{equation}
\begin{aligned}\label{derivative formula for inverse of B}
\int_{\mathbb{R}}
\zeta \langle \T{D}\mathcal{B}^{-1}(\eta)\dot{\eta},\tilde{\xi}\rangle\; dx=&-\int_{\mathbb{R}}\mathcal{B}^{-1}\bigg(\rho_+\left(-\partial_x(a_1^-(\eta,\mathcal{B}^{-1}\tilde{\xi})\dot{\eta}) \zeta + a_2^-(\eta,\mathcal{B}^{-1}\tilde{\xi})\dot{\eta}\mathcal{G}_- \zeta \right)\bigg)\;dx\\& -\int_{\mathbb{R}}\mathcal{B}^{-1}\bigg(\rho_-\left(-\partial_x(a_1^+(\eta,\mathcal{B}^{-1}\tilde{\xi})\dot{\eta})\zeta +  a_2^+(\eta,\mathcal{B}^{-1}\tilde{\xi})\mathcal{G}_+(\eta)\zeta \right)\bigg)\;dx\\=& -\sum_{\pm}\rho_{\pm}\int_{\mathbb{R}}\left( a_1^{\mp}(\eta,\mathcal{B}^{-1}\tilde{\xi})(\mathcal{B}^{-1}\zeta)_x \right)\dot{\eta}\\&
\phantom{-\sum_{\pm}\rho_{\pm}\int_{\mathbb{R}}\left( a_1^{\mp}(\eta,\mathcal{B}^{-1}\tilde{\xi})(\mathcal{B}) \right)}
 + \left( a_2^{\mp}(\eta, \mathcal{B}^{-1}\tilde{\xi}) \mathcal{B}^{-1}\mathcal{G}_{\mp}(\eta)\zeta \right) \dot{\eta}\; dx.
\end{aligned}
\end{equation}
From \eqref{derivative formula for inverse of B}, we arrive the  expression in the second row from the bottom of equation \eqref{Operator Q_c}.

Finally, it remains to show that the last integral in \eqref{Second derivative Vaug} yields the last expression in \eqref{Operator Q_c}. First, recall that $\mathcal{B}(\eta):=\sum_{\pm}\rho_{\pm} \mathcal{G}_{\mp}(\eta)$. Exploiting the second derivative formula for $\mathcal{G}_{\pm}(\eta)$, we obtain
\begin{equation}\label{Second derivative of B}
    \begin{aligned}
        \int_{\mathbb{R}} \tilde{\xi} \langle \T{D}^2 \mathcal{B}(\eta)[\dot{\eta},\dot{\eta}], \tilde{\xi} \rangle\; dx&= \sum_{\pm}\rho_{\pm} \int_{\mathbb{R}} \tilde{\xi}\langle \T{D}^2 \mathcal{G}_{\mp} (\eta)[\dot{\eta},\dot{\eta}],\tilde{\xi} \rangle\; dx\\&=\sum_{\pm}\rho_{\pm} \int_{\mathbb{R}} a_4^{\mp}(\eta,\tilde{\xi}) \dot{\eta}^2 + 2 a_2^{\mp} (\eta,\tilde{\xi})\dot{\eta}\mathcal{G}_{\mp}(\eta)(a_2^{\mp}(\eta,\tilde{\xi})\dot{\eta}).
    \end{aligned}
\end{equation}
Additionally, we have the following identity
\[
\T{D}^2\mathcal{B}(\eta)[\dot{\eta},\dot{\eta}]=-\mathcal{B}(\eta)\T{D}^2\mathcal{B}^{-1}(\eta)[\dot{\eta},\dot{\eta}]\mathcal{B}(\eta) +2 \T{D}\mathcal{B}(\eta)[\dot{\eta}]\mathcal{B}^{-1}(\eta) \T{D}\mathcal{B}(\eta)[\dot{\eta}].
\]
After rearranging terms and applying \eqref{Second derivative of B}, we obtain
\begin{align*}
    &\int_{\mathbb{R}} \rho_- \rho_+ \jump{\omega} \eta \eta_x \langle \T{D}^2B^{-1}(\eta)[\dot{\eta},\dot{\eta}],\jump{\omega} \eta \eta_x \rangle \; dx\\&= \int_{\mathbb{R}}2 \rho_+ \rho_-\left(\sum_{\pm} a_1^{-\mp}(\eta,\Upsilon_{\pm})\dot{\eta}\right)'\mathcal{B}^{-1}(\eta)\left(\sum_{\pm} a_1^{-\mp}(\eta,\Upsilon_{\pm})\dot{\eta}\right)'.
\end{align*}
Combining this with the statements above, yields the claimed formula for $\mathcal{Q}_c(\eta)$ and completes the proof.
\end{proof}

Based on the formula derived above, we can now determine the continuous spectrum of the linearized operator.  

\begin{lemma}[Continuous spectrum]
Let $u=(\eta,\tilde{\xi}) \in \mathcal{O} \cap \mathbb{V}$ be given. Then the operator 
$\mathcal{Q}_c(\eta)$ is a self-adjoint operator on $L^2(\mathbb{R})$ with domain $H^2(\mathbb{R})$. Further, the spectrum of this operator is equal to the one of $\mathcal{Q}_c(0)$, that is $[\tau_*,+\infty]$, where
 \begin{equation*} 
  \tau_*= \begin{cases}
    -g \jump{\rho}\left(1-\dfrac{\alpha_0}{\alpha}\right) \quad \T{for } \beta \geq \beta_0,\\
    -g \jump{\rho}\Bigg[1-\dfrac{1}{\alpha}\max_{\xi \in \mathbb{R}}\left(\sum_{\pm} \dfrac{\rho_{\pm}}{\rho_-}d_+\xi \coth(d_{\pm}\xi)-\beta d_+^2\xi^2+\left(\dfrac{ \omega_+d_+\varrho}{c}-\dfrac{\omega_-d_+}{c}\right)\right)\Bigg]\\ \quad \T{for } \beta < \beta_0.
   \end{cases}  
   \end{equation*}
\end{lemma}
\begin{proof}
The fact that $\mathcal{Q}_c(\eta)$ is self-adjoint on $L^2(\mathbb{R})$ with domain $H^2(\mathbb{R})$ follows directly from the regularity of $\eta$. Also, since $\eta(x)\to 0$ as $x\to\infty$, then the continuous spectrum of $\mathcal{Q}_c(\eta)$ coincides with that of $\mathcal{Q}_c(0)$. Further, due to the fact that $\mathcal{Q}_c(0)$ is translation invariant, the whole spectrum is therefore continuous. The symbol of $\mathcal{Q}_c(0)$ is
\begin{equation*}
     \mathfrak{q}_c(\xi)=-g \jump{\rho}\Bigg[1-\dfrac{1}{\alpha} \left(\sum_{\pm} \dfrac{\rho_{\pm}}{\rho_-}d_+\xi \coth(d_{\pm}\xi)-\beta d_+^2\xi^2+\left(\dfrac{ \omega_+d_+\varrho}{c}-\dfrac{\omega_-d_+}{c}\right)\right)\Bigg].
\end{equation*}
The continuous spectrum results from looking at the range of the mapping $\xi \mapsto \mathfrak{q}_c(\xi)$, and 
$\tau_*=\min\{\mathfrak{q}_c(\xi):\xi\in \mathbb{R}\}.$ For $\beta \geq \beta_0$, the minimum is attained at $\xi=0$. But for $\beta < \beta_0$, the minimum is attained at some value $\xi \neq 0$.
\end{proof}
\subsection{Rescaled operator} Now, we will make a use of a long-wave rescaling to obtain the information of the leading-order form of the operator $\mathcal{Q}_c(\eta)$. Assume $\beta>\beta_0$ 
and $\alpha=\alpha_0+\epsilon^2$, consider the following rescaling operator:
\[
S_\epsilon:=f\left(\dfrac{\epsilon \cdot}{d_+}\right).
\]
It is clear to see that $S_\epsilon$ is an isomorphism on $H^k(\mathbb{R})$. Note also that 
\[\partial_x S_\epsilon=\dfrac{\epsilon}{d_+}S_\epsilon \partial_x,\qquad \partial_x S_{\epsilon}^{-1}=\dfrac{d_+}{\epsilon}S_\epsilon^{-1}\partial_x.\]
This shows that $\partial_x S_\epsilon$ and $\partial_x S_\epsilon^{-1}$ are uniformly bounded in $\T{Lin}(H^{k+1},H^k).$ 

\begin{lemma}[Expansion of $\tilde{\mathcal{Q}}_\epsilon(\eta_\epsilon)$]\label{expansion of Q}
The operator $\tilde{\mathcal{Q}}_\epsilon(\eta_\epsilon)$ admits the following expansion

\[\tilde{\mathcal{Q}}_\epsilon(\eta_\epsilon)=\tilde{\mathcal{Q}}_\epsilon(0)+\tilde{\mathcal{R}}_\epsilon,\] where

\begin{equation}\label{R epsilon}
\tilde{\mathcal{R}}_\epsilon=-3\left(\varrho-\dfrac{1}{d^2}+\dfrac{\omega_+d_+\varrho}{c}+\dfrac{\omega_-d_+}{cd}+\dfrac{\omega_+^2d_+^2\varrho}{3c^2}-\dfrac{\omega_-^2d_+^2}{3c^2}\right)\tilde{\eta}+ O(\epsilon),
\end{equation} in $\T{Lin }(H^{k+2},H^k).$ 
\end{lemma}
\begin{proof}
Let $\mathcal{Q}_\epsilon$ be the operator obtained by evaluating operator $\mathcal{Q}_c$ at $\eta_\epsilon$:

\begin{equation}\label{operator Q epsilon}
    \begin{split}
        \mathcal{Q}_\epsilon(\eta_\epsilon)&:=-\partial_x \left(\sigma \dfrac{\partial_x}{\langle\eta'_\epsilon \rangle^3}\right)-\left(g\jump{\rho} +\sum_{\pm}\left(\pm \rho_{\pm}b_{1\epsilon}^{\pm}(b_{2\epsilon}^{\pm})'\right)+\jump{\rho\xi_{\epsilon}'\omega}-\eta_\epsilon \jump{\rho\omega^2}+c \jump{\rho \omega}\right)\\& \qquad+\sum_{\pm} \rho_{\pm} b_{1\epsilon}^{\pm} \partial_x\mathcal{G}_{\pm}(\eta_\epsilon)^{-1}\partial_x b_{1\epsilon}^{\pm} \\&\qquad-\sum_{\pm}\rho_{\pm} \left(\mp \omega_{\pm}\eta (b_2^{\pm})'-a_1^{\pm}(\eta,\Upsilon_{\mp})(b_2^{\pm})'\right) \\&\qquad +2 \rho_+\rho_-\jump{\omega}\partial_x \mathcal{B}(\eta_\epsilon)^{-1} \jump{\omega} \eta_\epsilon \eta_\epsilon' +2 \rho_+ \rho_- \jump{\omega} \eta_\epsilon \partial_x\left(\mathcal{B}(\eta_\epsilon)^{-1} \jump{\omega} ( \eta_\epsilon)_x\right)\\&\qquad- 4 \rho_-\rho_+ \sum_{\pm} \left(a_1^{\mp}(\eta_\epsilon, \Upsilon_{\epsilon_{\pm}})\mathcal{B}^{-1}\jump{\omega}(\eta_\epsilon)_x\right) \\& \qquad - 2 \rho_+ \rho_-\sum_{\pm} a_1^{\mp}(\eta_\epsilon,\Upsilon_{\epsilon_{\pm}}) \partial_x\left(\mathcal{B}^{-1}(\eta_\epsilon) \left(\sum_{\pm} a_1^{\mp}(\eta_\epsilon,\Upsilon_{\epsilon_{\pm}})\dot{\eta}\right)'\right).
    \end{split}
\end{equation}
Following the same rescaling argument  in \cite{Ming--WalshOrbital2022}, the rescaled surface tension term becomes

\begin{equation*}
    \dfrac{-1}{\epsilon^2}\dfrac{d_+}{c^2\rho_-}S_\epsilon^{-1}\partial_x \left(\sigma \dfrac{\partial_x}{\langle\eta'_\epsilon \rangle^3}\right)S_\epsilon=-\partial_x\left(\dfrac{\beta}{\langle\epsilon^3(\tilde{\eta}'+\tilde{r}_\epsilon ')\rangle^3}\partial_x \right).
\end{equation*}
Further, let us define the non-dimensionalized and rescaled relative velocity field
\begin{equation}\label{rescalledb1b2}
    b_{1\epsilon}^{\pm}=:cS_\epsilon \tilde{b}_1^{\pm}, \quad b_{2\epsilon}^{\pm}=:cS_\epsilon \tilde{b}_2^{\pm}. 
\end{equation}
Since $b_{2\epsilon}^{\pm}=\eta_\epsilon'\left(b_{1\epsilon}^{\pm} \right)$, we have $\tilde{b}_2^{\pm}=\epsilon^3\tilde{\eta}'\tilde{b}_1^{\pm}$. Using this, we obtain the following rescaled expression of the second and the third terms:
 \begin{equation*}
 \begin{split}
     \dfrac{-1}{\epsilon^2}\dfrac{d_+}{c^2\rho_-}S_\epsilon^{-1}\left(g\jump{\rho}+c\jump{\rho \omega} +\sum_{\pm}\left(\pm \rho_{\pm}b_{1\epsilon}^{\pm}(b_{2\epsilon}^{\pm})'\right)\right)S_\epsilon&=\dfrac{\alpha}{\epsilon^2}-\dfrac{1}{\epsilon^2}\left(\dfrac{\omega_+d_+\varrho}{c}-\dfrac{\omega_-d_+}{c}\right)\\& \quad -\epsilon^2\sum_{\pm}\pm \dfrac{\rho_{\pm}}{\rho_-}\tilde{b}_1^{\pm}(\tilde{\eta}'\tilde{b}_1^{\pm})'
 \end{split}
 \end{equation*}
 
 For later use in dealing with the non-local terms in $\mathcal{Q}_\epsilon(\eta_\epsilon)$, we define the following two operators:
 \begin{equation}\label{definition of M epsilon}
     \tilde{\mathcal{M}}^{\pm}_{\epsilon}(\eta_\epsilon):=\dfrac{d_+}{\epsilon^2}S^{-1}_{\epsilon} \partial_x\mathcal{G}_{\pm}(\eta_\epsilon)^{-1}\partial_x S_{\epsilon}
 \end{equation}
 and
 \begin{equation}
     \tilde{\mathcal{Z}}_{\epsilon}(\eta_\epsilon):=\dfrac{d_+}{\epsilon^2}S^{-1}_{\epsilon} \partial_x\mathcal{B}(\eta_\epsilon)^{-1}\partial_x S_{\epsilon}
 \end{equation}
 For any $f\in H^{k+2}$, we have
 \begin{equation*}
      \mathcal{F}(\tilde{\mathcal{M}}^{\pm}_{\epsilon}(0)f)(\xi)=\dfrac{d_+}{\epsilon^2}\dfrac{\epsilon}{d_+}\mathcal{F}(\partial_x\mathcal{G}_{\pm}(\eta_\epsilon)^{-1}\partial_x S_{\epsilon}f)(\dfrac{\epsilon}{d_+}\xi)=\dfrac{d_+}{\epsilon^2}\mathfrak{m}_{\pm}(\dfrac{\epsilon}{d_+}\xi) \hat{f}(\xi),
 \end{equation*} where $\mathfrak{m}_{\pm}:=-\xi \coth(d_{\pm}\xi)$ is the symbol for $\partial_x\mathcal{G}_{\pm}(0)^{-1}\partial_x$. Thus, $\tilde{\mathcal{M}}^{\pm}_{\epsilon}(0)$ is a Fourier multiplier with a symbol given by
\begin{equation*}
    \tilde{\mathfrak{m}}^{\pm}_{\epsilon}(\xi):=\dfrac{-1}{\epsilon^2}\dfrac{\epsilon \xi}{\tanh(d_{\pm}\epsilon \xi/d_+)}.
\end{equation*}
Additionally,  for any $f\in H^{k+2}$, we have
 \begin{equation*}
      \mathcal{F}(\tilde{\mathcal{Z}}_{\epsilon}(0)f)(\xi)=\dfrac{d_+}{\epsilon^2}\dfrac{\epsilon}{d_+}\mathcal{F}(\partial_x\mathcal{B}(\eta_\epsilon)^{-1}\partial_x S_{\epsilon}f)(\dfrac{\epsilon}{d_+}\xi)=\dfrac{d_+}{\epsilon^2}\mathfrak{b}_{\pm}(\dfrac{\epsilon}{d_+}\xi) \hat{f}(\xi),
 \end{equation*}
where $\mathfrak{b}_{\pm}:=-\xi/(\rho_+ \tanh(d_-\xi)+\rho_-\tanh(d_+\xi))$ is the symbol for $\partial_x\mathcal{B}(0)^{-1}\partial_x$. Thus, $\tilde{\mathcal{Z}}_{\epsilon}(0)$ is a Fourier multiplier with a symbol given by
\begin{equation*}
    \tilde{\mathfrak{b}}^{\pm}_{\epsilon}(\xi):=\dfrac{-1}{\epsilon^2}\dfrac{\epsilon \xi}{\left(\rho_+\tanh(d\epsilon \xi)+\rho_-\tanh(\epsilon \xi)\right)}.
\end{equation*}
 As a result we get
 \begin{equation}\label{inequality for M epsilon}
     \lVert\epsilon^2 \tilde{\mathcal{M}}^{\pm}_{\epsilon}(0)+\dfrac{d_+}{d_{\pm}} \rVert_{\T{Lin}(H^{k+2},H^k)} \leq \lVert\dfrac{1}{\langle \cdot \rangle^2}\left(\epsilon^2\tilde{\mathfrak{m}}^{\pm}_{\epsilon} +\dfrac{d_+}{d_{\pm}}\right)\rVert\lesssim \epsilon^2,
 \end{equation}
  and
 \begin{equation}
     \lVert\epsilon^2 \tilde{\mathcal{Z}}_{\epsilon}(0)+\dfrac{d_+}{\left(\rho_-d_{+}+\rho_+d_{-}\right)} \rVert_{\T{Lin}(H^{k+2},H^k)} \leq \lVert\dfrac{1}{\langle \cdot \rangle^2}\left(\epsilon^2\tilde{\mathfrak{b}}^{\pm}_{\epsilon} +\dfrac{d_+}{\left(\rho_-d_{+}+\rho_+d_{-}\right)}\right)\rVert\lesssim \epsilon^2.
 \end{equation}
  In particular, we obtain
  \begin{equation*}
      \tilde{\mathcal{Q}}_\epsilon(0)=\dfrac{1}{\epsilon^2}\left(-\epsilon^2 \beta \partial_x^2 +\alpha -\left(\dfrac{\omega_+d_+\varrho}{c}-\dfrac{\omega_-d_+}{c}\right)  +\sum_{\pm}\dfrac{\rho_{\pm}}{\rho_-}\epsilon^2 \tilde{\mathcal{M}}^{\pm}_{\epsilon}(0)\right).
  \end{equation*}
We are now ready to carefully analyze the remainder operator:
\begin{equation}\label{remainder operator}
\begin{split}
     &\tilde{\mathcal{R}}_{\epsilon}:= \tilde{\mathcal{Q}}_{\epsilon}(\eta_{\epsilon})-\tilde{\mathcal{Q}}_\epsilon(0)\\&=-\beta\partial_x\left(\dfrac{\beta}{\langle\epsilon^3(\tilde{\eta}'+\tilde{r}_\epsilon')\rangle^3}-1\right)\partial_x- \epsilon^2 \sum_{\pm} \pm \dfrac{\rho_{\pm}}{\rho_-}\tilde{b}_1^{\pm}(\tilde{\eta}'\tilde{b}_1^{\pm})'\\&+\sum_{\pm}\dfrac{\rho_\pm}{\rho_-}\left(\tilde{b}_1^{\pm} \tilde{\mathcal{M}}_\epsilon^{\pm}(\eta_\epsilon) \tilde{b}_1^{\pm}-\tilde{\mathcal{M}}_\epsilon^{\pm}(0) \right)
     - \dfrac{d_+}{\epsilon^2 c^2 \rho_-}\jump{\rho \xi_x \omega}+\dfrac{\jump{\rho \omega^2}}{c^2 \rho_-} \tilde{\eta}.
\end{split}
\end{equation}
It is straight forward to see that 
\begin{equation*}
    -\beta\partial_x\left(\dfrac{\beta}{\langle\epsilon^3(\tilde{\eta}'+\tilde{r}_\epsilon')\rangle^3}-1\right)\partial_x=O(\epsilon^9) \qquad \T{in } \T{Lin }(H^{k+2},H^{k}).
\end{equation*}
In view of \eqref{ b_1 and b_2} and \eqref{rescalledb1b2}, we know that

\begin{equation*}
    \tilde{b}_1^{\pm}=\dfrac{1}{c}S_\epsilon^{-1}\left(\partial_x \phi_{\epsilon_{\pm}}|_{\mathscr{S}}-c-\omega_{\pm}\eta_\epsilon\right). 
\end{equation*}
Expanding the Dirichlet--Neumann operator in $H^k$, we obtain
\begin{equation} \label{expansion of xi' epsilon}
\begin{split}
     \xi'_{\epsilon_{\pm}}&=\pm \partial_x \mathcal{G}_{\pm}(\eta_\epsilon)^{-1}\left(c \eta_{\epsilon}' +\omega_{\pm}\eta_\epsilon \eta_\epsilon'\right)\\&
     =\pm \partial_x \left[\mathcal{G}_{\pm}(0)^{-1}(c \eta_{\epsilon}' +\omega_{\pm}\eta_\epsilon \eta_\epsilon')+\langle \T{D}\mathcal{G}_{\pm}(0)^{-1}\eta_\epsilon,(c \eta_{\epsilon}' +\omega_{\pm}\eta_\epsilon \eta_\epsilon') \rangle\right]+O(\epsilon^6)\\ \\ (\partial_x\phi_{\epsilon_{\pm}})|_{\mathscr{S}}&=\dfrac{1}{1+(\eta_\epsilon')^2}\left(\xi_{\pm}' \pm \eta_\epsilon' \mathcal{G}_{\pm}(\eta_\epsilon)\xi_{\pm}\right)\\&=\pm\partial_x \left[\mathcal{G}_{\pm}(0)^{-1}(c \eta_{\epsilon}' +\omega_{\pm}\eta_\epsilon \eta_\epsilon')+\langle \T{D}\mathcal{G}_{\pm}(0)^{-1}\eta_\epsilon,(c \eta_{\epsilon}' +\omega_{\pm}\eta_\epsilon \eta_\epsilon') \rangle\right]+O(\epsilon^6)
\end{split}
\end{equation}
Combining the formula
\[
\langle \T{D}\mathcal{G}_{\pm}(0)^{-1} \eta_\epsilon, f\rangle=-\mathcal{G}_{\pm}(0)^{-1}\langle \T{D} \mathcal{G}_{\pm}(0) \eta_\epsilon, \mathcal{G}_{\pm}(0)^{-1}f\rangle
\]
with the first derivative formula for $\mathcal{G}_{\pm}(\eta)$, we can infer that
\begin{equation*}
\begin{split}
    \langle \T{D}\mathcal{G}_{\pm}(0)\eta_\epsilon,\mathcal{G}_{\pm}(0)^{-1}\partial_x S_\epsilon f \rangle&=\pm \partial_xS_{\epsilon} \epsilon^{4}\left(\tilde{\mathcal{M}}^{\pm}_{\epsilon}(0)f\right)\tilde{\eta}\pm \epsilon^3\mathcal{G}_{\pm}(0)S_\epsilon(\tilde{\eta}\partial_x f).
\end{split}
\end{equation*}
Patterning  the computation done in \cite{Ming--WalshOrbital2022}, therefore we can say that
\begin{equation}\label{b1 expansion}
    \tilde{b}_1^{\pm}=-1 \mp \epsilon^2 \dfrac{d_+}{d_\pm}\tilde{\eta}-\epsilon^2\dfrac{\omega_{\pm}d_+}{c}\tilde{\eta}-\epsilon^{4}\dfrac{d_+^2}{d_{\pm}^2} \tilde{\eta}^2+O(\epsilon^6).
\end{equation}
Therefore, the second term in the remainder operator \eqref{remainder operator}
\begin{equation}
\begin{aligned}
    \epsilon^2 \sum_{\pm} \pm \dfrac{\rho_{\pm}}{\rho_-}\tilde{b}_1^{\pm}(\tilde{\eta}'\tilde{b}_1^{\pm})&=\epsilon^2(1-\varrho)\tilde{\eta}'' +\epsilon^4\left(\dfrac{1}{d}-\varrho+\dfrac{\omega_-d_+}{c}-\dfrac{\omega_+d_+}{c}\right)\left[2\tilde{\eta}\tilde{\eta}''+(\tilde{\eta}')^2\right]\\&\qquad+O(\epsilon^6).
\end{aligned}
\end{equation}
Utilizing the expansion in \eqref{b1 expansion} yields
\begin{equation}\label{b1Mb1}
\begin{split}
    \tilde{b}_1^{\pm} \tilde{\mathcal{M}}_\epsilon^{\pm}(\eta_\epsilon) \tilde{b}_1^{\pm}&=\tilde{\mathcal{M}}_\epsilon^{\pm}(\eta_\epsilon)\pm \epsilon^2 \dfrac{d_+}{d_{\pm}}\left(\tilde{\eta}\tilde{\mathcal{M}}_\epsilon^{\pm}(\eta_\epsilon) + \tilde{\mathcal{M}}_\epsilon^{\pm}(\eta_\epsilon)\tilde{\eta} \right)\\& \qquad
    +\epsilon^2 \dfrac{\omega_{\pm}d_+}{c}\left(\tilde{\eta}\tilde{\mathcal{M}}_\epsilon^{\pm}(\eta_\epsilon) + \tilde{\mathcal{M}}_\epsilon^{\pm}(\eta_\epsilon)\tilde{\eta} \right)\\& \qquad
    +\epsilon^4\dfrac{d_+^2}{d_{\pm}^2}\left(\tilde{\eta}\tilde{\mathcal{M}}_\epsilon^{\pm}(\eta_\epsilon)\tilde{\eta} + \tilde{\eta}^2\tilde{\mathcal{M}}_\epsilon^{\pm}(\eta_\epsilon)+\tilde{\mathcal{M}}_\epsilon^{\pm}(\eta_\epsilon) \tilde{\eta}^2 \right)\\& \qquad+ \epsilon^4 \dfrac{\omega_{\pm}d_+}{c}\left(\pm2 \dfrac{d_+}{d_{\pm}}+\dfrac{\omega_{\pm}d_+}{c}\right)\tilde{\eta}\tilde{\mathcal{M}}_\epsilon^{\pm}(\eta_\epsilon)\tilde{\eta}+O(\epsilon^6) \; \textup{ in }\T{Lin}(H^{k+2},H^k).
\end{split}
\end{equation}
Furthermore, for $f\in H^{k+2}$ with $\norm{f}_{H^{k+2}}=1$, we have
\begin{equation}
    \begin{split}
        \left(\tilde{\mathcal{M}}_\epsilon^{\pm}(\eta_\epsilon)-\tilde{\mathcal{M}}_\epsilon^{\pm}(0)\right)f&=\dfrac{d_+}{\epsilon^2}S_{\epsilon}^{-1} \partial_x \left(\mathcal{G}_{\pm}(\eta_\epsilon)-\mathcal{G}_{\pm}(0)^{-1}\right)\partial_x S_\epsilon f\\&
        =\dfrac{d_+}{\epsilon^2}S_{\epsilon}^{-1} \partial_x \langle \T{D}\mathcal{G}_{\pm}(0)^{-1}\eta_\epsilon, \partial_x S_\epsilon f \rangle \\&
        +\dfrac{d_+}{2 \epsilon^2} S_\epsilon^{-1} \partial_x \langle \T{D}^2 \mathcal{G}_{\pm}(0)^{-1}[\eta_\epsilon,\eta_\epsilon],  \partial_x S_\epsilon f \rangle + O(\epsilon^4)
    \end{split}
\end{equation}
in $H^k$.
In order to understand the third term in \eqref{R epsilon} better, further expansion needs to be done to the equation above. In \cite{Ming--WalshOrbital2022} 
such computation has been done carefully. From that we can infer
\begin{equation}
    \begin{split}
        \left(\tilde{\mathcal{M}}_\epsilon^{\pm}(\eta_\epsilon)-\tilde{\mathcal{M}}_\epsilon^{\pm}(0)\right)f&= \mp \dfrac{d_+^2}{d_\pm^2}\tilde{\eta}f \mp \epsilon^2 \partial_x(\tilde{\eta}\partial_x f)-\epsilon^2 \dfrac{d_+^3}{d_{\pm}^3}\tilde{\eta}^2 f+ O(\epsilon^3),
    \end{split}
\end{equation}
in $H^k$.
The expression in \eqref{b1Mb1} becomes
\begin{equation}
\begin{split}
    \tilde{b}_1^{\pm} \tilde{\mathcal{M}}_\epsilon^{\pm}(\eta_\epsilon) \tilde{b}_1^{\pm}f&=\tilde{\mathcal{M}}_\epsilon^{\pm}(0)f\pm \epsilon^2 \dfrac{d_+}{d_{\pm}}\left(\tilde{\eta}\tilde{\mathcal{M}}_\epsilon^{\pm}(0)f + \tilde{\mathcal{M}}_\epsilon^{\pm}(0)\tilde{\eta}f \right)\\& \qquad+\epsilon^2 \dfrac{\omega_{\pm}d_+}{c}\left(\tilde{\eta}\tilde{\mathcal{M}}_\epsilon^{\pm}(0)f + \tilde{\mathcal{M}}_\epsilon^{\pm}(0)\tilde{\eta}f \right)\mp \dfrac{d_+^2}{d_{\pm}^2}\tilde{\eta}f +O(\epsilon^2).
\end{split}
\end{equation}
Finally, replacing $\tilde{\mathcal{M}}_\epsilon^{\pm}(0)$ with $-d_+/(d_{\pm}\epsilon^2)$ as in \eqref{inequality for M epsilon}, we obtain
\begin{equation}
    \tilde{b}_1^{\pm} \tilde{\mathcal{M}}_\epsilon^{\pm}(\eta_\epsilon) \tilde{b}_1^{\pm}f= \tilde{\mathcal{M}}_\epsilon^{\pm}(0)f \mp 3 \dfrac{d_+^2}{d_{\pm}^2}\tilde{\eta}f-2\dfrac{\omega_\pm d_+^2}{d_{\pm}c}\tilde{\eta} f+O(\epsilon^2).
\end{equation}

Hence, the third term in \eqref{remainder operator} reads
\begin{equation}\label{b1Mepsilonb1 minus M0}
\begin{split}
    \sum_{\pm}&\dfrac{\rho_\pm}{\rho_-} \left(\tilde{b}_1^{\pm} \tilde{\mathcal{M}}_\epsilon^{\pm}(\eta_\epsilon) \tilde{b}_1^{\pm}-\tilde{\mathcal{M}}_\epsilon^{\pm}(0)\right)f\\&=3 \sum_{\pm} \mp\dfrac{\rho_{\pm}}{\rho_-}\dfrac{d_+^2}{d_{\pm}^2}\tilde{\eta}f -2\sum_{\pm}\dfrac{\rho_{\pm}}{\rho_-}\dfrac{\omega_\pm d_+^2}{d_{\pm}c}\tilde{\eta} f+O(\epsilon^2)\\&=-3\left(\varrho-\dfrac{1}{d^2}\right)\tilde{\eta}f-2\left(\dfrac{\varrho \omega_+ d_+}{c}+\dfrac{\omega_- d_+ }{cd}\right)\tilde{\eta}f +O(\epsilon^2).
\end{split}
\end{equation}
The next term to analyze is $\jump{\rho\xi_{\epsilon}'\omega}$. For that, we are going to exploit the expression in \eqref{expansion of xi' epsilon}. Via the same type expansion procedure, one can infer
\begin{equation*}
    \xi_{\epsilon_{\pm}}'=\pm \dfrac{d_+^2}{c\rho_-d_{\pm}}+O(\epsilon^2) \qquad \T{in } H^k.
\end{equation*}
Hence,
\begin{equation}\label{jump rho xi tilde prime omega}
   - \jump{\rho\xi_{\epsilon}'\omega}=-\left(\dfrac{\omega_+d_+\varrho}{c}+\dfrac{\omega_-d_+}{cd}\right).
\end{equation}

Another term in $\mathcal{Q}_\epsilon (\eta_\epsilon)$ that we still have yet to handle is  $-\eta_\epsilon \jump{\rho \omega^2}$. It is straightforward to see that

\begin{equation}\label{jump rho omega squared term}
    -\dfrac{1}{\epsilon^2}\dfrac{d_+}{c^2\rho_-}S_{\epsilon}^{-1}  \eta_\epsilon \jump{\rho \omega^2}S_\epsilon= -\dfrac{1}{\epsilon^2}\dfrac{d_+}{c^2\rho_-}  \epsilon^2 d_+ \jump{\rho \omega^2}+O(\epsilon)=-\dfrac{\omega_+^2\varrho d_+^2}{c^2}+\dfrac{\omega_-^2d_+^2}{c^2}+O(\epsilon).
\end{equation}

Finally, it remains to deal with all the expressions of the last four rows in \eqref{operator Q epsilon}. Again, going through the same analysis as before and using the fact that $\tilde{\mathcal{Z}}_\epsilon(0)=-d_+/(\epsilon^2 d_{\pm})$, one can show that those terms are lower order (sufficiently small). 

Combining \eqref{b1Mepsilonb1 minus M0}, \eqref{jump rho xi tilde prime omega}, and \eqref{jump rho omega squared term} we arrive at \eqref{R epsilon}. The proof is then complete.

\end{proof}

\begin{lemma}[Limiting rescaled operator]\label{limiting rescaled} Consider the rescaled operator $\tilde{\mathcal{Q}}_\epsilon(\eta_\epsilon)$. Assume that $\beta>\beta_0$ and $\alpha=\alpha_0+\epsilon^2$. Then for any $k>1/2$ and $\zeta \in H^{k+2}$, we have
\[\lVert\tilde{\mathcal{Q}}_\epsilon(\eta_\epsilon)\zeta-\tilde{\mathcal{Q}}_0 \zeta \rVert_{H^k}\rightarrow 0 \quad \T{as } \epsilon \searrow 0,\]  where the operator 
\begin{equation}\label{Q tilde 0}
    \tilde{\mathcal{Q}}_0=-\left(\beta-\beta_0\right)\partial_x^2 +1 -3\left(\varrho-\dfrac{1}{d^2}+\dfrac{\omega_+d_+\varrho}{c}+\dfrac{\omega_-d_+}{cd}+\dfrac{\omega_+^2d_+^2\varrho}{3c^2}-\dfrac{\omega_-^2d_+^2}{3c^2}\right)\tilde{\eta}.
\end{equation}
\end{lemma}
\begin{proof}
Let $k>1/2$. Recall from the expansion of $\mathcal{Q}_{\epsilon}$ in Lemma \ref{expansion of Q}, we have $\tilde{\mathcal{Q}}_\epsilon (\eta_\epsilon)=\tilde{\mathcal{Q}}_\epsilon (0) + \tilde{\mathcal{R}}_\epsilon$. It is also important to note that $\tilde{\mathcal{R}}_\epsilon$ has a uniform limit at linear map from $H^{k+1}$ to $H^k$ as $\epsilon \searrow 0$. Moreover, the operator $\tilde{\mathcal{Q}}_\epsilon(0)$ is a Fourier multiplier in $H^{k+2}$ (as $\tilde{\mathcal{M}}_\epsilon (0)$ is). Precisely,
\begin{equation*}
\begin{aligned}
    \mathcal{F}(\tilde{\mathcal{Q}}_\epsilon(0)f) (\xi)&=\dfrac{1}{\epsilon^2}\left(\epsilon^2 \beta \xi^2 +\alpha-\dfrac{\jump{\rho \omega d_+}}{\rho_- c}- \sum_{\pm} \dfrac{\rho_{\pm}}{\rho_-}\epsilon \xi \coth\left(\dfrac{d_{\pm}}{d_+}\epsilon \xi\right) \right)\hat{f}(\xi)\\&=:\tilde{\mathfrak{q}}(\xi) \hat{f}(\xi).
\end{aligned}
\end{equation*}
Further the symbol $\tilde{\mathfrak{q}}$ can be re-written in the following way
\begin{align*}
     \tilde{\mathfrak{q}}_\epsilon(\xi)&= \epsilon^{2-n}(\beta-\beta_0)\xi^2 +\dfrac{\alpha-\alpha_0}{\epsilon^2}+\dfrac{1}{\epsilon^2}\left(\beta_0(\epsilon \xi)^2 + \alpha_0 -\dfrac{\jump{\rho \omega d_+}}{\rho_- c}-\sum_{\pm} \dfrac{\rho_{\pm}}{\rho_-}\epsilon \xi \coth\left(\dfrac{d_{\pm}}{d_+}\epsilon \xi\right)\right)\\&=\epsilon^{2-n}(\beta-\beta_0)\xi^2+\dfrac{\alpha-\alpha_0}{\epsilon^2}+O((\epsilon \xi)^4),
\end{align*}
as $\epsilon \xi \to 0$. Since $\alpha=\alpha_0+\epsilon^2$, hence for each fixed $\xi \in \mathbb{R}$, we obtain
\[\tilde{\mathfrak{q}}_\epsilon(\xi) \to (\beta-\beta_0)\xi^2 +1 \T{  as } \epsilon \searrow 0.\]
Combining this result with the expression for $\tilde{\mathcal{R}}_\epsilon$ in \eqref{R epsilon}, we obtain the formula of $\tilde{\mathcal{Q}}_0$.
\end{proof}
\subsection{Spectrum of the linearized augmented potential}
Having established the limiting behavior above, we will now analyze the spectrum of the operator $\mathcal{Q}_{\epsilon}(\eta_{\epsilon})$. Recall that, the operator $\mathcal{Q}_\epsilon (\eta_\epsilon)$ only converges point-wise to the operator $\mathcal{Q}_0(0)$ whose essential spectrum is $[0,\infty)$. This is certainly creates a challenge in deducing the spectrum of $\mathcal{Q}_\epsilon (\eta_\epsilon)$. Thanks to the rescaled operator $\tilde{\mathcal{Q}}_\epsilon(\eta_\epsilon)$ introduced earlier. It is known that such rescaled operator converges point-wise to $\tilde{\mathcal{Q}_0}$, which has a gap between 0 and the positive essential spectrum. 
\begin{lemma}[Spectrum of the Limiting Operator $\tilde{\mathcal{Q}}_0$]
Let the assumptions in Lemma \ref{limiting rescaled} hold. The limiting rescaled operator $\tilde{\mathcal{Q}}_0$ satisfies 
\begin{equation}
    \T{ess spec } \tilde{\mathcal{Q}}_0=[1,\infty),\qquad \T{spec } \tilde{\mathcal{Q}}_0 =\{-\tilde{\tau}^2,0\} \cup \tilde{\Lambda},
\end{equation}where the first two eigenvalues $-\tilde{\tau}^2<0$ and $0$ are both simple eigenvalues with the corresponding eigenfunctions $\tilde{g}_1$ and $\tilde{g}_2=\tilde{\eta}'$, respectively; and there exists $\tau_*>0$ such that $\tilde{\Lambda} \subset [\tau_*,\infty) $.
\end{lemma}
\begin{proof}
The spectra condition above is a classic result and can be found, for instance, in the surveys of Pava \cite{Pava2009NonlinearDE}. It is important to mention that the proof of that is not trivial by any means. The fact that $-\tilde{\tau}^2$ and 0 are simple follows from the result which says that: the Wronskian of any two solutions in $L^2$ of the eigenvalue problem $\tilde{\mathcal{Q}}_0f=\tilde{\tau}$ must be 0.
\end{proof}

\begin{theorem}
Suppose that the assumptions of Lemma \ref{limiting rescaled} hold. For each $a\in (0,\tau_*)$, there exists some $\epsilon_0>0$ such that for any $\epsilon \in (0,\epsilon_0)$ the operator $\mathcal{Q}_\epsilon(\eta_\epsilon)$ satisfies
\[\T{ess spec } \mathcal{Q}_{\epsilon}(\eta_{\epsilon}) \subset [\epsilon^2 c^2 \rho_-/d_+, \infty), \quad \T{spec } \mathcal{Q}_{\epsilon}(\eta_{\epsilon})=\{-\tau^2,0\} \cup \Lambda, \]
where $\Lambda \subset [a\epsilon^2c^2\rho_-/d_+,\infty)$, and 
\[\tau^2=\dfrac{\epsilon^2c^2\rho_-}{d_+}\tilde{\tau}^2+o(\epsilon^2)\quad \T{as } \epsilon \searrow 0.\]
The first two eigenvalues $\tau_1:=-\tilde{\tau}^2<0$ and $\tau_2:=0$ are both simple with the corresponding eigenfunctions given by $g_i:=S_\epsilon \tilde{g}_i +o(1)$ in $H^k$ as $\epsilon \searrow 0$. 
\end{theorem}
\begin{proof}
The proof of this theorem can be done in a similar manner as the one in \cite[Theorem 3.2]{Ming--WalshOrbital2022} which is mainly inspired by the proof of \cite[Theorem 4.3]{Mielke2002}. Therefore, we opt to avoid rewriting it here.
\end{proof}

\begin{lemma}[Extension of $\T{D}^2 E_c$]\label{Second derivative of E_c extension}
Let $\{U_c\}$ be a family of bound states, then the operator $\T{D}^2 E_c(U_c)$ can be extended uniquely to a bounded linear operator $H_c:\mathbb{X}\to\mathbb{X}^*$ such that
\begin{equation}
    \T{D}^2E_c(U_c)[\dot{u},\dot{v}]=\langle H_c \dot{u}, \dot{v} \rangle, \qquad \T{ for all } \dot{u},\dot{v}\in \mathbb{V},
\end{equation}
and $I^{-1}H_c$ is self-adjoint on $\mathbb{X}$.
\end{lemma}
\begin{proof}
Let $U_c=(\eta_c, \tilde{\xi}_c)$ be a bound state and $\dot{u}=(\dot{\eta},\dot{\tilde{\xi}})\in \mathbb{V}$ be given. From Lemma \ref{second derivative of Aug Potential} and Lemma \ref{quadratic form}, we know that
\begin{equation*}
\begin{split}
    \T{D}^2E_c(U_c)[\dot{u},\dot{u}]&= \T{D}^2V^{\T{aug}}_c(\eta_c)[\dot{\eta},\dot{\eta}]+\int_{\mathbb{R}} (\mathcal{S}_c-\mathcal{T}_c)\dot{\eta}\mathcal{A}(\eta_c)(\mathcal{S}_c-\mathcal{T}_c)\dot{\eta}\; dx\\& \qquad +2 \T{D}_{\tilde{\xi}}\T{D}_\eta E_c(U_c)[\dot{\eta},\dot{\tilde{\xi}}]+\T{D}^2_{\tilde{\xi}}E_c(U_c)[\dot{\tilde{\xi}},\dot{\tilde{\xi}}]\\&=\langle\mathcal{Q}_C(\eta_c)\dot{\eta},\dot{\eta}\rangle + \int_{\mathbb{R}} (\mathcal{S}_c-\mathcal{T}_c)\dot{\eta}\mathcal{A}(\eta_c)(\mathcal{S}_c-\mathcal{T}_c)\dot{\eta}\; dx+2\int_{\mathbb{R}}\dot{\tilde{\xi}}\langle \T{D}\mathcal{A}(\eta_c)\dot{\eta},\tilde{\xi}_c\rangle\; dx\\&\qquad+2\int_{\mathbb{R}}\dot{\tilde{\xi}} \mathcal{G}_-(\eta_c)\mathcal{B}^{-1}(\eta_c)\jump{\omega}(\eta_c \dot{\eta})_x \;dx \\&\qquad+2\int_{\mathbb{R}}\dot{\tilde{\xi}}\langle \T{D}\mathcal{G}_-(\eta_c)\dot{\eta},\mathcal{B}^{-1}(\eta_c) \jump{\omega}\eta_c \eta_c' \rangle\;dx\\& \qquad +2\int_{\mathbb{R}}\mathcal{G}_-(\eta_c)\dot{\tilde{\xi}}\langle \T{D}\mathcal{B}_-(\eta_c)\dot{\eta}, \jump{\omega}\eta_c \eta_c' \rangle \;dx\\& \qquad
    +2\int_{\mathbb{R}}\dot{\tilde{\xi}} \omega_-(\dot{\eta}\eta_c)_x\;dx +2c \int_{\mathbb{R}}\dot{\eta}' \dot{\tilde{\xi}}\; dx + \int_{\mathbb{R}} \dot{\tilde{\xi}} \mathcal{A}(\eta_c) \dot{\tilde{\xi}} \;dx.
\end{split}
\end{equation*}
Expanding some of the terms using the first derivative formula for $\mathcal{G}_-(\eta)$ and $\mathcal{B}^{-1}(\eta)$ yields
\begin{equation*}
    \begin{split}
        \T{D}^2E_c(U_c)[\dot{u},\dot{u}]&=\langle\mathcal{Q}_C(\eta_c)\dot{\eta},\dot{\eta}\rangle + \int_{\mathbb{R}} (\mathcal{S}_c-\mathcal{T}_c)\dot{\eta}\mathcal{A}(\eta_c)(\mathcal{S}_c-\mathcal{T}_c)\dot{\eta}\; dx +2c \int_{\mathbb{R}}\dot{\eta}' \dot{\tilde{\xi}}\; dx + \int_{\mathbb{R}} \dot{\tilde{\xi}} \mathcal{A}(\eta_c) \dot{\tilde{\xi}} \;dx\\&\qquad +2\sum_{\pm} \rho_{\pm} \int_{\mathbb{R}} \left(b_{1c}^{\pm}+c+\omega_{\pm}\eta_c \pm a_1^{\pm}(\eta_c, \Upsilon_{\mp})\right)\left(\mathcal{G}_{\pm}(\eta_c)^{-1}\mathcal{A}(\eta_c) \dot{\tilde{\xi}}\right)'\dot{\eta} \; dx \\& \qquad + 2 \sum_{\pm} \rho_{\pm} \int_{\mathbb{R}} \pm b_2^{\pm} \dot{\eta} \mathcal{A}(\eta_c) \dot{\tilde{\xi}} \; dx\\&\qquad+2\int_{\mathbb{R}}\dot{\tilde{\xi}} \mathcal{G}_-(\eta_c)\mathcal{B}^{-1}(\eta_c)\jump{\omega}(\eta_c \dot{\eta})_x \;dx+2\int_{\mathbb{R}}\dot{\tilde{\xi}}\langle \T{D}\mathcal{G}_-(\eta_c)\dot{\eta},\mathcal{B}^{-1}(\eta_c) \jump{\omega}\eta_c \eta_c' \rangle\;dx\\& \qquad +2\int_{\mathbb{R}}\mathcal{G}_-(\eta_c)\dot{\tilde{\xi}}\langle \T{D}\mathcal{B}_-(\eta_c)\dot{\eta}, \jump{\omega}\eta_c \eta_c' \rangle\;dx.
    \end{split}
\end{equation*}
Upon cancellation and regrouping, we get
\begin{equation}\label{Final form of Second derivative of Ec}
\begin{split}
    &\T{D}^2E_c(U_c)[\dot{u},\dot{u}]\\&=\langle\mathcal{Q}_C(\eta_c)\dot{\eta},\dot{\eta}\rangle + \int_{\mathbb{R}} \left((\mathcal{S}_c-\mathcal{T}_c)\dot{\eta}+\dot{\tilde{\xi}}\right)\mathcal{A}(\eta_c)\left((\mathcal{S}_c-\mathcal{T}_c)\dot{\eta}+\dot{\tilde{\xi}}\right)\; dx\\&\quad +2\sum_{\pm}\rho_{\pm}\int_{\mathbb{R}}\mathcal{G}_{\pm}(\eta_c)^{-1}\left(a_1^{-}(\eta_c, \Upsilon_{-})\dot{\eta}\right)'\mathcal{A}(\eta_c) \dot{\tilde{\xi}}\;dx-\int_{\mathbb{R}}2(a_1^{-1}(\eta_c,\Upsilon_+)\dot{\eta})' \dot{\tilde{\xi}}\;dx.
\end{split}
\end{equation}
Using the fact that 
\[
\mathcal{A}(\eta)^{-1}=\rho_+\mathcal{G}_+(\eta)^{-1} +\rho_-\mathcal{G}_-(\eta)^{-1},
\]
the expression on the second row \eqref{Final form of Second derivative of Ec} vanishes. Therefore, we obtain
\begin{equation}\label{Extension Hc}
\begin{split}
    \T{D}^2E_c(U_c)[\dot{u},\dot{u}]&=\langle\mathcal{Q}_C(\eta_c)\dot{\eta},\dot{\eta}\rangle + \int_{\mathbb{R}} \left((\mathcal{S}_c-\mathcal{T}_c)\dot{\eta}+\dot{\tilde{\xi}}\right)\mathcal{A}(\eta_c)\left((\mathcal{S}_c-\mathcal{T}_c)\dot{\eta}+\dot{\tilde{\xi}}\right)\; dx.
\end{split}
\end{equation}
From the expression above, $\T{D}^2E_c(U_c)$ extends to an element in $\mathbb{X}^*$.
\end{proof}

\begin{theorem}[Spectrum] Let $\{U_c\}$ be one family of bound states given in the Corollary \eqref{Bound States Corollary}. Then
\[
\T{spec } I^{-1}H_c = \{-\mu_c^2,0\} \cup \Sigma_c,
\]
where $-\mu_c^2<0$ is a simple eigenvalue that corresponds to a unique eigenvector $\chi_c$; 0 is simple eigenvalue generated by $T$; and $\Sigma_c \subset (0,\infty)$ is uniformly bounded away from 0.
\end{theorem}
\begin{proof}
The following proof relies on the nice structure of $H_c$ in Lemma \ref{Second derivative of E_c extension}, therefore $I^{-1}H_c$ and the idea presented in \cite[Proposition 5.3]{Mielke2002} and \cite[Theorem 4.1.3]{Ming--WalshOrbital2022}. Consider the following operator
\[
\mathcal{Q}_c(\eta_c)+(\lambda-\tau_c^2)\langle\cdot,g_{1c}\rangle g_{1c} +\lambda \langle \cdot,\eta_c' \rangle \eta_c'.
\] One can easily check that for $\lambda>0$, $-\tau_c^2$ is the negative eigenvalue of $\mathcal{Q}_c(\eta_c)$and $g_{1c}$ is the corresponding eigenfunction. Recall that $\mathcal{A}(\eta_c)$ is a positive definite operator. Therefore, using the expression for $H_c$ from Lemma \ref{Second derivative of E_c extension}, we can conclude that 
\[
\langle H_c u,u\rangle_{\mathbb{X}^*\times \mathbb{X}}+(\alpha-\tau_c^2)\langle I^{-1}(g_{1c},0),u\rangle_{\mathbb{X}^*\times \mathbb{X}}^2+\lambda \langle I^{-1}(\eta_c',0) ,u\rangle_{\mathbb{X}^*\times \mathbb{X}}^2 \gtrsim_{c} \norm{u}_{\mathbb{X}}^2, 
\]
for any $u\in \mathbb{X}$. Hence $I^{-1}H_c$ is positive definite on a codimension 2 space. Further, we know that $T'(0)U_c \in \T{Ker }H_c$. Further,  from \eqref{Extension Hc}, for $u=(g_{1c},(\mathcal{S}_C-\mathcal{T}_c)g_{1c})$, we have
\[
\langle H_c u,u\rangle_{\mathbb{X}^*\times \mathbb{X}}=\langle \mathcal{Q}_c(\eta_c)g_{1c},g_{1c} \rangle_{\mathbb{X}^*\times \mathbb{X}}=-\tau_c^2<0.
\]
Hence, we can conclude that $I^{-1}H_c$ has a one-dimensional kernel generated by $T'(0)U_c$, a one-dimensional negative definite subspace. Moreover, It is positive definite on the orthogonal complement. This then proves Assumption \ref{assumption on spectrum}.
\end{proof}

\section{Proof of theorem}\label{proof of theorem}
In this section, we will prove the (conditional) orbital stability of the bound states presented in Corollary \eqref{Bound States Corollary}. As mentioned previously, it requires us to check the sign of the second derivative of the scalar valued function $d$.

\begin{theorem}[Stability/instability for strong surface tension]
Fix $c_*$ such that $0<\alpha_{c_*}-\alpha_0\ll 1$, the conditional orbital stability/instability of the bound states $U_{c_*}$ in Corollary~\ref{Bound States Corollary} can be determined by looking at the sign of the derivative of a scalar-valued function $m:=m(c)$ stated in \eqref{d'}.
\end{theorem}
\begin{proof}
Having confirmed all the assumptions, The conclusion on (conditional) orbital stability is drawn by showing $d''(c_*)>0$. Recall that since $U_{c_*}$ is a critical point of $E_c$, we therefore have
\begin{equation}
    d'(c_*)=-P(U_{c_*}).
\end{equation}
We need to show that $d'(c)$ is strictly increasing at $c=c_*$. We start by defining a rescaling operator:
\[S_c f:=f\left(\dfrac{\epsilon_c \cdot}{d_+ \sqrt{\beta-\beta_0}}\right).\] Using the asymptotics for the free surface, we define
\[\eta_c=:\epsilon^2_c d_+S_c(\tilde{\eta}_c +\tilde{r}_c), \qquad \T{ for } \tilde{r}_c= O(\epsilon_c).\]

Via the definition of $P$ in \eqref{Total Momentum}, we can express $d'(c)$ as follows
\begin{equation*}
\begin{split}
    d'(c)&=c\int_{\mathbb{R}} \eta_c \partial_x \left(\mathcal{A}(\eta_c)^{-1} \eta_c'\right)\; dx+ \int_{\mathbb{R}} \eta_c \partial_x\left(\mathcal{A}(\eta_c)^{-1}\mathcal{G}_-(\eta_c)\mathcal{B}(\eta_c)^{-1} \rho_+\jump{\omega}\eta_c \eta_c'\right)\; dx \\& \qquad +\int_{\mathbb{R}} \eta_c \omega_- \partial_x\left(\mathcal{A}(\eta_c)^{-1} \eta_c \eta_c'\right) \; dx- \int_{\mathbb{R}}\dfrac{1}{2} \jump{\rho \omega} \eta_c^2\; dx\\&=c\epsilon_c^4 d_+^2\int_{\mathbb{R}} S_c(\tilde{\eta}_c +\tilde{r}_c) \partial_x \left(\mathcal{A}(\eta_c)^{-1} \partial_x S_c(\tilde{\eta}_c+\tilde{r}_c)\right)\; dx\\&\qquad+ \epsilon_c^6 d_+^3 \int_{\mathbb{R}} S_c(\tilde{\eta}_c+\tilde{r}_c) \partial_x\left(\mathcal{G}_+(\eta_c)^{-1}\rho_+\jump{\omega} S_c(\tilde{\eta}_c +\tilde{r}_c) \partial_x S_c(\tilde{\eta}_c +\tilde{r}_c) )\right)\; dx\\&\qquad
    +\epsilon_c^6 d_+^3 \int_{\mathbb{R}}S_c(\tilde{\eta}_c +\tilde{r}_c) \omega_- \partial_x\left(\mathcal{A}(\eta_c)^{-1} S_c(\tilde{\eta}_c +\tilde{r}_c) \partial_xS_c(\tilde{\eta}_c +\tilde{r}_c)\right) \; dx
    \\&\qquad -\epsilon_c^4 d_+^2 \int_{\mathbb{R}}\dfrac{1}{2} \jump{\rho \omega} (S_c(\tilde{\eta}_c+\tilde{r}_c))^2\; dx.
\end{split}
\end{equation*}
Undoing the scaling,
\begin{equation*}
    \begin{split}
    &=c\epsilon_c^3 d_+^3 \sqrt{\beta_c-\beta_0}\int_{\mathbb{R}} (\tilde{\eta}_c +\tilde{r}_c) S_c^{-1}\partial_x \left(\mathcal{A}(\eta_c)^{-1} \partial_x S_c(\tilde{\eta}_c+\tilde{r}_c)\right)\; dx\\&\qquad+ \dfrac{\epsilon_c^5 d_+^4\sqrt{\beta_c-\beta_0}}{2} \int_{\mathbb{R}} (\tilde{\eta}_c+\tilde{r}_c) S_c^{-1}\partial_x\left(\mathcal{G}_+(\eta_c)^{-1}\rho_+\jump{\omega} \partial_x (S_c(\tilde{\eta}_c +\tilde{r}_c))^2\right)\; dx\\&\qquad
    + \dfrac{\epsilon_c^5 d_+^4\sqrt{\beta_c-\beta_0}}{2}\omega_-\int_{\mathbb{R}}(\tilde{\eta}_c +\tilde{r}_c)  S_c^{-1}\partial_x\left(\mathcal{A}(\eta_c)^{-1} \partial_x (S_c(\tilde{\eta}_c +\tilde{r}_c))^2\right) \; dx
    \\&\qquad -\epsilon_c^3 d_+^3\sqrt{\beta_c-\beta_0} \int_{\mathbb{R}}\dfrac{1}{2} (\tilde{\eta}_c +\tilde{r}_c) \jump{\rho \omega} (\tilde{\eta}_c+\tilde{r}_c)\; dx.
\end{split}
\end{equation*}
Using the idea in \cite{Ming--WalshOrbital2022}, we define
$\tilde{\mathcal{M}}_c^{\pm}(\eta_c):= d_+ S_\epsilon^{-1}\partial_x\mathcal{G}_{\pm}(\eta_c)^{-1}\partial_x S_{\epsilon}$.
Following similar line of argument as in Lemma \ref{expansion of Q}, we find that 
 \begin{equation*}
     \lVert \tilde{\mathcal{M}}^{\pm}_{\epsilon}(0)+\dfrac{d_+}{d_{\pm}} \rVert_{\T{Lin}(H^{k+2},H^k)}\lesssim \epsilon_c^2, \qquad \lVert \tilde{\mathcal{M}}^{\pm}_{\epsilon}(\eta_c)-\tilde{\mathcal{M}}^{\pm}_{\epsilon}(0) \rVert_{\T{Lin}(H^{2},L^2)}\lesssim \epsilon_c^2,
 \end{equation*}
 which yields
\begin{equation}\label{expression of d'}
    \begin{split}
&=-c\epsilon_c^3 d_+^2\sqrt{\beta_c-\beta_0}\sum_{\pm}\rho_{\pm}\dfrac{d_+}{d_{\pm}}\int_{\mathbb{R}} (\tilde{\eta}_c)^2 \; dx -\dfrac{\epsilon_c^3 d_+^3\sqrt{\beta_c-\beta_0}}{2}\jump{\rho \omega} \int_{\mathbb{R}} (\tilde{\eta}_c)^2  \; dx+O(\epsilon_c^4)\\&
    =-c\left(\varrho+\dfrac{1}{d}+\dfrac{\omega_+\varrho d_+}{2c}-\dfrac{\omega_-d_+}{2c}\right)\bigg[\epsilon_c^3\rho_-d_+^2\sqrt{\beta_c-\beta_0}\int_{\mathbb{R}}\tilde{\eta}_c^2 \;dx\bigg]+O(\epsilon_c^4).
\end{split}
\end{equation}
Recall that
\[
\tilde{\eta}_c=\dfrac{\sech^2{(x/2)}}{\varrho-\dfrac{1}{d^2}
+\dfrac{\omega_+d_+\varrho}{c}+\dfrac{\omega_-d_+}{cd}+\dfrac{\omega_+^2d_+^2\varrho}{3c^2}-\dfrac{\omega_-^2d_+^2}{3c^2}}.
\]
Thus,
\begin{equation}\label{d'}
\begin{split}
    d'(c) &=-4\dfrac{\left(\varrho+\dfrac{1}{d}+\dfrac{\omega_+\varrho d_+}{2c}-\dfrac{\omega_-d_+}{2c}\right)c\epsilon_c^3\rho_-d_+^2\sqrt{\beta_c-\beta_0}}{\left(\varrho-\dfrac{1}{d^2}
+\dfrac{\omega_+d_+\varrho}{c}+\dfrac{\omega_-d_+}{cd}+\dfrac{\omega_+^2d_+^2\varrho}{3c^2}-\dfrac{\omega_-^2d_+^2}{3c^2}\right)^2}
+O(\epsilon_c^4) \\
& =: m(c) + O(\epsilon_c^4).
\end{split}
\end{equation}
For notation simplicity, let us define the following new variables:
\begin{align*}
	\mathfrak{A} & :=\varrho+\dfrac{1}{d}+\dfrac{\omega_+\varrho d_+}{2c}-\dfrac{\omega_-d_+}{2c}, \\
	\mathfrak{B} & :=\varrho-\dfrac{1}{d^2}
+\dfrac{\omega_+d_+\varrho}{c}+\dfrac{\omega_-d_+}{cd}+\dfrac{\omega_+^2d_+^2\varrho}{3c^2}-\dfrac{\omega_-^2d_+^2}{3c^2}.
\end{align*}
The derivative of $m$ with respect to $c$ in terms of $\mathfrak{A}$ and $\mathfrak{B}$ reads
\begin{equation}
\begin{aligned}
\label{derivative m expression}
&m'(c)=\\& \dfrac{-4}{\mathfrak{B}^4\sqrt{\beta_c -\beta_0}}\Bigl [ \mathfrak{B}^2 \left(\mathfrak{A}' \epsilon_c^3 c (\beta_c-\beta_0) +\mathfrak{A}\epsilon_c^3 (\beta_c-\beta_0)+3\mathfrak{A} c \epsilon_c^2 \epsilon_c' (\beta_c-\beta_0) +\dfrac{\mathfrak{A} c \epsilon_c^3\beta_c'}{2}  \right) \Bigr. \\& \Bigl.
	\phantom{ \mathfrak{B}^2\mathfrak{A}' \epsilon_c^3 c (\beta_c-\beta_0) +\mathfrak{A}\epsilon_c^3 (\beta_c-\beta_0)}
-2\mathfrak{A} \mathfrak{B} \mathfrak{B}' c \epsilon_c^3 (\beta_c -\beta_0)
 \Bigr].
\end{aligned}
\end{equation}
From Theorem~\ref{stability theorem}, we see that for $0 < \epsilon_c \ll 1$, the wave is stable provided $m^\prime(c) > 0$ and unstable if $m^\prime(c) < 0$.  Notice that the sign is determined by the quantity in square brackets on the right-hand side of \eqref{derivative m expression}.  

Next, let us consider the special cases described in Figures~\ref{table1} and \ref{table2}.  From the expression of $m'$ (essentially $d''$ without the lower order term) together with definitions of $\mathfrak{A}, \mathfrak{B}, \epsilon_c$ and their derivatives, we can conclude that for any $c\in \mathcal{I}$, orbital stability ($d''(c)>0$) holds in each of the three cases below.
\begin{itemize}
\item  Sufficiently small-amplitude and near critical Bond number waves: $\beta_c-\beta_0 \ll 1$.
\item Waves of elevation with $\{c<0, \omega_-\geq 0,\frac{2g(\varrho-1)}{c}>\omega_-$ and $\omega_+=0\}$ \textbf{OR}  $\{c>0, \omega_-\leq 0, \frac{2g(\varrho-1)}{c}<\omega_-$ and $\omega_+=0\}$.
\item Waves of depression with  $\{c<0, \omega_+ \leq 0, \frac{2g \jump{\rho}}{c}>-\omega_+\rho_+$ and $\omega_-=0\}$ \textbf{OR}  $\{c>0, \omega_+\geq 0,\frac{2g \jump{\rho}}{c}<-\omega_+\rho_+$ and $\omega_-=0\}$.
\end{itemize}
Moreover, we find some cases where orbital instability ($d''(c)<0$) occur:
\begin{itemize}
\item Waves of depression with $\{c<0, \omega_-\geq 0,\frac{2g(\varrho-1)}{c}<\omega_-$ and $\omega_+=0\}$ \textbf{OR}  $\{c>0, \omega_-\leq 0, \frac{2g(\varrho-1)}{c}>\omega_-$ and $\omega_+=0\}$.
\item Waves of elevation with $\{c<0, \omega_+ \leq 0, \frac{2g \jump{\rho}}{c}<-\omega_+\rho_+$ and $\omega_-=0\}$ \textbf{OR}  $\{c>0, \omega_+\geq 0,\frac{2g \jump{\rho}}{c}>-\omega_+\rho_+$ and $\omega_-=0\}$.
\end{itemize}
This concludes the proof.
\end{proof}

\appendix
\section{Center Manifold} \label{quoted results appendixa}

The first part of appendix records the center manifold reduction theorem introduced in \cite{Mielke1995} that was implemented in, for example, \cite{Nilsson2017}. It is used in proving the existence of small-amplitude internal water waves.

\begin{theorem}[Center manifold]
Consider the differential equation of the form 
\begin{equation}\label{differential equation}
\dot{u}=Lu+F(u,\mu),
\end{equation}
where the unknown $u \in E$ for some Hilbert space $E$, $\mu \in \mathbb{R}^n$ is a parameter and $L: \mathcal{D}(L) \subset E \to E$ is a closed linear operator. Assume that the differential equation \eqref{differential equation} is Hamilton's equations that correspond to the Hamiltonian system $(E, \Omega, H)$ with $0$ being its fix point. Moreover, assume also the following:
\begin{itemize}
\item[\T{H1}] The space $E$ has two closed and $L$-invariant subspaces, namely $E_1$ and $E_2$ such that 
\begin{equation}
\begin{aligned}
E&= E_1 \oplus E_2,\\
\dot{u_1}&=L_1u_1+F_1(u_1+u_2,\mu),\\
\dot{u_2}&=L_2u_2+F_2(u_1+u_2,\mu),
\end{aligned}
\end{equation}
where $L_i=L|_{\mathcal{D}L_i\cap E_i}:\mathcal{D}L_i\cap E_i \to E_i,$ for $i=1,2$ and $F_1=PF$, $F_2=(I-P)F,$ where the operator $P$ is a projection of $E$ onto $E_1$.
\item[\T{H2}] $E_1$ is a finite dimensional Hilbert space and the spectrum of $L_1$ is purely imaginary.
\item[\T{H3}] The imaginary axis lies in the resolvent of $L_2$ and 
\begin{equation}
\norm{(L_2-iaI)^{-1}}\leq \dfrac{C}{1+|a|}, \quad \T{for } a\in \mathbb{R}.
\end{equation}
\item[\T{H4}] There exists $k\in \mathbb{N}$ and neighborhoods $\Lambda \subset \mathbb{R}^n$ and $U\subset \mathcal{D}(L)$ of 0 such that $F$ is $k+1$ continuously differentiable on $U \times \Lambda$ and the derivatives of $F$ are all bounded and uniformly continuous on $U \times \Lambda$ with 
\[
F(0,\mu_0)=0, \qquad dF[0,\mu_0]=0.
\]
\end{itemize}
Under the hypothesis $\T{H1-H4}$ there exist neighborhoods $\tilde{\Lambda}\subset \Lambda$ and $\tilde{U}_1 \subset U \cap E_1,$ $\tilde{U}_2 \subset U \cap E_2$ of zero and a reduction function $r:\tilde{U}_1 \times \tilde{\Lambda} \to \tilde{U}_2$ with the following properties. The reduction  function $r$ is $k$ times continuously  differentiable on $\tilde{U}_1 \times \tilde{\Lambda}$ and the derivatives  of $r$ are bounded and uniformly continuous on $\tilde{U}_1 \times \tilde{\Lambda}$ with
\[
r(0,\mu_0)=0,\qquad dr[0,\mu_0]=0.
\]
The graph 
\[
X^{\mu}_C=\{u_1+r(u_1,\mu)\in \tilde{U}_1\times\tilde{U}_2:u_1\in \tilde{U_1}\},
\]
is a Hamiltonian center manifold for \eqref{differential equation} with the following properties:

\begin{itemize}
\item Through every point in $X^{\mu}_C$ there passes a unique solution of \eqref{differential equation} that stays on $X^{\mu}_C$ as long as it remains in $\tilde{U}_1 \times \tilde{U}_2$. We say that $X^{\mu}_C$ is a locally invariant manifold of \eqref{differential equation}.
\item  Every small bounded solution $u(x),$ $x\in \mathbb{R}$ of \eqref{differential equation} that satisfies $u_1(x) \in \tilde{U}_1$ and $u_2(x) \in \tilde{U}_2$ lies completely in $X^{\mu}_C$.
\item Every solution $u_1$ of the reduced equation 
\begin{equation}\label{reduced differential equation}
\dot{u_1}=L_1u_1+F_1(u_1+r(u_1,\mu),\mu),
\end{equation}
generates a solution 
\[
u(x)=u_1(x)+r(u_1(x),\mu)
\] of \eqref{differential equation}.
\item $X^{\mu}_C$ is a symplectic submanifold $E$, and the flow determined by the Hamiltonian system $X^{\mu}_C, \tilde{\Omega}, \tilde{H}$, where the tilde denotes the  restriction to $X^{\mu}_C$, coincides with the flow on $X^{\mu}_C$ determined by $(E,\Omega, H)$. The reduced equation \eqref{reduced differential equation} represents Hamilton's equations for $(X^{\mu}_C,\tilde{\Omega},\tilde{H})$.
\item If \eqref{differential equation} is reversible, that is if there exists a linear symmetry $S$ that anti-commutes with the right hand side of \eqref{differential equation}, then the reduction function $r$ can be chosen so that it commutes with $S$.
\end{itemize}
\end{theorem}

\section{Formulae}
\label{formulas appendix}
Many of the computations in the present work make use of the first and second derivative formulas of the non-local operators $\mathcal{G}_{\pm}(\eta)$ and $\mathcal{A}(\eta)$. We record them in a series of lemmas below. A derivation can be found in \cite{Ming--WalshOrbital2022}.
\begin{lemma}[First Derivatives]\label{derivatives of G and A} Let $\eta,\tilde{\xi}\in \mathcal{O}\cap \mathbb{V}, \dot{\eta} \in \mathbb{V}_1$ and  be given.
\begin{enumerate}[(a)]
\item The Fr\'echet derivative of $\mathcal{G}_{\pm}(\eta)$ admits the representation formula 
\begin{equation}
    \int_{\mathbb{R}} \zeta \langle \T{D}\mathcal{G}_{\pm}(\eta)\dot{\eta}, \tilde{\xi} \rangle \; dx = \int_{\mathbb{R}} \left( a_1^{\pm} (\eta,\tilde{\xi}) \zeta' +a_2^{\pm}(\eta,\tilde{\xi})\mathcal{G}_{\pm} (\eta) \zeta\right) \dot{\eta}\; dx,
\end{equation}
with 
\begin{equation}\label{equation of a_1 and a_2}
    \begin{split}
        &a_1^{\pm}(\eta,\tilde{\xi}):=\dfrac{1}{1+(\eta')^2} \left(\mp\tilde{\xi}' - \eta' \mathcal{G}_{\pm}(\eta)\tilde{\xi} \right)\\&
        a_2^{\pm}(\eta,\tilde{\xi}):=\dfrac{1}{1+(\eta')^2} \left(\pm\mathcal{G}_{\pm}(\eta)\tilde{\xi} -\eta' \tilde{\xi}' \right).
    \end{split}
\end{equation}
\item The Fr\'echet derivative of $\mathcal{A}(\eta)$ admits the representation formula
\begin{equation}
    \begin{split}
        &\int_{\mathbb{R}} \zeta \langle \T{D} \mathcal{A}(\eta)\dot{\eta}, \tilde{\xi}\rangle \; dx\\&
        =\sum_{\pm} \rho_{\pm} \left( a_1^{\pm} (\eta, \mathcal{A}(\eta)\mathcal{G}_{\pm}(\eta)^{-1}\tilde{\xi}) \left(\mathcal{A}(\eta)\mathcal{G}_{\pm}(\eta)^{-1} \zeta \right)' \right) \dot{\eta}\;dx\\&\quad
        +\sum_{\pm} \rho_{\pm} \int_{\mathbb{R}} \left( a_2^{\pm} (\eta, \mathcal{A}(\eta) \mathcal{G}_{\pm}(\eta)^{-1}\tilde{\xi})\mathcal{A}(\eta)\zeta\right) \dot{\eta}\; dx.
    \end{split}
\end{equation}
\end{enumerate}
\end{lemma}

\begin{lemma}[Second derivative of $\mathcal{G}_{\pm}$]

For all $u=(\eta,\tilde{\xi})\in \mathcal{O}\cap \mathbb{V}$ and $\dot{\eta} \in \mathbb{V}_1,$ it holds that
\begin{equation}
    \begin{split}
        &\int_{\mathbb{R}} \tilde{\xi} \langle \T{D}^2\mathcal{G}_{\pm} (\eta)[\dot{\eta},\dot{\eta}], \tilde{\xi}\rangle\; dx\\& \quad
        =\int_{\mathbb{R}}\left(a_4^{\pm}(u)\dot{\eta}^2 +2 a_2^{\pm}(u)\dot{\eta} \mathcal{G}_{\pm}(\eta) \left(a_2^{\pm}(u)\dot{\eta}\right)\right)\; dx,
    \end{split}
\end{equation}
where 
\begin{equation}
    a_4^{\pm}(u):=-2a_1{\pm}(u)'a_2{\pm}(u),
\end{equation}
and $a_1^{\pm},a_2^{\pm}$ are given by \eqref{equation of a_1 and a_2}.
\end{lemma}

\begin{lemma}[Second derivative formula of $\mathcal{A}$]For all $u=(\eta,\tilde{\xi})\in \mathcal{O}\cap \mathbb{V}$ and $\dot{\eta} \in \mathbb{V}_1,$ it holds that
\begin{equation}
    \begin{split}
        &\int_{\mathbb{R}} \tilde{\xi} \langle \T{D}^2\mathcal{A} (\eta)[\dot{\eta},\dot{\eta}], \tilde{\xi}\rangle\; dx\\& \quad
        =\int_{\mathbb{R}}\left(a_4(u)\dot{\eta} +2 \sum_{\pm} \rho_{\pm} a_2^{\pm}(\eta,\theta_{\pm})\mathcal{G}_{\pm}(\eta)(a_2^{\pm}(\eta,\theta_{\pm})\dot{\eta})-2\mathcal{M}(u)\dot{\eta}+2\mathcal{N}(u)\dot{\eta} \right)\dot{\eta}\;dx,
    \end{split}
\end{equation}
where 
\begin{equation}
    \theta_{\pm}(u):=\mathcal{G}_{\pm}(\eta)^{-1}\mathcal{A}(\eta)\tilde{\xi}, \qquad a_4(u):=\sum_{\pm} \rho_{\pm} a_4^{\pm}(\eta,\theta_{\pm})
\end{equation}
\begin{equation}
\begin{split}
    &\mathcal{L}_{\pm}(u)\dot{\eta}:=-\mathcal{G}_{\pm}(\eta)^{-1}\left( a_1^{\pm} (\eta,\theta_{\pm})\dot{\eta} \right)' + a_2^{\pm}(\eta,\theta_{\pm})\dot{\eta}, \quad \mathcal{L}(u):=\sum_{\pm}\rho_{\pm}\mathcal{L}_{\pm}(u)\\
    &\mathcal{M}(u)\dot{\eta}:=\sum_{\pm} \rho_{\pm} \left(a_1^{\pm}(\eta,\theta_{\pm})(\mathcal{L}_{\pm}(u)\dot{\eta})'+a_2^{\pm}(\eta,\theta_{\pm})\mathcal{G}_{\pm}(\eta) \mathcal{L}_{\pm}(u)\dot{\eta}\right)\\
    &\mathcal{N}(u)\dot{\eta}:=\sum_{\pm}\rho_{\pm}\left(a_1^{\pm}(\eta,\theta_{\pm})\left(\mathcal{A}(\eta)\mathcal{G}_{\pm}(\eta)^{-1}\mathcal{L}(u)\dot{\eta}\right)'+a_2^{\pm}(\eta,\theta_{\pm})\mathcal{A}(\eta)\mathcal{L}(u)\dot{\eta}\right).
\end{split}
\end{equation}
\end{lemma}

\section*{Acknowledgments}

The work of DS was supported in part by the NSF through the award NSF DMS-1812436.

\begin{center}
\bibliographystyle{alpha}
\bibliography{Bibjournal.bib}
\end{center}

\end{document}